\pdfoutput=1

\documentclass[11pt]{article}
\usepackage[letterpaper, left=1in, right=1in, top=1in, bottom=1in]{geometry}

\usepackage[dvipsnames]{xcolor}
\usepackage[hidelinks]{hyperref}
\hypersetup{
	colorlinks=true,
	linkcolor=blue!70!black,
	citecolor=blue!70!black,
	urlcolor=blue!70!black}
\usepackage{microtype}

\usepackage{xspace}

\usepackage{algorithm}

\ifdefined\usefunnybib
\usepackage[style=alphabetic,maxalphanames=4,minalphanames=4,
natbib,backend=biber,maxnames=2,maxbibnames=99]{biblatex}

\DeclareFieldFormat{titlecase}{\MakeSentenceCase*{#1}}
\else
\usepackage[numbers]{natbib}
\fi

\usepackage{amsthm}
\usepackage{mathtools}
\usepackage{amsmath}
\usepackage{bbm}
\usepackage{amsfonts}
\usepackage{amssymb}
\usepackage{xpatch}
\usepackage{enumitem}

\theoremstyle{definition}  %

\newtheorem{lemma}{Lemma}
\newtheorem{observation}{Observation}

\newtheorem{proposition}{Proposition}

\theoremstyle{plain}

\newtheorem{theorem}{Theorem}
\newtheorem{definition}{Definition}

\xpatchcmd{\proof}{\itshape}{\normalfont\proofnameformat}{}{}
\newcommand{\proofnameformat}{\bfseries}

\usepackage{prettyref}
\newcommand{\pref}[1]{\prettyref{#1}}
\newcommand{\pfref}[1]{Proof of \prettyref{#1}}
\newcommand{\savehyperref}[2]{\texorpdfstring{\hyperref[#1]{#2}}{#2}}
\newrefformat{eq}{\savehyperref{#1}{\textup{(\ref*{#1})}}}
\newrefformat{eqn}{\savehyperref{#1}{Equation~\ref*{#1}}}
\newrefformat{lem}{\savehyperref{#1}{Lemma~\ref*{#1}}}
\newrefformat{obs}{\savehyperref{#1}{Observation~\ref*{#1}}}
\newrefformat{def}{\savehyperref{#1}{Definition~\ref*{#1}}}
\newrefformat{line}{\savehyperref{#1}{line~\ref*{#1}}}
\newrefformat{thm}{\savehyperref{#1}{Theorem~\ref*{#1}}}
\newrefformat{corr}{\savehyperref{#1}{Corollary~\ref*{#1}}}
\newrefformat{sec}{\savehyperref{#1}{Section~\ref*{#1}}}
\newrefformat{app}{\savehyperref{#1}{Appendix~\ref*{#1}}}
\newrefformat{ass}{\savehyperref{#1}{Assumption~\ref*{#1}}}
\newrefformat{ex}{\savehyperref{#1}{Example~\ref*{#1}}}
\newrefformat{fig}{\savehyperref{#1}{Figure~\ref*{#1}}}
\newrefformat{alg}{\savehyperref{#1}{Algorithm~\ref*{#1}}}
\newrefformat{rem}{\savehyperref{#1}{Remark~\ref*{#1}}}
\newrefformat{conj}{\savehyperref{#1}{Conjecture~\ref*{#1}}}
\newrefformat{prop}{\savehyperref{#1}{Proposition~\ref*{#1}}}
\newrefformat{proto}{\savehyperref{#1}{Protocol~\ref*{#1}}}
\newrefformat{prob}{\savehyperref{#1}{Problem~\ref*{#1}}}
\newrefformat{claim}{\savehyperref{#1}{Claim~\ref*{#1}}}

\usepackage{algorithm}
\usepackage{verbatim}
\usepackage[noend]{algpseudocode}

\usepackage[nointegrals]{wasysym}
\usepackage{color-edits}

\definecolor{bluish}{rgb}{.05,0.3,0.8}

\addauthor{df}{red}
\addauthor{yc}{bluish}
\addauthor{ya}{purple}
\addauthor{bw}{brown}

\newcommand{\R}{\mathbb{R}}
\newcommand{\E}{\mathbb{E}}
\newcommand{\N}{\mathbb{N}}

\newcommand{\grad}{\nabla}

\newcommand{\defeq}{\coloneqq}
\newcommand{\eqdef}{\eqqcolon}

\newcommand{\ldef}{\defeq}
\newcommand{\rdef}{\eqdef}

\newcommand{\half}{\frac{1}{2}}
\newcommand{\quarter}{\frac{1}{4}}

\DeclarePairedDelimiter{\abs}{\lvert}{\rvert} %
\DeclarePairedDelimiter{\brk}{[}{]}
\DeclarePairedDelimiter{\crl}{\{}{\}}
\DeclarePairedDelimiter{\prn}{(}{)}
\DeclarePairedDelimiter{\nrm}{\|}{\|}
\DeclarePairedDelimiter{\norm}{\|}{\|}
\DeclarePairedDelimiter{\tri}{\langle}{\rangle}

\DeclarePairedDelimiter{\ceil}{\lceil}{\rceil}
\DeclarePairedDelimiter{\floor}{\lfloor}{\rfloor}

\let\Pr\undefined
\let\P\undefined
\DeclareMathOperator{\En}{\mathbb{E}}

\DeclareMathOperator{\P}{\mathbb{P}}

\DeclareMathOperator{\projop}{P}
\DeclareMathOperator{\projopP}{\Pi}

\DeclareMathOperator*{\argmin}{arg\,min} %
\newcommand{\mc}[1]{\mathcal{#1}}
\newcommand{\wt}[1]{\widetilde{#1}}
\newcommand{\wh}[1]{\widehat{#1}}

\def\ddefloop#1{\ifx\ddefloop#1\else\ddef{#1}\expandafter\ddefloop\fi}
\def\ddef#1{\expandafter\def\csname 
bb#1\endcsname{\ensuremath{\mathbb{#1}}}}
\ddefloop ABCDEFGHIJKLMNOPQRSTUVWXYZ\ddefloop
\def\ddefloop#1{\ifx\ddefloop#1\else\ddef{#1}\expandafter\ddefloop\fi}
\def\ddef#1{\expandafter\def\csname 
b#1\endcsname{\ensuremath{\mathbf{#1}}}}
\ddefloop ABCDEFGHIJKLMNOPQRSTUVWXYZ\ddefloop
\def\ddef#1{\expandafter\def\csname 
c#1\endcsname{\ensuremath{\mathcal{#1}}}}
\ddefloop ABCDEFGHIJKLMNOPQRSTUVWXYZ\ddefloop
\def\ddef#1{\expandafter\def\csname 
h#1\endcsname{\ensuremath{\widehat{#1}}}}
\ddefloop ABCDEFGHIJKLMNOPQRSTUVWXYZ\ddefloop
\def\ddef#1{\expandafter\def\csname 
hc#1\endcsname{\ensuremath{\widehat{\mathcal{#1}}}}}
\ddefloop ABCDEFGHIJKLMNOPQRSTUVWXYZ\ddefloop
\def\ddef#1{\expandafter\def\csname 
t#1\endcsname{\ensuremath{\widetilde{#1}}}}
\ddefloop ABCDEFGHIJKLMNOPQRSTUVWXYZ\ddefloop
\def\ddef#1{\expandafter\def\csname 
tc#1\endcsname{\ensuremath{\widetilde{\mathcal{#1}}}}}
\ddefloop ABCDEFGHIJKLMNOPQRSTUVWXYZ\ddefloop

\newcommand{\ls}{\ell}
\newcommand{\indicator}[1]{\mathbbm{1}{\crl{#1}}}    %
\newcommand{\pmo}{\crl*{\pm{}1}}
\newcommand{\eps}{\epsilon}
\newcommand{\veps}{\eps}

\newcommand{\vsigma}{\varsigma}

\usepackage{nicefrac}

\newsavebox\CBox

\newcommand{\DeltaF}{\Delta}

\newcommand{\LipGrad}{L}
\newcommand{\LipGradBar}{\bar{L}}

\newcommand{\lip}[1]{\ell_{#1}}
\newcommand{\lipBar}[1]{\bar{\ell}_{#1}}

\newcommand{\Fclass}{\mathcal{F}(\DeltaF, \LipGrad)}

\newcommand{\OclassPop}{\cO(K,\sigma^{2})}
\newcommand{\OclassMSS}{\cO(K,\sigma^{2},\LipGradBar)}

\newcommand{\Pr}{P_r}

\newcommand{\zset}{\mc{Z}}
\newcommand{\rset}{\mc{R}}

\newcommand{\op}{\mathrm{op}}
\newcommand{\opnorm}[1]{\norm{#1}_{\rm op}}

\newcommand{\openright}[2]{\left[{#1}, {#2}\right)}

\newcommand{\alg}{\mathsf{A}}

\newcommand{\AlgZR}{\mathcal{A}_{\textnormal{\textsf{zr}}}}

\newcommand{\AlgRand}{\mathcal{A}_{\textnormal{\textsf{rand}}}}

\newcommand{\AlgClassRand}{\AlgRand(K)}
\newcommand{\AlgClassZR}{\mathcal{A}_{\textnormal{\textsf{zr}}}(K)}

\newcommand{\minimaxZRMSS}{\bar{\mathfrak{m}}_{\epsilon}^{\mathsf{zr}}(K,\Delta,\bar{L},\sigma^{2})}
\newcommand{\minimaxRand}{\mathfrak{m}_{\epsilon}^{\mathsf{rand}}(K,\Delta,L,\sigma^{2})}
\newcommand{\minimaxRandMSS}{\bar{\mathfrak{m}}_{\epsilon}^{\mathsf{rand}}(K,\Delta,\bar{L},\sigma^{2})}

\newcommand{\threshfunc}{\Gamma}
\newcommand{\softindfunc}{\Theta}
\newcommand{\noisingfunc}{\nu}

\newcommand{\Ortho}{\mathsf{Ortho}}

\newcommand{\unscaled}[1]{#1}

\newcommand{\Funscaled}{\unscaled{F}_T}
\newcommand{\Frand}{\tilde{F}_{T,U}}
\newcommand{\Frandcomp}{\wh{F}_{T,U}}
\newcommand{\Fscaled}{F^{\star}_T}
\newcommand{\Ffinal}{F^{\star}_{T,U}}

\newcommand{\gunscaled}{\bar{g}_T}
\newcommand{\gscaled}{g^{\star}_T}
\newcommand{\grand}{\tilde{g}_{T,U}}
\newcommand{\grandcomp}{\wh{g}_{T,U}}
\newcommand{\gfinal}{g^{\star}_{T,U}}
\newcommand{\funscaled}{\unscaled{f}_T}

\newcommand{\oracle}{\mathsf{O}}
\newcommand{\oracleF}{\mathsf{O}_F}

\newcommand{\gunscaledBasic}{g_T}
\newcommand{\minimaxZR}{\mathfrak{m}_{\epsilon}^{\mathsf{zr}}(K,\Delta,L,\sigma^{2})}

\newcommand{\Otilde}{\wt{O}}

\newcommand{\trn}{\top}

\newcommand{\ind}[1]{^{(#1)}}

\newcommand{\sgn}{\mathrm{sgn}}

\DeclareMathOperator*{\support}{\mathrm{support}}

\newcommand{\prog}[1][\frac{1}{4}]{\mathrm{prog}_{#1}}

\newcommand{\spn}{\mathrm{span}}

\definecolor{innerboxcolor}{rgb}{.9,.95,1}
\definecolor{outerlinecolor}{rgb}{.6,0,.2}

\newcommand{\pop}{bounded variance\xspace}
\newcommand{\Pop}{Bounded variance\xspace}
\newcommand{\mss}{mean-squared smooth\xspace}

\title{{\huge Lower Bounds for Non-Convex Stochastic Optimization}}

\author
{
	Yossi Arjevani\\
	The Hebrew University\\
	{\small\texttt{yossi.arjevani@gmail.com}}\\
	\and
	Yair Carmon\\
	Tel Aviv University\\
	{\small\texttt{ycarmon@tauex.tau.ac.il}}\\
	\and
	John C.\ Duchi\\
	Stanford University\\
	{\small\texttt{jduchi@stanford.edu}}\\
	\and
	Dylan J.\ Foster\\
	Microsoft Research\\
	{\small\texttt{dylanfoster@microsoft.com}}\\
	\and
	Nathan Srebro\\
	TTIC\\
	{\small\texttt{nati@ttic.edu}}\\
	\and
	Blake Woodworth\\
	Inria\\
	{\small\texttt{blake.woodworth@inria.fr}}\\
}

\date{}

\begin{document}
\maketitle

\begin{abstract}
We lower bound the complexity of finding $\epsilon$-stationary points 
(with gradient norm at most $\epsilon$) using stochastic first-order 
methods. In a well-studied model where algorithms access smooth, 
potentially non-convex functions through queries to an unbiased 
stochastic gradient oracle with bounded variance, we prove that (in the 
worst case) any algorithm requires at least $\epsilon^{-4}$ queries to find 
an $\epsilon$-stationary point. The lower bound is tight, and establishes 
that stochastic gradient descent is minimax optimal in this model. In a 
more restrictive model where the noisy gradient estimates satisfy a 
mean-squared smoothness property, we prove a lower bound of 
$\epsilon^{-3}$ queries, establishing the optimality of recently proposed 
variance reduction techniques.
\end{abstract}

\section{Introduction}

Stochastic gradient methods---especially variants of stochastic 
gradient descent (SGD)---are the workhorse of modern machine learning 
and data-driven optimization
\citep{bottou2008tradeoffs,bottou2018optimization} more broadly. 
Much of the success of these methods stems from their broad applicability: 
any 
problem that admits an unbiased gradient estimator is fair
game. Consequently, there is 
considerable interest in understanding the fundamental performance limits 
of methods using stochastic gradients across broad problem classes. For  
\emph{convex} problems, a long line of work
\citep{nemirovski1983problem,nesterov2004introductory,
  agarwal2012information, woodworth2016tight}
sheds lights on these 
limits, and they are by now well-understood. However, many 
problems of interest (e.g., neural network training) are not
convex. This has led to intense development of improved methods for
non-convex stochastic optimization, but little is known about the
optimality of these methods. In 
this paper, we establish new fundamental limits for stochastic first-order 
methods in the non-convex setting.
In general non-convex optimization, it is intractable to find 
approximate global minima \citep{nemirovski1983problem}
or 
even to test if a point is a local minimum or a high-order saddle point \citep{murty1987some}.
As an alternative measure of optimization convergence, we 
consider $\epsilon$-approximate stationarity. That is, given differentiable 
$F:\bbR^{d}\to\bbR$, our goal is to find a point
$x\in\bbR^{d}$ with
\begin{equation}
\nrm*{\grad{}F(x)}\leq{}\veps.\label{eq:stationary}
\end{equation}
The use of stationarity as a convergence criterion dates back to the early 
days of nonlinear optimization 
\citep[cf.][]{vavasis1993black,nocedal2006numerical}.
Recent years have seen rapid development of a body of work that
studies non-convex optimization through the lens of non-asymptotic
convergence rates to $\epsilon$-stationary points
\citep{nesterov2006cubic,ghadimi2013stochastic,carmon2017convex,  
lei2017non,fang2018spider,zhou2020stochastic,fang2019sharp}. 
Another growing body of work motivates this study by identifying 
sub-classes of non-convex problems for which all stationary (or 
second-order stationary) points are globally optimal
\citep{ge2015escaping, ge2016matrix, sun2018geometric, 
mc2019implicit}. 

We prove our lower bounds in an oracle 
model~\citep{nemirovski1983problem,traub1988information}, where 
algorithms access the function $F$ through a \emph{stochastic first-order 
oracle} consisting of a gradient estimator 
$g:\R^d\times\zset\to\R^d$ and distribution $P_z$ on $\zset$ satisfying
\begin{equation}
  \label{eq:g_oracle}
\En_{z}\brk*{g(x,z)}=\grad{}F(x),\quad\text{and}\quad 
\En_{z}\nrm*{g(x,z)-\grad{}F(x)}^{2}\leq{}\sigma^{2}.
\end{equation}
At the $t$th optimization step, the algorithm queries at a
point $x\ind{t}$, the oracle draws  $z\ind{t}\sim{}P_z$, and the
algorithm observes the noisy gradient estimate $g(x\ind{t},z\ind{t})$. 
We make the standard assumption that  
the objective $F$ has bounded initial 
subobtimality
and Lipschitz gradient:
\begin{equation}
  \label{eq:basic_assumptions}
F(x\ind{0})-\inf_{x}F(x)\leq{}\Delta~~~\text{and}~~~
\nrm*{\grad{}F(x)-\grad{}F(y)}\leq{}L\cdot\nrm*{x-y}~~
\forall{}x,y\in\R^d.
\end{equation}
Following common practice, we refer to functions $F$ with 
$L$-Lipschitz gradients as ``$L$-smooth.''

For problem instances $(F,g)$ satisfying~\pref{eq:g_oracle}
and~\pref{eq:basic_assumptions}, given a tolerance $\epsilon$, SGD finds a
point $x$ such that $\E \norm{\grad F(x)} \le \epsilon$ using $O\prn{\Delta
  L \epsilon^{-2}\prn{1 +\sigma^2 \epsilon^{-2}}}$ oracle
queries~\citep{ghadimi2013stochastic}. In the typical regime, and the one we
focus on here, $\epsilon \le \sigma$ so the complexity reduces to
$O\prn{\Delta L \sigma^2 \epsilon^{-4}}$.  
The literature on variance reduction for finding stationary 
points~\citep{lei2017non,fang2018spider,zhou2020stochastic} considers  
the following additional assumptions:
\begin{enumerate}[leftmargin=*]
\item The stochastic gradient $g$ satisfies a
  \emph{mean-squared smoothness} property
  \begin{equation}
    \label{eq:mss}
    \En_{z}\nrm*{g(x,z)-g(y,z)}^{2}\leq{}\LipGradBar^2\cdot\nrm*{x-y}^{2}
    \quad\forall{}x,y\in\bbR^{d}.
  \end{equation}
\item The algorithm is allowed $K$ \emph{simultaneous queries}: at step $t$,
  the algorithm queries $x\ind{t,1},\ldots,x\ind{t,K}$ and observes
  $g(x\ind{t,1},z\ind{t}),\ldots,g(x\ind{t,K},z\ind{t})$, where the random 
  seed $z\ind{t}\sim{}P_z$ is shared. 
\end{enumerate}
Under the mean-squared smoothness assumption and using $K=2$
simultaneous queries the SPIDER~\citep{fang2018spider} and 
SNVRG~\citep{zhou2020stochastic} %
algorithms find a point $x$ such that 
$\E \norm{\grad 
	F(x)} \le 
\epsilon$ using $O(\Delta{}\LipGradBar\sigma\veps^{-3} + \sigma^{2}\veps^{-2})$ oracle 
queries.
This improvement over the $\epsilon^{-4}$ rate of SGD raises natural  
questions. 
 Can we improve this rate further? 
Alternatively, 
can 
we improve the rate of SGD without the additional 
assumption~\pref{eq:mss}? We settle both questions in
the negative.

\subsection{Contributions}

We prove lower bounds for finding stationary points in the 
stochastic first-order oracle model. Our main
result is \pref{thm:main_randomized}, which states:
\begin{enumerate}[leftmargin=*]
\item There exists a distribution over instances $(F,g)$ satisfying 
assumptions~\pref{eq:g_oracle} and~\pref{eq:basic_assumptions} under 
which every randomized algorithm requires at least $c \cdot \Delta L 
\sigma^2\epsilon^{-4}$ oracle queries to find $x$ satisfying $\E 
\norm{\grad F(x)} \le \epsilon$, where $c>0$ is a universal constant and 
where the expectation is taken over the randomness in both the oracle and the algorithm.

\item When $g$ also satisfies the mean-squared smoothness 
property~\pref{eq:mss}, every randomized algorithm requires 
$c\cdot\prn*{\Delta{}\bar{L}\sigma\veps^{-3} + \sigma^{2}\veps^{-2}}$ 
oracle queries. 

\end{enumerate}
Both lower bounds hold for any number $K$ of simultaneous queries, with 
the dimension $d$ of the hard instance depending polynomially on $K$
and $\epsilon^{-1}$ (see expressions for $d$ in \pref{sec:approach} below).

Our lower bounds continue to 
hold when the oracle is subject to  more stringent assumptions. In 
particular, we 
show that gradient estimators of the form $g(x,z)=\grad_x f(x,z)$ give rise 
to 
the same lower bounds; these gradient estimators arise in statistical 
learning problems such as empirical risk minimization. Furthermore, our 
results extend to \emph{active} oracles where the algorithm may choose the 
seed $z$.
This setting includes the special case of finite 
sum minimization, where $F(x)=\frac{1}{n}\sum_{i=1}^n f_i(x)$, each oracle 
query consists of point $x$ and index $i$, and the oracle response is  
$\grad f_i(x)$.

The main implications of our results are 
as follows.
\begin{itemize}[leftmargin=\parindent]
\item \textbf{Optimality of SGD and recent variance-reduction schemes.}  Our
  $\epsilon^{-4}$ lower bound matches (up to a numerical constant) the rate
  of convergence of SGD~\citep{ghadimi2013stochastic} under
  assumptions~\pref{eq:g_oracle} and~\pref{eq:basic_assumptions}, thereby
  characterizing the optimal complexity and proving that SGD attains
  it. Similarly, under the additional assumption~\pref{eq:mss} our
  $\epsilon^{-3}$ lower bound matches the rates of~\citet{fang2018spider}
  and \citet{zhou2020stochastic}, thereby proving their optimality.
  
\item \textbf{Separation between smoothness assumptions.}  Our results
  highlight that the mean-squared smoothness assumption~\pref{eq:mss} is
  critical for variance reduction: we prove that in its absence, any scheme
  will require a number of queries that scales as $\epsilon^{-4}$ at least.
  These results are salient, as this assumption appears in numerous recent 
  works on non-convex 
  optimization~\citep{fang2018spider,zhou2020stochastic,
    zhou2019lower,fang2019sharp}.

\item \textbf{Separation between convex and non-convex stochastic 
  optimization.} \citet{foster2019complexity} show that for 
  \emph{convex} functions satisfying assumptions~\pref{eq:g_oracle} 
  and \pref{eq:basic_assumptions}, the optimal rate for finding 
  $\veps$-stationary points is 
  $\widetilde{\Theta}\prn{\sqrt{\Delta{}L\veps^{-2}} + 
    \sigma^{2}\veps^{-2}}$. 
  Our $\Omega(\Delta L \sigma^2 \epsilon^{-4})$ lower bound thus 
  implies a gap between the convex
  and non-convex setting that scales as $\veps^{-2}$. Conceptually,
  both rates admit a simple interpretation. The convex complexity is 
  the sum of the noiseless convex optimization complexity
  $\sqrt{\Delta{}L\veps^{-2}}$~\citep{carmon2019lower_ii} and the
  estimation complexity $\sigma^{2}\veps^{-2}$. In contrast, in the 
  non-convex case the noiseless complexity $\Delta 
  L\epsilon^{-2}$~\citep{carmon2019lower_i} and the estimation 
  complexity 
  $\sigma^2\epsilon^{-2}$ \emph{multiply} rather than add. This 
  observation underpins our proofs.
\end{itemize}

\subsection{Our approach}\label{sec:approach}
We build on the noiseless lower bound construction 
of~\citet{carmon2019lower_i}, itself inspired by Nesterov's notion of a
chain-like function~\cite{nesterov2004introductory}. The key 
technique is to construct a function such that any noiseless oracle query 
reveals the index of at most a single ``relevant'' coordinate; the lower 
bound follows from the fact that any 
$\epsilon$-stationary point is non-zero in $\Omega(L \Delta 
\epsilon^{-2})$ relevant coordinates. We amplify this lower bound by 
designing a noisy oracle that reveals a relevant coordinate only with low 
probability $p=\Theta(\epsilon^2/\sigma^2)$. This increases the number 
of required queries by a factor proportional to 
$1/p=\Theta(\sigma^2\epsilon^{-2})$, giving our $\epsilon^{-4}$ lower 
bound. The main challenge lies in making sure that the oracle is not too noisy,
in the sense that the variance requirement~\pref{eq:g_oracle} is met. To do so, we focus all of the noise 
 on the single new coordinate $i_x$ that the query $x$ would 
discover next
via the noiseless gradient. More specifically, we let 
$z\sim\mathrm{Bernoulli}(p)$, and set $g_{i_x}(x,0)=0$ and $g_{i_x}(x,1)$ to 
be such that $g$ is unbiased. By careful 
analysis of the noiseless construction of~\cite{carmon2019lower_i} we 
show 
that the variance bound holds and we obtain our lower bound. 

Proving the $\epsilon^{-3}$ lower bound requires additional nuance, as 
the ``incoming coordinate'' index $i_x$ is not continuous in $x$, and so 
the gradient estimator above does not satisfy the mean-square smoothness 
requirement~\pref{eq:mss}. Leveraging the special structure of the 
noiseless construction once more, we design a continuous 
surrogate for $i_x$, and arrive at a mean-square smooth 
construction for which $g_{i_x}(x,z)$ is again non-zero only with 
probability 
$p$. Scaling this construction such that $L = \Theta(\LipGradBar \epsilon / \sigma)$ yields the 
$\epsilon^{-3}$ lower bound.

For ease of exposition, we first carry out our proof strategy for the 
sub-class of ``zero-respecting'' algorithms, whose queries are non-zero 
only in coordinates where previous oracle responses were not zero. We then 
lift our results to
the class of all randomized algorithms using the method of random
rotations \citep{woodworth2017lower,carmon2019lower_i}. On a high level, 
we argue that in a random 
coordinate system, any algorithm operating on our constructions is 
essentially zero-respecting.

Our lower bound constructions are high-dimensional. For zero-respecting 
algorithms, the dimension we require is exactly the number of relevant 
coordinates: $d_{\mathsf{zr}}=\Theta(\Delta L \epsilon^{-2})$ for the 
bounded variance 
case and $d_{\mathsf{zr}}=\Theta(\Delta \LipGradBar \sigma^{-1} 
\epsilon^{-1})$ for the 
mean-square smooth case. To handle general, potentially randomized algorithms
that allow $K$ simultaneous oracle queries for every random realization
$z \sim P_z$,
we 
add many irrelevant coordinates, and our 
proof requires dimension $\Otilde(K d_{\mathsf{zr}}^2 / p)$, 
where $p=\Theta(\epsilon^2/\sigma^2)$ is the progress 
probability. Lower 
bound 
constructions with dimension that scales polynomially in $\epsilon^{-1}$ are 
common~\citep{nemirovski1983problem, nesterov2004introductory,
  woodworth2016tight, foster2019complexity}, and natural
for algorithms that (nominally) work in arbitrary Hilbert spaces.
In the noiseless setting, obtaining tight and algorithm-independent lower 
bounds on dimension-independent convergence rates necessitates 
high-dimensional constructions; see \citet[Section 
1.2]{carmon2019lower_i} for additional discussion. Since the noiseless 
setting is a special case of our noisy setting, it seems likely that here too 
high-dimensional constructions are to some extent unavoidable.

\subsection{Related work}
Lower bounds for first-order convex optimization in the noiseless setting 
are well-studied \citep{nemirovski1983problem, 
nesterov2004introductory}. For $L$-smooth functions in the 
high-dimensional regime, it is well-known that 
$\Theta\prn[\big]{\sqrt{D^{2}L\eps^{-1}}}$ gradient evaluations are 
necessary and sufficient to find an $\eps$-suboptimal point given 
$x\ind{0}$ with $\nrm*{x\ind{0}-x^{\star}}\leq{}D$;
Nesterov's accelerated gradient method~\cite{nesterov1983method}
achieves this rate.

For smooth high-dimensional non-convex
optimization in the noiseless setting, \citet{carmon2019lower_i}
establish that $\Theta(\Delta{}L\eps^{-2})$ gradient evaluations are
necessary and sufficient for finding $\eps$-stationary points; this rate is achieved by 
gradient descent. An earlier line of
work develops lower bounds for finding stationary points of non-convex
functions in the low-dimensional regime where $d$ is constant, but they obtain 
either weaker lower bounds \citep{vavasis1993black} 
or tight bounds that hold only for specific algorithm classes
\citep{cartis2010complexity, cartis2012complexity, cartis2012much, 
cartis2017worst}.

A long line of work on lower bounds for stochastic convex optimization
traces back to Nemirovski and Yudin's seminal information-based
complexity~\cite{nemirovski1983problem}.  Extensions since then have allowed
sharp dimension-dependent bounds via reductions to statistical estimation
problems \citep{raginsky2011information,agarwal2012information}, as well as
extension to structured problems common in machine learning, such as finite
sums, by restrictions on the form of the update
rules~\cite{arjevani2016dimension} and high-dimensional constructions
\citep{woodworth2016tight,foster2019complexity}. Our technique for proving
stochastic lower bounds differs qualitatively from these methods in that we
preserve the sequential hardness of the noiseless non-convex lower bound
construction of \cite{carmon2019lower_i}, and use the noise in the
stochastic setting to amplify the hardness of this construction.

For non-convex stochastic optimization, few lower bounds are known. 
\citet{drori2019complexity} recently showed that SGD itself 
cannot
obtain a rate better than $\eps^{-4}$ for finding $\eps$-stationary points, even for convex functions. This is an 
algorithm-specific result, whereas we show that \emph{no algorithm} can 
improve over this rate. For finite sum problems where 
$F(x)=\frac{1}{n}\sum_{i=1}^{n}f_i(x)$, \citet{fang2018spider}
show that $\Omega\prn*{\Delta{}\bar{L}\veps^{-2}\sqrt{n}}$ stochastic
gradient queries are required to find a $\veps$-stationary point; SPIDER and 
SNVRG \citep{fang2018spider,zhou2020stochastic} have matching upper 
bounds. This
lower bound is incomparable to ours: the stochastic gradient
construction in the paper~\cite{fang2018spider} has unbounded variance, 
so it cannot imply results along the lines of 
\pref{thm:main_randomized}. Indeed,
\citet{fang2018spider} leave obtaining the $\veps^{-3}$ lower bound 
we provide in \pref{thm:main_randomized} as an open problem.

We now turn to upper bounds for 
finding stationary points in the stochastic setting.  
In the \emph{convex} 
setting (where achieving approximate global optimality is possible and 
hence usually the goal) \citet{allen2018how} proposes algorithms with 
rates for finding stationary points improving over SGD, and  
\citet{foster2019complexity} give
improvements on these bounds and establish their optimality.
For the non-convex setting, \citet{ghadimi2013stochastic} 
establish an $O(\Delta{}L\sigma^{2}\veps^{-4})$ upper bound for SGD, 
and a 
large body of recent work attempts to improve this rate. These attempts 
roughly divide into two categories: \emph{variance reduction} and \emph{high-order 
information}. 

Works in the variance reduction category make either the mean-squared 
smoothness  assumption~\pref{eq:mss} or a stronger variant wherein every 
$g(\cdot, z)$ is $\LipGradBar$-Lipschitz. The earliest results consider 
only the finite sum setting, and establish improved dependence on the 
number of summands~\citep{allen2016variance, reddi2016stochastic}. 
Under the bounded variance assumption~\pref{eq:g_oracle},  
\citet{lei2017non} obtain a rate of $\eps^{-10/3}$, demonstrating that 
in the non-convex setting variance reduction provides benefits beyond 
finite sum optimization. Subsequent algorithms by
\citet{fang2018spider} and \citet{zhou2020stochastic} 
obtain an improved rate of $\eps^{-3}$, which we
prove is optimal. Recent 
work~\citep{wang2018spiderboost,cutkosky2019momentum} offers further 
refinements of these algorithms that also obtain the $\eps^{-3}$ rate.

Smoothness in higher derivatives, such as Lipschitz
continuity of the Hessian, allows additional
possibilities~\citep{xu2018first,allen2018natasha,allen2018neon2,
  fang2018spider}.
\citet{tripuraneni2018stochastic} provide a sub-sampled cubic 
regularization method that uses
stochastic Hessian-vector products and attains a rate of 
$\eps^{-3.5}$ without relying on mean-squared smoothness~\pref{eq:mss} or
simultaneous gradient queries.
\citet{fang2019sharp} show that it is possible to 
obtain the rate  $\eps^{-3.5}$ using SGD with perturbed gradients and 
restarts without the need for Hessian-vector products. Most works that 
assume Lipschitz Hessian also provide guarantees for finding 
second-order stationary points. 

\subsection{Organization}
\pref{sec:prelim} introduces the formal oracle model in which we prove our 
lower bounds. In \pref{sec:zero_respecting}, we prove our results for the 
subclass of zero-respecting algorithms. In 
\pref{sec:randomized} we apply random rotations to prove lower bounds for 
all randomized algorithms, leading to our main result. 
\pref{sec:extensions} describes the extensions of our results to statistical 
learning and active oracles, and 
\pref{sec:discussion} concludes with discussion of some remaining open 
problems.

\paragraph{Notation} For a vector $x\in\bbR^{d}$, we let $\support(x) 
\defeq \crl{i \mid x_i \ne 0}$
and $x_{\geq{}i} \defeq \prn*{x_i,\ldots,x_d}\in\bbR^{d-i+1}$. 
For $\alpha\in[0,1)$ we define the ``progress'' of $x$ as  $\prog[\alpha](x) 
\defeq \max\crl*{i\ge 0 \mid \abs{x_i} > 
	\alpha}$, where we assume $x_0 \equiv 1$.
For a
differentiable function $f$, we adopt the convention
$\brk{\grad{}f(x)}_i=\grad_{i}f(x)=\frac{\partial}{\partial{}x_i}f(x)$. When
$f$ is twice-differentiable, we likewise define
$\brk{\grad^{2}f(x)}_{ij}=\grad^{2}_{ij} f(x) 
=\frac{\partial^{2}}{\partial{}x_i\partial{}x_j}f(x)$.
 Throughout, $\nrm*{x}$ denotes the Euclidean norm of $x$ and  
 $\nrm*{x}_{\infty}$  denotes its
 $\ls_{\infty}$ norm. For a matrix 
$A\in\bbR^{d_1\times{}d_2}$,
$\nrm*{A}_{\op}$  denotes the operator norm.
Given functions $f,g:\cX\to \openright{0}{\infty}$ where $\cX$ is any set,
we use non-asymptotic big-$O$ notation: $f=O(g)$ if there exists a numerical
constant $c < \infty$ such that $f(x)\leq{}c\cdot{}g(x)$ for all $x\in\cX$
and $f=\Omega(g)$ if there is a numerical constant $c>0$ such that
$f(x)\geq{}c\cdot{}g(x)$. We write $f = \Otilde(g)$ as shorthand for $f=O(g
\max\{1,\log g\})$.

\section{Setup}
\label{sec:prelim}
We study the stochastic optimization problem of finding an 
$\epsilon$-stationary point through the well-known framework of
 oracle complexity
\citep{nemirovski1983problem}, which we set up formally in this section.

\paragraph{Function class}
We develop lower bounds for algorithms
that find stationary points of functions in the set
\begin{equation*}
  \cF(\Delta,L) \defeq
  \left\{F:\bbR^{d}\to\bbR
  ~~\mbox{s.t.}~~
  F(0)-\inf_{x}F(x)\leq{}\Delta,
  ~
  \nrm*{\grad{}F(x)-\grad{}F(y)}\leq{}L\nrm*{x-y}
  ~\mbox{for all}~ x,y\right\}.
\end{equation*}
We state explicitly the value of the dimension $d$ required for each lower 
bound construction; the reader may
otherwise regard $d$ as a free parameter.

\paragraph{Optimization protocol}
We consider algorithms that access an unknown function
$F\in\cF(\Delta,L)$ through a \emph{stochastic first-order oracle}
$\oracle$. Each oracle $\oracle$
consists of a distribution $P_z$ on a 
measurable space $\zset$ and an
unbiased mapping $\oracle_F(x,z)=(F(x),g(x,z))$, meaning
for each $F\in\Fclass$
and $x$, if $z \sim P_z$ then $\En\brk*{g(x,z)}=\grad{}F(x)$.
We consider a protocol in which algorithms
interact with the oracle through multiple rounds of batch queries. At each round $i$, the algorithm queries a batch
\begin{equation}
  \label{eq:k_query}
  x\ind{i}\ldef\prn*{x\ind{i,1},\ldots,x\ind{i,K}},~ \text{where~~} 
  x\ind{i,k}\in\bbR^{d}~\mbox{and}~k\in\brk*{K}
\end{equation}
of size $K$, and for each batch query
$x\ind{i}$, the oracle $\oracle$
performs an independent draw $z\ind{i}\sim{}P_z$ and responds with 
\[
\oracleF(x\ind{i},z\ind{i}) \ldef{} \prn*{\oracleF(x\ind{i,1},z\ind{i}),\ldots,\oracleF(x\ind{i,K},z\ind{i})}. 
\]
When $K=1$ this is the classical first-order stochastic optimization
framework. By considering larger batches we can subsume variance-reduction
methods
such as SPIDER and SNVRG
\citep{fang2018spider,zhou2020stochastic}, both of which query each
stochastic gradient at $K=2$ points.\footnote{See also the
  $K$-parallel model of \citet{nemirovski1994parallel}.}
Note
that we allow the algorithm to observe the function value $F(x)$
exactly for each query, which is a weaker assumption than typical in lower and upper bounds for stochastic optimization.

\paragraph{Optimization algorithms} An algorithm $\alg$ consists of a 
distribution 
$\Pr$ over a measurable set $\rset$ and a sequence of 
measurable mappings $\crl{\alg\ind{i}}_{i\in\N}$ such that 
$\alg\ind{i}$ takes in the first  $i-1$ oracle 
responses and the random seed $r\in\rset$ to produce the $i$th 
query. 
We let $\crl{x\ind{i}_{\alg\brk{\oracleF}}}_{i\in\N}$ denote the (random) 
sequence of queries resulting from applying algorithm 
$\alg$ with $\oracle$, defined recursively as
\begin{equation}
\label{eq:alg_iteration}
x\ind{i}_{\alg\brk{\oracleF}}= \alg\ind{i}\prn*{r, 
	\oracleF\big(x\ind{1}_{\alg\brk{\oracleF}},z\ind{1}\big), \ldots,
	\oracleF\big(x\ind{i-1}_{\alg\brk{\oracleF}},z\ind{i-1}\big)},
\end{equation}
where $r\sim \Pr$ is drawn a single time at the beginning of the
protocol (this is no loss of generality~\cite{nemirovski1983problem}).
We define $\AlgClassRand$ to be the class of all algorithms
that follow the protocol \pref{eq:alg_iteration} with $K$ batch
queries per round.

\paragraph{Oracle classes}
We consider two natural classes of oracles. For the \emph{bounded variance} class, denoted $\OclassPop$, we require that the stochastic gradient be unbiased and have the bounded
variance property \pref{eq:g_oracle},
but otherwise allow arbitrary $g(x,z)$. This well-studied setting subsumes the standard analysis of stochastic gradient
descent for finding approximate stationary points \citep{ghadimi2013stochastic}. 

The \pop setting places few restrictions on the stochastic gradient function $g(x,z)$, but there are many
applications in which the stochastic gradients may have
additional structure. In the \emph{mean-squared smooth} setting, we
require that in addition to the bounded-variance property \pref{eq:g_oracle}, the stochastic
gradient satisfies the mean-squared smoothness property \pref{eq:mss}.
We use $\OclassMSS$ to
denote the class of all such oracles. By
Jensen's inequality, any function that admits 
an $\LipGradBar$-mean-squared smooth oracle must itself be
$\LipGradBar$-smooth.

Our results also extend to more structured oracles appearing in the
statistical learning and/or finite-sum
settings. We defer the details to \pref{sec:extensions}.

\paragraph{Complexity measures} Our main results are tight
lower bounds on the \emph{distributional 
  complexity}~\cite{yao1977probabilistic,nemirovski1983problem,
  braun2017lower} of finding  
$\eps$-stationary points. Let $\cP\brk{\Fclass}$ be set of all 
distributions over $\Fclass$; the distributional complexity in the \pop 
setting is
\begin{equation}
\minimaxRand
\defeq 
\sup_{\oracle\in\OclassPop}
\sup_{P_F \in \cP\brk{\cF(\Delta,\LipGrad)}}
\inf_{\alg\in\AlgClassRand}
\inf\crl[\bigg]{
	T \in \N
	~
	\bigg\vert
	~
	\E\,\norm[\big]{ \grad F\prn[\big]{
			x\ind{T,1}_{\alg[\oracle_F]} }
	}  \le \epsilon
},\label{eq:minimax_rand}
\end{equation} 
where the expectation is over the sampling of $F$ from $P_F$, the 
randomness in the oracle $\oracle$, and the randomness in the algorithm 
$\alg$, though randomization in $\alg$ does not affect distributional
complexity~\cite{yao1977probabilistic, nemirovski1983problem}.
The distributional complexity for the \mss setting is
\begin{equation}
\minimaxRandMSS
\defeq 
\sup_{\oracle\in\OclassMSS}
\sup_{P_F \in \cP\brk{\cF(\Delta,\LipGradBar)}}
\inf_{\alg\in\AlgClassRand}
\inf\crl[\bigg]{
	T \in \N
	~
	\bigg\vert
	~
	\E\,\norm[\big]{ \grad F\prn[\big]{
			x\ind{T,1}_{\alg[\oracle_F]} }
	}  \le \epsilon
}.\label{eq:minimax_rand_mss}
\end{equation} 
Lower bounds on distributional complexity imply lower bounds on  
minimax complexity~\cite[cf.][]{nemirovski1983problem,braun2017lower}. 
That is, $\minimaxRand 
> 
T$ implies that there 
exists $\oracle\in\OclassPop$  such that for every $\alg\in\AlgClassRand$ 
there exists a function $F\in\Fclass$ for which $\E\,\norm[\big]{ \grad 
F\prn[\big]{ x\ind{T,1}_{\alg[\oracle_F]}}} > \epsilon$, where here the 
expectation is over randomness in $\alg$ and $\oracle$.

\section{Lower bounds for zero-respecting algorithms}
\label{sec:zero_respecting}

Before presenting our results in full generality, we first develop the
key components of our technique by proving lower bounds for a
restricted class of \emph{zero-respecting
  algorithms}~\cite{carmon2019lower_i}. The class of zero-respecting 
  algorithms generalizes the well-known linear
span-assumption \citep[see][Section 2.1.2]{nesterov2004introductory},
and encompasses many standard optimization algorithms. More
importantly, the lower bound instances we introduce in this section
form the core of our lower bounds for general algorithms via a reduction in the next section.

An algorithm $\alg$ is zero-respecting if its queries at
each round have support in the supports of all previous oracle responses:
\begin{definition}
A stochastic first-order algorithm $\alg$ is \emph{zero-respecting} if
for any oracle $\oracle $ and any realization of $z\ind{1},z\ind{2},\ldots$, for all $t\ge1$ and $k\in\brk*{K}$,
\begin{equation}
\label{eq:zero_respecting}
\support\big(x\ind{t,k}_{\alg\brk{\oracleF}}\big) \subseteq
\bigcup_{i < t,k'\in\brk*{K}} 
\support\big(g\ind{i,k'}\big),
\end{equation}
where $\prn[\big]{f\ind{t,1}, g\ind{t,1}},\ldots,\prn[\big]{f\ind{t,K},
  g\ind{t,K}} =\oracleF\big(x\ind{t}_{\alg\brk{\oracleF}},z\ind{t}\big)$
denote the
oracle responses for round $t$.
\end{definition}
We let $\AlgClassZR$ denote the class of all zero-respecting
algorithms. %
Our main result for this section is to establish tight lower bounds on the minimax
oracle complexity for zero-respecting algorithms, which we denote
by $\minimaxZR$ for the \pop setting and $\minimaxZRMSS$ for the \mss
setting; these complexities are as in \pref{eq:minimax_rand}
and \pref{eq:minimax_rand_mss}, with $\AlgClassZR$ replacing 
$\AlgClassRand$. The zero-respecting structure allows us to attain tight 
lower bounds using $P_F$ supported on a single hard function.

\subsection{Probabilistic zero-chains}
At the core of our development is an 
 embedding of the task of finding a stationary point into that of finding 
 a point $x$ with high coordinate \emph{progress}, which we define as
 \begin{equation}
\label{eq:progress}
 \prog[\alpha](x) \defeq \max\crl*{i\ge 0 \mid \abs{x_i} > 
 \alpha}~~\mbox{(where $x_0\equiv 1$)},
 \end{equation}
i.e., $\prog[\alpha](x)$ is the highest index whose entry is $\alpha$-far from zero, for some 
threshold $\alpha\in[0,1)$.
The starting point for our lower bounds is the notion of
a \emph{first-order zero-chain} \citep{carmon2019lower_i}, which is a 
function 
$F$ that 
satisfies $\prog[0](\grad F(x)) \le \prog[0](x)+1$ for all $x$, 
generalizing Nesterov's concept of a ``chain-like'' 
function~\citep{nesterov2004introductory}. In the noiseless case 
($g\ind{i}=\grad F(x\ind{i})$), zero-chains control the rate of progress of 
zero-respecting algorithms: every query can ``discover'' at most one  
coordinate, and therefore $\prog[0](x\ind{i}) < T$ for all $i\le T$.

Our key insight is that in the stochastic setting, noise can amplify progress control: we 
construct stochastic gradient functions for which any zero-respecting
algorithm requires \emph{many} queries in order to activate one coordinate. 
We call such functions \emph{probabilistic zero-chains}.

\begin{definition}
\label{def:zero-chain}
A stochastic gradient function $g(x,z)$ is a \emph{probability-$p$ 
zero-chain} if 
\begin{align}
    \P \bigl(\,\exists{}x: \prog[0]\prn*{g(x,z)} =\prog(x)+1 \bigr) \leq{}
    p,\label{eq:prob-zc}\intertext{and}
    \P \bigl(\,\exists{}x:  \prog[0]\prn*{g(x,z)} >\prog(x)+1 \bigr) =0.
    \label{eq:standard-zc}
\end{align}
\end{definition}
The constant ${1}/{4}$ in \pref{eq:prob-zc} is only used in our
lower bound for general algorithms, and any non-zero constant would
suffice in its place. Even $\prog[0](x)$ is sufficient for
the constructions in this section; we keep $\prog(x)$ in the definition
only for notational consistency. We also note that the 
requirement~\pref{eq:standard-zc} implies that any $F$ for which $g$ is an 
unbiased gradient estimator must itself be a zero-chain.

The next lemma formalizes the idea that any zero-respecting algorithm
interacting with a probabilistic zero-chain requires many rounds to
discover all coordinates.

\begin{lemma}
\label{lem:prob-zero-chain}
Let $g(x,z)$ be a probability-$p$ zero-chain gradient estimator for 
$F:\bbR^{T}\to\bbR$, and let $\oracle$ be any oracle with
$\oracle_{F}(x,z)=(F(x),g(x,z))$. Let $\crl[\big]{x\ind{t,k}_{\alg[\oracle_F]}}$ 
be 
the 
queries of any $\alg\in\AlgZR(K)$ interacting with $\oracle_F$. Then, with 
probability at least $1-\delta$,
	\begin{equation*}
          \max_{k\in\brk*{K}}\prog[0]\prn*{
		x\ind{t,k}_{\alg[\oracle_F]}
	} < T, \quad\text{for all } t \leq{} \frac{T-\log(1/\delta)}{2p}.
	\end{equation*}
\end{lemma}
The intuition behind \pref{lem:prob-zero-chain} is that any zero-respecting algorithm must activate coordinates in sequence, and must wait at least $\Omega(1/p)$ rounds between activations on average, leading to a
total waiting time of $\Omega(T/p)$ rounds. The proof below makes this 
intuition formal; note that throughout the proof we use that 
$\prog[\alpha]$ is non-increasing
in $\alpha$. 
\begin{proof}
  For brevity, we omit the subscript $\alg[\oracle_F]$ from 
  $\crl[\big]{x\ind{t,k}_{\alg[\oracle_F]}}$. Recall that $r$ is the algorithm's 
  random seed (Eq.~\pref{eq:alg_iteration}).
  Let $\prn[\big]{f\ind{i,1}, g\ind{i,1}},\ldots,\prn[\big]{f\ind{i,K},
  g\ind{i,K}}$ denote the oracle responses for the $i$th batch query
  $x\ind{i}=(x\ind{i,1},\ldots,x\ind{i,K})$, and let 
  $g\ind{i}=(g\ind{i,1},\ldots,g\ind{i,K})$. 
  Define a filtration 
  \[
  \cG\ind{i} \defeq \sigma(r,x\ind{1},\ldots,x\ind{i},g\ind{1},\ldots,g\ind{i}).
  \] 
  We define two measures of the algorithm's progress:
	\begin{equation*}
	\pi\ind{t} = \max_{i\le t} \max_{k\in\brk*{K}}\prog[0](x\ind{i,k}) = \max
	 \crl*{j\le T\mid 
	x_j\ind{i,k} \ne 0 \text{ for some }i\le t, k\in\brk*{K}},
	\end{equation*}
	and similarly,
	\begin{equation*}
	\gamma\ind{t} = \max_{i\le t} \max_{k\in\brk*{K}}\prog[0](g\ind{i,k}) = \max \crl*{j\le T\mid 
		g_j\ind{i,k} \ne 0 \text{ for some }i\le t, k\in\brk*{K}},
	\end{equation*}
	so $\pi\ind{t} \in \cG\ind{t-1}$ and $\gamma \ind{t} \in  
	\cG\ind{t}$. Note that $\prog[0](x)$ is the largest index in 
	$\support(x)$, so the zero-respecting property implies that
	\begin{equation*}
	\pi\ind{t} \le \gamma \ind{t-1}
	\end{equation*}
	for all $t$, with probability 1, where we let
        $\gamma\ind{0}\equiv 0$. 
	Therefore, it suffices to show that
	\begin{equation}\label{eq:pzc-lem-need-to-show}
	\P(\gamma\ind{t} \ge T) \le \delta~~\mbox{for all}~~t\le 
	\frac{T-\log\frac{1}{\delta}}{2p}.
	\end{equation}

        To show this, first observe that with probability $1$,
	\begin{equation*}
          \max_{k\in\brk*{K}}\prog[0](g\ind{t,k}) \stackrel{(\star)}{\le} 1 + \max_{k\in\brk*{K}}\prog(x\ind{t,k}) \le 1 + 
          \max_{k\in\brk*{K}}\prog[0](x\ind{t,k}) 
	\le 1 + \pi\ind{t} \le 1+\gamma\ind{t-1},
	\end{equation*}
	where inequality $(\star)$ holds by the zero chain 
	property~\pref{eq:standard-zc}, and the other inequalities hold by 
	definition. Since $x\ind{t}\in\cG\ind{t-1}$, we have that 
	$g\ind{t,k}=g(x\ind{t,k},z\ind{t})$ is independent of $x\ind{t}$ given 
	$\cG\ind{t-1}$. Consequently, the 
	zero-chain property~\pref{eq:prob-zc} implies that 
        \[
       \bbP\prn*{ \max_{k\in\brk*{K}}\prog[0](g\ind{t,k}) \le 
	\max_{k\in\brk*{K}}\prog(x\ind{t,k})\le 
	\gamma\ind{t-1}\,\Big\vert\,\cG\ind{t-1}}\geq{}1-p.
        \]
	Since $\gamma\ind{t} = \max\crl{\gamma\ind{t-1}, \max_{k\in\brk*{K}}\prog[0](g\ind{t,k})}$, we 
	conclude that
	\begin{equation}
	\label{eq:iota}%
	\P(\gamma\ind{t} - \gamma\ind{t-1} \notin \{0,1\}\mid \cG\ind{t-1})=0
	~~\mbox{and}~~
	\P(\gamma\ind{t} - \gamma\ind{t-1} = 1 \mid \cG\ind{t-1})\le p.
	\end{equation}
	Therefore, denoting the increment $\iota\ind{t} \defeq \gamma\ind{t} - 
	\gamma\ind{t-1}$, we have via the Chernoff method,
	\begin{equation*}
	\P(\gamma\ind{t} \ge T) = 
	\P(e^{\sum_{i=1}^t \iota\ind{i}} \ge e^T)
	\le e^{-T}\cdot\E {e^{\sum_{i=1}^t \iota\ind{i}} }.
	\end{equation*}
	Using $\iota\ind{t}\in\cG\ind{t}$ and 
	$\P(\iota\ind{t}=1\mid\cG\ind{t-1})=1-\P(\iota\ind{t}=0\mid\cG\ind{t-1})\le
	 p$, we obtain
	\begin{equation*}
	\E {e^{\sum_{i=1}^t \iota\ind{i}} }
	= \E \brk*{\prod_{i=1}^t \E\brk*{ e^{\iota\ind{i}} \;\Big\vert\; \cG\ind{i-1}}}
	\le (1-p+ p\cdot e)^t \le e^{pt\cdot(e-1)} \le e^{2pt}.
	\end{equation*}
	It follows that $\P(\gamma\ind{t} \ge T) \le e^{2pt-T}\le \delta$ for every 
	$t\le \frac{1}{2p}(T+\log \delta)$, giving~\eqref{eq:pzc-lem-need-to-show}.
\end{proof}

\subsection{Lower bound for the \pop setting}
\pref{lem:prob-zero-chain} suggests a natural lower bound strategy:
\begin{enumerate}[leftmargin=\parindent]
\item[i.] Construct a function $F\in\Fclass$
  whose gradients are large for all
$x\in\bbR^T$ with $\prog[0]\prn*{x\ind{i}} < T$.
\item[ii.] Construct $g$, a probability-$p$ zero chain gradient estimator for 
$F$.
\end{enumerate}
Together with \pref{lem:prob-zero-chain}, these steps guarantee  
that any zero-respecting algorithm interacting with $g$ will take 
at least $\Omega(T/p)$ rounds to
make the gradient of $F$ small.
We first execute our strategy for the \pop setting \pref{eq:g_oracle}.

We choose the underlying function $F$ to be the 
construction of \citet{carmon2019lower_i}. For each $T\in\bbN$, we define
\begin{equation}
\Funscaled(x) \defeq -\Psi(1)\Phi(x_1) +
\sum_{i=2}^{T}\brk*{\Psi(-x_{i-1})\Phi(-x_i) -
	\Psi(x_{i-1})\Phi(x_i)},\label{eq:f_unscaled}
\end{equation}
where the component functions $\Psi$ and $\Phi$ are
\begin{equation}
\Psi(t) = \left\{
\begin{array}{ll}
0,\quad&t\leq{}1/2,\\
\exp\prn*{1-\frac{1}{(2t-1)^{2}}},\quad&t>1/2.
\end{array}
\right.\quad\quad\text{and}\quad\quad\Phi(t) =
\sqrt{e}\int_{-\infty}^{t}e^{-\frac{1}{2}\tau^{2}}d\tau.\label{eq:psi_phi}
\end{equation}
The function $\Funscaled$ is a (deterministic)
zero-chain, and has large gradient unless all coordinates are large
$(\prog[1](x)\geq{}T)$.  
We enumerate all the relevant properties of $\Funscaled$ in the following.
\begin{lemma}[\cite{carmon2019lower_i}]
	\label{lem:deterministic-construction} The function
        $\Funscaled$ satisfies:
	\begin{enumerate}
		\item \label{item:val} $\Funscaled(0) -
                  \inf_{x}\Funscaled(x) \leq \Delta_0\cdot T$, where $\Delta_0 = 
                  12$.
		\item \label{item:lip} The gradient of $\Funscaled$ is
		$\lip{1}$-Lipschitz continuous, where $\lip{1}=152$.
		\item \label{item:grad} For all $x\in\bbR^{T}$,
		$\nrm*{\grad{}\Funscaled(x)}_{\infty}\leq{}\gamma_\infty$, where 
		$\gamma_\infty = 23$.
		\item \label{item:zero-chain} For all $x\in\R^T$,
                  $\prog[0]\prn*{\grad\Funscaled(x)}\leq{}\prog[\frac{1}{2}](x)+1$.
		\item \label{item:large-grad} For all $x\in\R^T$, if $\prog[1](x)<T$ 
		then $\nrm*{\grad\Funscaled(x)} \ge 
		|\grad_{\prog[1](x)+1}\Funscaled(x)| > 1$.
	\end{enumerate}
\end{lemma}
\noindent Parts~\pref{item:val}--\pref{item:grad} of the 
lemma follow 
from~\cite[Lemma 3]{carmon2019lower_i} and its proof; we derive the 
precise value $\ell_1=152$ in \pref{app:basics}. Part \pref{item:zero-chain} 
follows from~\cite[Observation 3]{carmon2019lower_i} and part 
\pref{item:large-grad} is~\cite[Lemma 2]{carmon2019lower_i}.

We now turn to the construction of a probabilistic zero-chain for
$F_T$. The main technical difficulty in the construction lies in keeping
the variance of the stochastic gradient function bounded and, in
particular, independent of the dimension $T$. Indeed, consider a naive
construction that when queried at point $x$, returns $0$ with
probability $1-p$ and returns $\frac{1}{p}\cdot\grad\Funscaled(x)$
with probability $p$. While this is clearly a probability-$p$
zero-chain, the variance at point $x$ is
$\Omega\prn[\big]{{\nrm*{\grad\Funscaled(x)}^{2}_2}/{p}}$, which can 
be
as large as $T/p$. As we  let the dimension $T$ depend polynomially on 
$1/\eps$, removing this dimension dependence from the variance is 
critical for making the
oracle belong to $\OclassPop$ after rescaling.

Our key observation is that, since
$\nrm*{\grad\Funscaled(x)}_{\infty}\leq{}23$ by
\pref{lem:deterministic-construction}.\pref{item:grad}, we can keep
the variance bounded if, instead of deleting all coordinates
uniformly, we delete only a single important coordinate. Since our
goal is to construct a probabilistic zero-chain, and since $F_T$ is
itself a deterministic zero-chain, a natural choice of coordinate is
$\prog(x)+1$. This leads to the following stochastic gradient function:
\begin{align} \label{eq:basic-construction}
\brk*{\gunscaledBasic(x,z)}_i \defeq \grad_i \Funscaled(x) \cdot 
\prn*{1+\indicator{i > \prog(x)}\prn*{\frac{z}{p}-1}},
\end{align}
where $z\sim \mathrm{Bernoulli}(p)$. Note that for
all $i > \prog(x) + 1$,
$\grad{}_i\Funscaled(x)=0$, so only the specific
coordinate $\prog(x)+1$ is noisy.
\begin{lemma}\label{lem:pzc-basic}
  The stochastic gradient estimator $\gunscaledBasic$ is a
  probability-$p$ zero-chain, is unbiased for $\grad \Funscaled$, and has
  variance
	\begin{equation*}
	\E \norm{\gunscaledBasic(x,z) - \grad \Funscaled(x)}^2 \le \varsigma^2 
	\cdot \frac{1-p}{p}~~
	\text{for all }x\in\R^T,~\mbox{where }\varsigma=23.
	\end{equation*}
\end{lemma}
\begin{proof}%
	\label{lem_proof:pzc-basic}
	First, we observe that $\E\brk*{\gunscaledBasic(x,z)} =  \grad \Funscaled(x)$ 
	for all $x\in\R^T$ by the definition~\pref{eq:basic-construction} and 
	the fact $\E\brk{\frac{z}{p}} = 1$, so any
        $\oracle$ with
        $\oracle_{\Funscaled}(x,z)=(\Funscaled(x),\gunscaledBasic(x,z))$ is
        indeed a stochastic first-order oracle.
	
	Second, we argue that the probability-$p$ zero-chain property holds. 
	Recall that $\prog[\alpha](x)$ is non-increasing in $\alpha$, so 
	$\prog(x) \ge \prog[\half](x)$. Therefore, by 
	\pref{lem:deterministic-construction}.\ref{item:zero-chain}, 
	$\brk{\gunscaledBasic(x,z)}_i = \grad_i \Funscaled(x) =0$ for all 
	$i>\prog(x)+1$, all $x\in\R^T$ and all $z\in\{0,1\}$. Moreover, since 
	$\P(z\ne0) = p$, we have 
	$\P\prn{\brk{\gunscaledBasic(x,z)}_{\prog(x)+1} \ne 0} \le p$,
        establishing that the construction \pref{eq:basic-construction} satisfies 
	\pref{def:zero-chain}. 
	
	Finally, we bound the variance. Note that the error term $\gunscaledBasic(x,z) - \grad \Funscaled(x)$ is non-zero only in the 
	coordinate $i_x\ldef{}\prog(x)+1$. Therefore,
	\begin{equation*}
	\E \norm{\gunscaledBasic(x,z) - \grad \Funscaled(x)}^2
         = \abs*{\grad_{i_x} \Funscaled(x)}^2 \, \E \prn*{\frac{z}{p}-1}^2
	\le \frac{\nrm{\grad \Funscaled(x)}_\infty^2 (1-p)}{p} \le \frac{23^2(1-p)}{p}, 
	\end{equation*}
	where the last inequality follows from \pref{lem:deterministic-construction}.\ref{item:grad}.
\end{proof}
With the construction in hand, we prove our first lower bound.
\begin{theorem}
\label{thm:zr_population}
There exist numerical constants
$c,c'>0$ such that for all $L, \Delta, \sigma^2>0$ 
and 
$\epsilon \le  c' \sqrt{L\Delta}$,
	\label{thm:main_zero_respecting}
	\begin{equation*}
	\minimaxZR
	\ge 
	c\cdot\frac{\DeltaF \LipGrad \sigma^2}{\epsilon^4}.
	\end{equation*}
Constructions of dimension $d=O\prn*{\frac{\Delta{}L}{\eps^{2}}}$ realize 
the lower bound.
\end{theorem}
Before giving the proof, let us make  a few remarks.
\begin{itemize}[leftmargin=\parindent]
\item The bound is tight, in that it matches (up to a numerical constant) the 
  convergence rate for SGD (which is zero-respecting) \citep[][Eq.~(2.13)]{ghadimi2013stochastic}.
  Note that
  the restriction that $\epsilon \le  c' \sqrt{L\Delta}$ is without loss of 
  generality, since for $\epsilon >  c' \sqrt{L\Delta}$ we have
  $\nrm*{\grad{}F(0)}= O(\eps)$ for all functions $F\in\Fclass$, 
  so an $\epsilon$-stationary point is trivial to find.
\item The optimal complexity
  $\Theta\prn{\frac{\Delta{}L\sigma^{2}}{\veps^{4}}}$ is the product of the
  first-order oracle complexity for the deterministic setting, which is
  $\Theta\prn{\frac{\Delta{}L}{\veps^{2}}}$ \citep{carmon2019lower_i}, and
  the sample complexity of estimating a single gradient to precision $\eps$,
  which is $\Theta\prn{\frac{\sigma^{2}}{\eps^{2}}}$. This is the first
  setting we are aware of where the \emph{product} of these respective
  complexities characterizes the stochastic first-order complexity. Contrast
  to the convex setting, where the complexity scales with the
  sum~\citep{foster2019complexity}.

\item The lower bound does not depend on $K$, meaning that additional
  batch queries cannot by themselves improve on the rate obtained by SGD. 
  While at
  first glance this may seem like a strange consequence of the
  zero-respecting assumption, we will show that the same holds true
  for arbitrary algorithms, provided the dimension is sufficiently large.
\end{itemize}
\begin{proof}[\pfref{thm:zr_population}]
	Let $\Delta_0, \ell_1$ and $\varsigma$ be the numerical constants in 
	\pref{lem:deterministic-construction}.\ref{item:val}, 
	\pref{lem:deterministic-construction}.\ref{item:lip} and 
	\pref{lem:pzc-basic}, respectively. 
  Given accuracy parameter  
$\epsilon$, initial suboptimality $\Delta$, smoothness parameter $\LipGrad$ and 
variance parameter $\sigma^2$, we define
\begin{align*}
\Fscaled(x)=\frac{\LipGrad\lambda^{2}}{\lip{1}}
\Funscaled\left(\frac{x}{\lambda}\right),\quad\text{ 
where\quad
}\lambda=\frac{\lip{1} }{L}\cdot 2\epsilon,\quad\text{ and }\quad 
T=\floor*{
	\frac{\Delta}{\Delta_0 (\LipGrad \lambda^2 /\ell_1)}} =
\floor*{
\frac{\LipGrad\Delta}{\Delta_0\lip{1} 
	(2\epsilon)^{2}}},
\end{align*}
where we assume  $T\ge 3$, or equivalently  
$\epsilon \le  \sqrt{\frac{L\Delta}{12\Delta_0\lip{1} 
}}$. 
Let 
\[\gscaled{}(x,z)=\frac{L\lambda}{\lip{1}}\cdot \gunscaledBasic(x/\lambda, z)
=2\eps \cdot \gunscaledBasic(x/\lambda, z)
\]
denote the corresponding scaled stochastic
gradient function.  Now, by \pref{lem:deterministic-construction}.\ref{item:val}
and \pref{lem:deterministic-construction}.\ref{item:lip}, we have that
$\Fscaled$ is $\frac{L}{\ls_1}\cdot \ls_1=L$-smooth  and
has initial suboptimality bounded by $\Delta$. Likewise, by \pref{lem:pzc-basic},
\begin{align*}
\E \norm{\gscaled(x,z) - \grad \Fscaled(x)}^2 = \left(\frac{L\lambda}{\lip{1}}\right)^2
\E \left\|{\gscaled\left(\frac{x}{\lambda},z\right) - \grad 
	\Fscaled\left(\frac{x}{\lambda}\right)}\right\|^2 \le \frac{(2\varsigma   
	\epsilon)^2(1-p)}{p}.
\end{align*}
Therefore, setting 
$p=\min\left\{(2\varsigma\epsilon)^2/\sigma^2,1\right\}$ guarantees 
a variance 
bound of 
$\sigma^{2}$.

Next, Let $\oracle$ be any oracle in $\OclassPop$ for which 
$\oracle_{\Fscaled}(x,z)=(\Fscaled(x),\gscaled(x,z))$. Instantiating  
\pref{lem:prob-zero-chain} for $\delta=1/2$, we 
have that with probability at least $1/2$, 
$\max_{k\in\brk*{K}}\prog[0]\prn*{x\ind{t,k}_{\alg\brk{\oracleF}}} < T$ for 
all $t\leq{} 
{(T-1)}/{2p}$ and $k\in[K]$.
Now, by 
\pref{lem:deterministic-construction}.\ref{item:large-grad}, for every $x\in\R^T$ 
 such that  
$\prog[0](x)<T$, it holds that
\begin{equation*}
\nrm*{\grad \Fscaled(x)} = \frac{L\lambda}{\lip{1}} \nrm*{\grad 
	\Funscaled\left(\frac{x}{\lambda}\right)} >\frac{L\lambda}{\ls_1} =
2\epsilon,
\end{equation*}
So with probability at least $1/2$, we have for all $t\leq{} 
{(T-1)}/{2p}$ and $k\in[K]$ that $\nrm[\big]{\grad 
\Fscaled(x\ind{t,k}_{\alg\brk{\oracleF}})}>2\epsilon$. Therefore,
\begin{equation}
\E\norm[\big]{ \grad \Fscaled\prn[\big]{
		x\ind{t,k}_{\alg\brk{\oracleF}}
} }
>\eps,
\end{equation}
by which it follows that 
\begin{align*}
\minimaxZR
>
\frac{T-1}{2p}
\ge 
\left(\left\lfloor \frac{L\Delta}{4\Delta_0\lip{1} 
	\epsilon^{2}}\right\rfloor -1\right)
\frac{\sigma^2}{2(2\varsigma\epsilon)^2}
\ge 
\frac{1}{2^6\ell_1\Delta_0\varsigma^2} \cdot 
\frac{L\Delta\sigma^2}{\epsilon^4},
\end{align*}
where the last inequality uses that $\floor{x}-1\geq{}x/2$ whenever 
$x\geq{}3$.
\end{proof}

\subsection{Lower bound for the \mss setting}
\newcommand{\smoothOrc}{\overline{\oracle}} 
We now turn to lower bounds for the \mss setting. Here, we must ensure
that in addition to the variance constraint, our stochastic gradient
function satisfies the mean-squared smoothness constraint
\pref{eq:mss}. This requires a more
sophisticated construction than before, as the use of the
indicator function $\indicator{i >\prog(x)}$ makes the stochastic
gradient $\gunscaledBasic$ discontinuous. Indeed, let
$x=(1,1/4-\delta,0)$ and $y=(1,1/4,0)$. Then $\prog(x)=1<2 =
\prog(y)$, and for any $\delta\in(0,1/2)$ we have
\begin{align*}
\En_{z}\nrm*{\gunscaledBasic(x,z)-\gunscaledBasic(y,z)}^{2}
&\geq{}
  \En_{z}\abs*{\brk*{\gunscaledBasic(x,z)}_2-\brk*{\gunscaledBasic(y,z)}_2}^{2}\\
&= (1-p)\abs{\Psi(1)\Phi'(1/4-\delta)}^{2} + p\abs*{\tfrac{1}{p}\Psi(1)\Phi'(1/4-\delta)-\Psi(1)\Phi'(1/4)}^{2},
\end{align*}
which does not approach zero as $\delta\to{}0$. 

To overcome this issue, we replace the indicator
$\indicator{i >\prog(x)}$ with a smooth surrogate. Let
$\Gamma:\bbR\to\bbR$ be any smooth non-decreasing Lipschitz function with
$\Gamma(t)=0$ for all $t\leq{}1/4$ and
$\Gamma(t)=1$ for all $t\geq{}1/2$. For each $i$, we
define the following smoothed version of $\indicator{i > \prog(x)}$:
\begin{equation}
\softindfunc_{i}(x) \defeq 
\threshfunc\prn*{1-\prn*{\sum_{k=i}^T\threshfunc^2(|x_k|)}^{1/2}}
= \threshfunc\prn*{1-\norm*{\threshfunc\prn*{|x_{\ge 
          i}|}}},
\end{equation}
where $\Gamma(\abs*{x_{\geq{}i}})$ is a shorthand for a vector with entries 
$\Gamma(|x_i|), \Gamma(|x_{i+1}|), \ldots, \Gamma(|x_T|)$. Observe
that $\Theta_i$ indeed acts as a smoothed indicator: We have $\softindfunc_{i}(x) 
= 1$ for all $i > \prog[\frac{1}{4}](x)$ and $\softindfunc_{i}(x)=0$ for all 
$i\le\prog[\frac{1}{2}](x)$, and therefore
\begin{equation*}
  \indicator{i > \prog[\frac{1}{4}](x)} \le \softindfunc_{i}(x)  \le \indicator{i 
    > 
    \prog[\frac{1}{2}](x)}.
\end{equation*}

We define a new stochastic gradient function $\gunscaled$ by replacing the indicator function in $\gunscaledBasic$ with the smoothed indicator
$\Theta_i$:
\begin{equation}\label{eq:pairwise-construction}
\brk*{\gunscaled(x,z)}_i \defeq 
\grad_i \Funscaled(x)
\cdot 
\noisingfunc_i(x,z),
~~\mbox{where}~~
\noisingfunc_i(x,z) \defeq 1 + \softindfunc_{i}(x) \prn*{\frac{z}{p}-1},
\end{equation}
and $z\sim  \mathrm{Bernoulli}(p)$. To fully specify the construction,
we  take
\begin{equation}
\label{eq:theshfunc-def}
\threshfunc(t) = \frac{\int_{1/4}^{t} \Lambda(\tau) d\tau}{
	\int_{1/4}^{1/2} \Lambda(\tau') d\tau'},\quad
~~\mbox{where}~~\quad
\Lambda(t) = \begin{cases}
0, & t \le \frac{1}{4}~\mbox{or}~t \ge \frac{1}{2}, \\
\exp\prn*{ - 
	\frac{1}{100\prn*{t-\tfrac{1}{4}}\prn*{\tfrac{1}{2}-t}}},
& \frac{1}{4}  < t < \frac{1}{2}. \\
\end{cases}
\end{equation}
This is simply an integrated bump
function construction; see \pref{fig:gamma}. 

\begin{figure}
	\centering
	\includegraphics[width=0.9\textwidth]{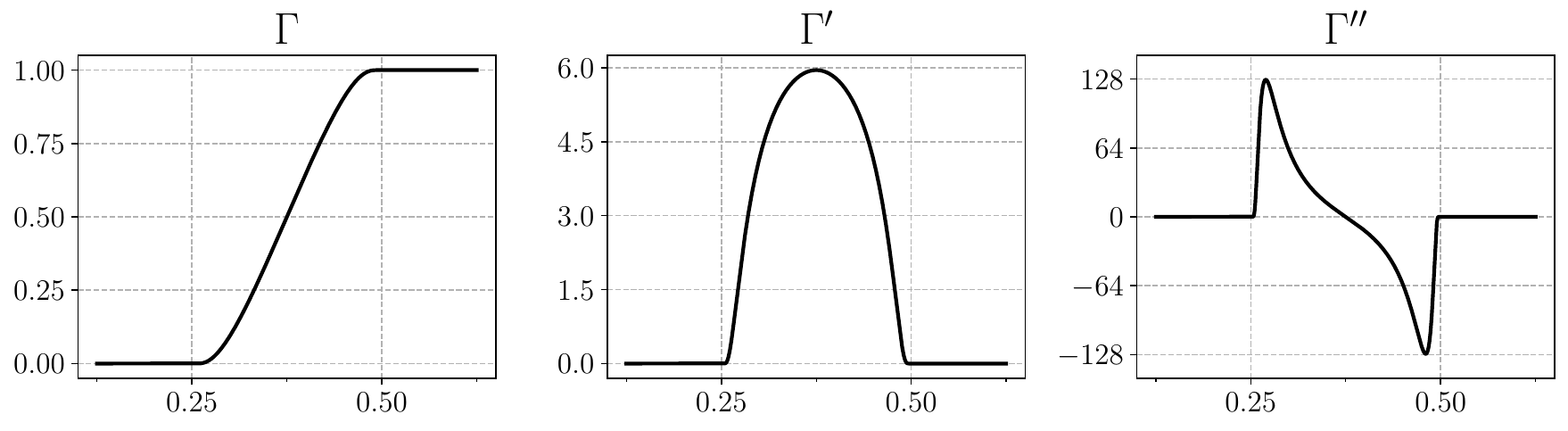}
	\caption{The construction $\Gamma$ in Eq.~\pref{eq:theshfunc-def} and 
	its derivatives; 
	\pref{obs:thresh} is evident.}\label{fig:gamma}
\end{figure}

\begin{observation}\label{obs:thresh}
	The function $\threshfunc$ satisfies 
	\begin{enumerate}
		\item \label{item:thresh-0} $\threshfunc(t)=0$ for all $t\in(-\infty,1/4]$.
		\item \label{item:thresh-1} $\threshfunc(t)=1$ for all $t\in[1/2,\infty)$.
		\item \label{item:thresh-lip}
                  $\threshfunc\in\mc{C}^\infty$, with $0
                  \le \threshfunc'(t) \le 6$ and $|\threshfunc''(t)|
                  \le 128$ for all $t\in\bbR$.
	\end{enumerate}
\end{observation}
With these properties established, we prove the following \mss
analogue of \pref{lem:pzc-basic}.
\begin{lemma}\label{lem:pzc-pair}
 The stochastic gradient estimator $\gunscaled$ is a
probability-$p$ zero-chain, is unbiased for $\grad \Funscaled$, and 
satisfies
	\begin{equation}\label{eq:pzc-var-mss}
	\E \,\norm{\gunscaled(x,z) - \grad \Funscaled(x)}^2 \le 
	\frac{\varsigma^2(1-p)}{p}
		~~\text{and}~~
                \E\, \norm{\gunscaled(x,z) - \gunscaled(y,z) }^2 \le 
	\frac{\bar{\ls}_1^2}{p} \norm{x-y}^2,
	\end{equation}
	for all $x,y\in\R^T$, where $\varsigma=23$ and $\lipBar{1}=328$.
      \end{lemma}
We defer the proof of \pref{lem:pzc-pair} to \pref{app:pzc-pair}. The 
proofs for the probability-$p$ zero-chain property and variance
  bound are similar to \pref{lem:pzc-basic}. For the mean-squared
  smooth property, we show that for any $x$, the vector
  $\delta(x,z) = \gunscaled(x,z)-\grad{}\Funscaled(x)$ has at most one
  non-zero coordinate, given by
  $\prog[\frac{1}{2}](x)+1$. If we denote
  $i_x=\prog[\frac{1}{2}](x)+1$ and
  $i_y=\prog[\frac{1}{2}](y)+1$, then we can bound $\E
  \norm{\gunscaled(x,z) - \gunscaled(y,z) }^2$ by first appealing to
  smoothness of $\Funscaled$, and then using the Lipschitz property of
  $\Theta_i$ to bound
  $\En\abs*{\delta_{i_x}(x,z)-\delta_{i_x}(y,z)}^{2}$ and
  $\En\abs*{\delta_{i_y}(x,z)-\delta_{i_y}(y,z)}^{2}$.

  Our lower bound for the \mss setting now follows from another simple
scaling argument. 
\begin{theorem}
  \label{thm:mean_smooth_zero_respecting}
There exist numerical constants $c,c'>0$ such that for all $\LipGradBar, \Delta, 
\sigma^2>0$ 
and 
$\epsilon \le  c' \sqrt{\LipGradBar\Delta}$,
\begin{equation*}
\minimaxZRMSS
\ge
c\cdot\prn*{\frac{\DeltaF \LipGradBar 
\sigma}{\epsilon^3}+\frac{\sigma^{2}}{\eps^{2}}}.
\end{equation*}
Constructions of dimension $d=O(1 + 
\frac{\Delta\LipGradBar}{\sigma \epsilon})$ realize 
	the lower bound.
\end{theorem}
\pref{thm:mean_smooth_zero_respecting} is tight, since the upper bounds for
SPIDER~\citep{fang2018spider} and SNVRG \citep{zhou2020stochastic} match it up to constants. %
As with \pref{thm:zr_population}, the restriction
$\eps\leq{}O(\sqrt{\LipGradBar\Delta})$ is essentially without loss of
generality. \pref{thm:mean_smooth_zero_respecting} leaves open the
possibility that there exists an algorithm that achieves
$O(\veps^{-3})$ in the \mss setting using $K=1$; see
\pref{sec:discussion} for further discussion.

We defer the proof of \pref{thm:mean_smooth_zero_respecting} to
\pref{app:mean_smooth_zero_respecting}, as it is very similar to that of 
\pref{thm:zr_population}. In 
particular, it uses the same scaling argument and replaces $\LipGrad$ with 
roughly $\LipGradBar \epsilon / \sigma$. This results in the final instance 
scaled as 
$\Fscaled(x)\propto{\eps\sigma\LipGradBar^{-1}}\Funscaled(\frac{\LipGradBar
 x}{\sigma})$. The 
new scaling introduces an additional restriction that
$\eps\leq{}O(\frac{\Delta\bar{L}}{\sigma})$. When this does not hold,
one has $\frac{\Delta\bar{L}\sigma}{\eps^{3}}\geq{}c\cdot\frac{\sigma^{2}}{\eps^{2}}$,
and the claimed lower bound follows from a standard estimation lower bound
(see \pref{lem:stat_lower_bound} in \pref{app:basics}).

\section{Lower bounds for randomized algorithms}
\label{sec:randomized}

\newcommand{\Frandgeneric}{\tilde{F}_U}
\newcommand{\grandgeneric}{\tilde{g}_U}
\newcommand{\frandgeneric}{\tilde{f}_U}
\newcommand{\Orandgeneric}{\tilde{\mathsf{O}}_U}
\newcommand{\Osrandgeneric}{\tilde{\mathsf{O}}_U^{\text{Stat}}}

We now extend our lower bound construction for
zero-respecting algorithms into a lower bound for arbitrary,
potentially randomized algorithms.
Our main theorem provides optimal lower bounds on the minimax complexities \pref{eq:minimax_rand} and \pref{eq:minimax_rand_mss} for the
\pop and \mss settings. 
\begin{theorem}
  \label{thm:main_randomized}
There exist numerical constants $c,c'>0$ such that for all $L, \Delta, 
\sigma^2>0$ 
and 
$\epsilon \le  c' \sqrt{L\Delta}$,
	\begin{equation}\label{eq:main-result-pop}
          \minimaxRand
	\ge
	c\cdot\frac{\DeltaF \LipGrad \sigma^2}{\epsilon^4},
	\end{equation}
	and for all $\LipGradBar >0$ and $\epsilon \le  c' 
	\sqrt{\LipGradBar\Delta}$, we have
	\begin{equation}\label{eq:main-result-mss}
          \minimaxRandMSS
	\ge
	c\cdot\prn*{\frac{\DeltaF \LipGradBar 
	\sigma}{\epsilon^3}+\frac{\sigma^{2}}{\veps^{2}}}.
	\end{equation}
Constructions of dimension 
$\Otilde\prn*{ K \Delta^2 \LipGrad^2 \sigma^2 \epsilon^{-6} }$ realize the 
lower bound~\pref{eq:main-result-pop} and constructions of 
dimension $d=\Otilde\prn*{ K \Delta^2 \LipGradBar^2 \epsilon^{-4} }$  
realize the lower bound~\pref{eq:main-result-mss}. 
\end{theorem}
In the remainder of the section we outline the proof of 
\pref{thm:main_randomized}; we defer all formal proofs to 
\pref{app:randomized}. Our approach is to lift the instance developed in the 
previous
section to a hard distribution over functions such that for any randomized 
algorithm a 
a function drawn from this distribution is 
hard high probability. This
approach closely follows \cite{carmon2019lower_i,woodworth2017lower}, 
though the analysis differs in a few technical points.

Given a function $F(x)$ and a gradient 
estimator $g(x,z)$, we define the rotated instance 
\begin{equation*}
\Frandgeneric(x)\defeq{}F(U^\top x),\quad\text{and}\quad
\grandgeneric(x,z)\defeq{}Ug(U^\top x, z),
\end{equation*}
where 
$U\in\Ortho(d,T)\ldef\crl*{U\in\bbR^{d\times{}T}\mid{}U^{\trn}U=I_{T}}$ is 
a matrix with orthogonal unit norm columns. For any such $U$ we define an 
oracle for the rotated function according to
\begin{equation}\label{eq:rotated-oracle}
\oracle_{\Frandgeneric}(x,z)=(\Frandgeneric(x),\grandgeneric(x,z)).
\end{equation}

When $U$ is drawn uniformly from $\Ortho(d,T)$, any algorithm interacting 
with $\oracle_{\Frandgeneric}$ produces queries $\{x\ind{t,k}\}$ 
such that the sequence $\{U^\top x\ind{t,k}\}$ behaves essentially like the 
queries of a zero-respecting algorithm interacting with of $\oracle_F$. 
More precisely, for sufficiently large $d$ we can guarantee that every entry 
of $U^\top x\ind{t,k}$ that is significantly far from zero (say, with 
absolute value $>1/4$) is in the support of a previous oracle 
response $g(U^\top x\ind{t',k'},z\ind{t'})$ for some $t'<t$ and $k'\in[K]$. 
This follows because oracle responses provide essentially no 
information on coordinates outside that support, and therefore, 
these coordinates of $U^\top x\ind{t,k}$ behave roughly as coordinates of 
a spherically uniform vector in dimension 
$d-tK$, and we can obtain a high probability bound on their magnitude that 
scales as 
$\norm{x\ind{t,k}}\sqrt{tK}/\sqrt{d-tK}$; the precise argument requires 
careful 
handling of the  information leaked at each step. By assuming that the 
queries are bounded 
and choosing sufficiently large $d$, we guarantee that coordinates outside 
the support are smaller than $1/4$ and therefore that the zero-respecting 
structure obtains. Combining this structure with \pref{def:zero-chain} of 
probabilistic zero-chains implies control over $\prog(U^\top x\ind{t,k})$, 
as we state formally in the following generalization of 
\pref{lem:prob-zero-chain}, whose proof we provide in
\pref{app:proof-randomized-bounded}.
\begin{lemma}
  \label{lem:randomized_bounded}
  Let $F:\bbR^{T}\to\bbR$ and let $g:\R^d\times\zset\to\R^d$ be 
  probability-$p$ zero chain. Let $R>0$, $\delta\in(0,1)$, and 
  $\alg\in\AlgRand(K)$ be any algorithm that produces queries with norm 
  bounded by $R$. Additionally let $d\geq{}\ceil{18 \frac{R^2 KT}{p}\log 
    \frac{2KT^2}{p\delta}}$, $U$ be uniform on $\Ortho(d,T)$, and 
  $\oracle_{\Frandgeneric}$ be as in \pref{eq:rotated-oracle}.
  Then with probability at least $1-\delta$,\footnote{The 
    event holds with
    probability at least $1-\delta$ with respect to the random choice of $U$ 
    and the 
    oracle seeds $\{z\ind{t}\}$, even when conditioned over any randomness 
    in $\alg$.}
  \begin{equation}\label{eq:randomized_bounded}
    \max_{k\in\brk*{K}}\prog(U^{\trn} 
    x\ind{t,k}_{\alg[\oracle_{\Frandgeneric}]}) < T~~\mbox{for all}~~t\le 
    \frac{T-\log\frac{2}{\delta}}{2p}.
  \end{equation}
\end{lemma}

Applying \pref{lem:randomized_bounded} to the hard instance 
$(\Funscaled,\gunscaled)$ defined in Eq.~\pref{eq:basic-construction} and 
\pref{eq:pairwise-construction} provides the lower bound we want, but 
restricted to algorithms with bounded iterates. To handle unbounded
iterates, we follow \citet{carmon2019lower_i} and compose the
construction with a soft projection to a ball centered at the origin. Our final 
(unscaled) construction is
\begin{equation}
  \label{eq:compression}
  \Frandcomp(x) =\Funscaled(U^\trn\rho(x)) +
  \frac{\eta}{2}\nrm*{x}^{2},~~\text{where}~
  \rho(x)=\frac{x}{\sqrt{1+\nrm*{x}^{2}/R^2}},~\mbox{$R=230\sqrt{T}$, 
  and $\eta=1/5$.}
\end{equation}
 The
corresponding stochastic gradient estimator is %
\begin{equation}
  \label{eq:compression_oracles}
  \grandcomp(x,z)= J(x)^{\trn}U\,\gunscaled(U^\trn\rho(x),z)+\eta{}\cdot 
  x,%
\end{equation}
where $J(x)=\brk[\big]{\frac{\partial\rho_i(x)}{\partial{}x_j}}_{i,j}$ is the 
Jacobian 
of $\rho$.
The next lemma shows that this new construction is difficult for any
algorithm in $\AlgRand$. The lemma has two components: First, since
the iterates always satisfy
$\nrm*{\rho(x\ind{t,k})}\leq{}R$, we can apply
\pref{lem:randomized_bounded} to this sequence to control progress. Second, the additional regularization term in \pref{eq:compression} ensures that we cannot
make the gradient small by increasing the norm, so low progress indeed
implies large gradient.
\newcommand{\Fcompgeneric}{\wh{F}_U}
\newcommand{\gcompgeneric}{\wh{g}_U}
\begin{lemma}
  \label{lem:randomized_unbounded}
  Let $\oracle$ be any oracle with 
  $\oracle_{\Frandcomp}(x,z)=(\Frandcomp(x),\grandcomp(x,z))$, where 
  $\Frandcomp$ is the compressed and rotated hard instance 
  \pref{eq:compression} and $\grandcomp$ is the corresponding 
  probability-$p$ zero chain \pref{eq:compression_oracles}. Let 
  $\delta\in(0,1)$, $d\geq{}\ceil{18\cdot230^2
    \frac{KT^2}{p}\log \frac{2KT^2}{p\delta}}$, and $U$ be uniformly 
  distributed on $\Ortho(d,T)$. Then for any $\alg\in\AlgRand(K)$,   
  with probability
  at least $1-\delta$,
  \begin{equation}
    \label{eq:randomized_unbounded}
    \min_{k\in\brk*{K}}\,\nrm[\Big]{\grad{}\Frandcomp(x\ind{t,k}_{
  	\alg\brk{\oracle_{\Frandcomp}}})}\geq{}\frac{1}{2}~~\text{for 
      all}~~t 
    \leq{}\frac{T-\log\frac{2}{\delta}}{2p}.
  \end{equation}
\end{lemma}
\noindent
(See \pref{app:proof-randomized-bounded} for a proof.)

All that remains is to verify that the final constructions \pref{eq:compression}
and \pref{eq:compression_oracles} still satisfy the various
boundedness properties required for the lower bound. The following
bounds are a consequence of a generic result about
rotation and soft projection, which we prove in \pref{app:proof-compression}.
\begin{lemma}
\label{lem:compression}
The function $\Frandcomp$ and stochastic gradient function 
$\grandcomp$ satisfy the
following properties for all $U\in\Ortho(d,T)$.
\begin{enumerate}
\item $\Frandcomp(0) - \inf_{x}\Frandcomp(x) \leq \Delta_0T$, where 
$\Delta_0=12$.
\item The first derivative of $\Frandcomp$ is $\ls_1$-Lipschitz  
continuous, where $\ell_1 = 155$.
\item $\En\nrm*{\grandcomp(x,z)-\grad{}\Frandcomp(x)}^{2}\leq{}\frac{\varsigma^2(1-p)}{p}$ for all
  $x\in\bbR^{d}$, where $\varsigma=23$.
\item $\En\nrm*{\grandcomp(x,z)-\grandcomp(y,z)}^{2}\leq{}\frac{\bar{\ls}_1^2}{p}\cdot\nrm*{x-y}^{2}$ for all
  $x,y\in\bbR^{d}$, where $\lipBar{1} = 336$.
\end{enumerate}
\end{lemma}

\section{Extensions}
\label{sec:extensions}

While \pref{thm:main_randomized} constitutes our main technical result, implying
lower bounds for methods using stochastic first-order information,
it is interesting to extend the bounds to allow more sophisticated querying
strategies and more informative oracles.

\subsection{Statistical learning oracles}\label{sec:statistical}
To this point, our assumptions on the stochastic gradient function $g(x,z)$
concern only its first and second moments (requirements~\pref{eq:g_oracle}
and~\pref{eq:mss}).  Yet the oracles in statistical learning and
stochastic approximation problems often have the common structural
property that $g(x,z)$ is the gradient of a function.  Here we show that
this property does not improve the worst-case complexity of stochastic
optimization. Specifically, we consider oracles specified by a function
$f:\bbR^{d}\times{}\cZ\to\bbR$ for which
\begin{equation}
  \label{eq:saa}
F(x)=\En\brk*{f(x,z)}\quad\text{and}\quad g(x,z)=\grad_x f(x,z).
\end{equation}
All of the lower bounds in this paper extend to this setting, at
the cost of a slightly more
involved construction. The idea is the same as in the preceding
construction, but to construct a valid function $f(x,z)$ with
$F(x)=\En\brk*{f(x,z)}$ we apply the smoothed progress function to the
function value for $\Funscaled$ rather than the gradient. Letting
$\cZ=\crl*{0,1}$ be the oracle seed space, we define 
\begin{equation}
\label{eq:f_SAA}
\funscaled(x,z) =
-\Psi(1)\Phi(x_1)\noisingfunc_1(x,z) +
\sum_{i=2}^{T}\brk*{\Psi(-x_{i-1})\Phi(-x_i) -
	\Psi(x_{i-1})\Phi(x_i)}\,\noisingfunc_i(x,z) ,
\end{equation}
where $\noisingfunc_i(x,z)$ the smoothed
indicator~\pref{eq:pairwise-construction} and the random seed
$z\sim\mathrm{Bernoulli}(p)$. %
It is immediate that $\En\brk*{\funscaled(x,z)}=\Funscaled(x)$. The
stochastic gradient function $\grad_x \funscaled(x,z)$ has a similar
form to our previous construction $\gunscaled(x,z)$, but with 
nuisance terms arising from the gradient of the soft progress
function itself. The thrust of the analysis for the new construction
is to show that these nuisance terms do not spoil the key properties of
$\gunscaled$.
\begin{lemma}
  \label{lem:pzc-saa}
  The stochastic gradient function $\grad_x\funscaled$ is a
  probability-$p$ zero-chain, is unbiased for  $\grad \Funscaled$, and 
  for numerical constants $\varsigma$ and $\lipBar{1}$ independent of $p$ 
  and $T$ satisfies 
  \begin{equation}\label{eq:pzc-saa-var}
    \E \norm{\grad\funscaled(x,z) - \grad \Funscaled(x)}^2 \le 
    \frac{\varsigma^2}{p}
  \end{equation}
  and
  \begin{equation}\label{eq:pzc-saa-mss}
    \E \norm{\grad\funscaled(x,z) - \grad\funscaled(y,z) }^2 \le 
    \frac{\lipBar{1}^2}{p} \norm{x-y}^2
  \end{equation}
  for all $x,y\in\bbR^{T}$.
      \end{lemma}
      We prove \pref{lem:pzc-saa} in \pref{app:extensions}. With the lemma in hand, all that is required to prove the
       $\Omega(\tfrac{\Delta{}L\sigma^{2}}{\veps^{4}})$ lower
      bound for the \pop setting and the 
      $\Omega(\tfrac{\Delta{}L\sigma}{\veps^{3}}+\frac{\sigma^{2}}{\veps^{2}})$
      lower bound for the \mss{} setting is to compose the instance
      with a rotation and soft projection as in \pref{eq:compression},
      then rescale as in \pref{thm:main_randomized}. This leads to the following result.
      \begin{proposition}
        \pref{thm:main_randomized} holds (with different numerical 
        constants) even when restricting the oracle class to statistical 
        learning-type stochastic gradient
        functions of the form \pref{eq:saa}.
      \end{proposition}

\newcommand{\gunscaledCo}{g_{T}^{\mathrm{coord}}}

\subsection{Active oracles}
Our main results consider a model in which the algorithm performs batches 
of $K$ simultaneous queries, but the random seed $z$ is drawn i.i.d.\ once 
per batch.
Another stronger model allows \emph{active} oracles, where the queries
consist of both a point $x$ and a seed $z$~\citep{SchmidtLeBa11,
  ShalevZh13, DefazioBaLa14, woodworth2016tight,lei2017non,
  fang2018spider,zhou2020stochastic}. Active oracles are essential to
finite-sum optimization problems where $F(x)=\sum_{i=1}^{n}f_i(x)$ and are
more general than our $K$-query oracles, since a randomized algorithm can
simulate a $K$-query oracle using an active oracle by drawing $z\sim P_z$
and querying $(x\ind{1},z),\ldots,(x\ind{K},z)$.
For \emph{convex} finite-sum minimization problems, 
stochastic oracles are significantly weaker than active oracles 
\citep{arjevani2017limitations}. Nevertheless, in this section we show
that our $\eps^{-4}$ lower bound for
zero-respecting algorithms (\pref{thm:zr_population}) extends to active 
oracles, even with additional finite-sum structure ($\cZ$ is finite, $P_z$ is 
uniform). We believe further extensions for randomized algorithms, \mss 
gradient estimators and statistical learning oracles are straightforward, but 
we omit them 
for brevity.

The precise active oracle model we consider is as follows: at 
round
$i$, the algorithm proposes a point $x\ind{i}$ and seed $z\ind{i}$ and
receives an oracle response $(F(x\ind{i}), g(x\ind{i},z\ind{i})) = \oracleF(x\ind{i},z\ind{i})$. As before,
we assume that the stochastic gradients are unbiased and have variance
bounded by $\sigma^{2}$, and we allow the algorithm to know the
distribution $P_z$.

The key step in converting our basic probabilistic zero-chain 
construction~\eqref{eq:basic-construction} to achieve a lower bound
for the active finite-sum setting 
is to allow for independent randomness in each of the chain coordinates; 
this safeguards against algorithms that ``abuse'' the 
active  oracle by repeatedly querying the same (informative) value of $z$.
More formally, we take $\crl*{0,1}^{T}$ to be the oracle seed 
space and consider the stochastic gradient function
$g_{T}^{\mathrm{coord}}:\R^T 
\times \crl*{0,1}^{T}\to\R^T$,
\begin{align} \label{eq:basic-construction-coord}
\brk*{\gunscaledCo(x,z)}_i \defeq \grad_i \Funscaled(x) \cdot 
\prn*{1+\indicator{i > \prog(x)}\prn*{\frac{z_i}{p}-1}};
\end{align}
the only difference compared to the passive construction~\pref{eq:basic-construction} is that 
the seed $z=(z_1,\ldots,z_T)$ is now a vector of $T$ bits, and we use the
$i$th bit only for coordinate $i$ of the stochastic gradient function. If we draw the bits 
of $z$ i.i.d.\ from a Bernoulli$(p)$ distribution, then $\gunscaledCo$ is 
unbiased for $\grad \Funscaled$ and satisfies the variance bound in 
\pref{lem:pzc-basic}. 

The next step is to convert the distribution over $z\in\{0,1\}^T$ into a 
uniform distribution over a larger set, so that the instance has
finite-sum structure. To do so, we assume without loss of 
generality that $p=1/N$ for $N\in\N$ (we can always round $1/p=c\cdot 
\sigma^2/\epsilon^2$ appropriately).  We choose $\cZ = 
\{1,...,N^T\}$ as the seed space and define $\zeta:\cZ\to \{0,1\}^T$ as 
\begin{equation*}
\zeta_j(k)\defeq\indicator{\text{the $j$th digit of $k$ in the $N$-ary basis 
		is 0}}.
\end{equation*}
To obtain the hard active oracle construction, we 
take
\begin{equation*}
g_\pi(x;i) \defeq \gunscaledCo(x,\zeta(\pi(i))),
\end{equation*}
where $\pi$ is any permutation of $N^T$ elements. Note that for any 
choice of the permutation $\pi$, the random function $g_\pi(\cdot; i)$ with $i$ uniform in $\cZ$ has 
the same 
distribution as $\gunscaledCo(\cdot; z)$ with the elements of 
$z$ i.i.d.  Bernoulli$(p)$, 
and therefore $g_\pi$ is also unbiased for $\grad \Funscaled$ and satisfies the 
variance bound in \pref{lem:pzc-basic}. By choosing $\pi$ to be a random permutation, the active oracle 
corresponding to $g_\pi$ satisfies a progress bound analogous to \pref{lem:prob-zero-chain}.
\begin{lemma}
	\label{lem:prob-zero-chain-active}
	Let $\delta\in(0,1)$, let $N,T>1$ be integers, let $\pi$ be a random 
	permutation of 
	$N^T$ elements and consider the active oracle 
	$\oracle_{\Funscaled}^\pi(x,i)=(\Funscaled(x), g_\pi(x;i))$. Let 
	$\{x\ind{i}\}$ be the iterates of any zero-respecting algorithm 
	interacting with $\oracle_{\Funscaled}^\pi$. Then, for $p=1/N$, with 
	probability at 
	least $1-\delta$ over the random 
	choice of $\pi$,
	\begin{equation*}
	\prog[0]\prn*{
		x\ind{t}
	} < T, ~~\text{for all ~} t \leq{} \frac{T-\log(1/\delta)}{4p}.
	\end{equation*}
      \end{lemma}
We prove \pref{lem:prob-zero-chain-active} in~\pref{app:active} and
sketch the intuition behind the result here. Let $(x\ind{1},i\ind{1}),\ldots, 
(x\ind{t},i\ind{t})$ be the algorithm's queries and
$g\ind{1},\ldots,g\ind{t}$ be the oracle responses up to some iteration $t$. 
Let $\gamma=\max_{t'<t} \prog[0](g\ind{t'})$, so that $\prog[0](x\ind{t})\le 
\gamma$ by the zero-respecting assumption. For an algorithm to guarantee
$\prog[0](g\ind{t})=1+\gamma$ (and thereby make progress in $x$), the
$(1+\gamma)$th coordinate of $\zeta(\pi(i\ind{t}))$ must be 1. The key 
observation is that the algorithm's previous queries provide very little 
information on $\zeta_{1+\gamma}(\pi(\cdot))$. In particular, we argue 
that 
after $t-1$ queries, the most we 
can possibly know is a set of $t-1$ indices $i$ for which 
$\zeta_{1+\gamma}(\pi(i))=0$. Since all other indices are identically 
distributed, any query $i\ind{t}$ has probability at most $N^{T-1}/(N^T - 
(t-1))$ of satisfying $\zeta_{1+\gamma}(\pi(i\ind{t}))=1$. Since $t \le T/p 
< N^T/2$, the probability of making a unit of progress at any  
iteration is no more that $2/N=2p$, which gives the result via the same 
arguments that prove~\pref{lem:pzc-basic}.

Using the same scaling arguments as in the proof of 
\pref{thm:main_zero_respecting}, \pref{lem:prob-zero-chain-active} 
implies an analogous lower bound for the active setting. However, the 
distributional complexity we now lower bound is slightly different, because 
we randomize over the choice of oracles instead of choosing a fixed oracle. 
Consequently, we let the supremum in Eq. \pref{eq:minimax_rand} be over 
all distributions $P_\oracle$ on $\OclassPop$, and take the expectation 
also with respect to a draw of $\oracle\sim P_\oracle$. (For 
zero-respecting lower bounds, we still replace $\AlgClassRand$ with 
$\AlgClassZR$ and it still suffices to consider point masses for $P_F$). 
\begin{proposition}
\label{prop:active}
\pref{thm:zr_population} also holds in the active oracle model, with 
the above complexity measure, finite 
$\zset$, and uniform $P_z$.
\end{proposition}
This lower bound has the following implication on minimax complexity: 
 For every 
zero-respecting algorithm there exists a ``hard'' active oracle 
(corresponding to some permutation of the coordinates) for a scaled 
version of $\Funscaled$ such that finding an $\epsilon$-stationary point 
requires at least $\Omega(\epsilon^{-4})$ iterations.
Using the techniques of \pref{sec:randomized} we can lift these results to  
finite sum active oracle lower 
bounds for randomized algorithms. Moreover, the ``different bit per 
coordinate'' approach extends straightforwardly the mean-square smooth 
construction~\pref{eq:pairwise-construction} as well as the ``statistical 
learning'' construction~\pref{eq:f_SAA}. 

The set $\zset$ in the lower bounds described above 
is very large---since $N$ scales as ${\sigma^2}/{\epsilon^2}$ and $T$ is 
polynomial in $1/\epsilon$, the cardinality $|\zset|=N^T$ is 
super-exponential in 
$1/\epsilon$. 
Designing lower bound constructions with smaller cardinality $|\zset|=n$ 
remains an open problem. We note that for the mean-square smooth 
setting, the smallest possible value for 
$n$ is $\Omega(\sigma^2/\epsilon^2)$, since for  $n=o(\sigma^2/\epsilon^2)$ the 
upper bound $O(\sqrt{n}\LipGradBar\Delta \epsilon^{-2})$ attained by 
SPIDER~\citep{fang2018spider} will be smaller than the desired $n$-independent 
lower bound $\Omega(\LipGradBar\Delta\sigma \epsilon^{-3})$. We also
remark that \citet{fang2018spider} prove a lower bound of
$\Omega(\sqrt{n}\LipGradBar\Delta\epsilon^{-2})$ for active oracles,
but their construction does not keep the variance $\sigma^{2}$
bounded.

\section{Discussion}
\label{sec:discussion}

We have established tight lower bounds on the stochastic first-order complexity of finding stationary points for non-convex functions, with and without mean-squared smoothness. We hope that the basic ideas behind our lower bound constructions will find further use in non-convex stochastic optimization. A few natural open questions and future directions along these lines are as follows.

\paragraph{Lower bounds for \mss oracles with a single query} In the 
\mss setting, all known algorithms that achieve the optimal $O(\veps^{-3})$
oracle complexity (SPIDER \citep{fang2018spider}, SNVRG
\citep{zhou2020stochastic}) require $K=2$ simultaneous queries. With $K=1$,
the best result known for the \mss setting is still the standard
$O(\Delta{}L\sigma^{2}\veps^{-4})$ rate obtained by SGD. However, under
additional higher-order smoothness assumptions, perturbed SGD can achieve
convergence $O(\veps^{-3.5})$ with $K=1$~\citep{fang2019sharp}. It remains
an open question whether any algorithm can achieve complexity scaling as
$\veps^{-3}$ when $K=1$, or whether the $\veps^{-4}$ rate of SGD is optimal.

\paragraph{Lower bounds under additional oracle assumptions}
Rather than assuming a \mss oracle, one can make the stronger 
assumption that the stochastic gradient function $g(\cdot,z)$ is smooth 
almost surely, or assume that the error $\nrm*{g(x,z)-\grad{}F(x)}$ is 
bounded by $\sigma$ almost surely. 
We are not aware of any algorithms that leverage such stronger 
assumptions, and yet extending our lower bounds to handle them seems 
non-trivial. Resolving the importance of these assumptions therefore 
remains 
an interesting topic for future work. 

\paragraph{Lower bounds for higher-order algorithms} 
Our results resolve the complexity of finding first-order stationary points 
with stochastic first-order methods, but we have not addressed the oracle 
complexity of other basic non-convex stochastic optimization problems, 
such as finding first-order stationary points with higher-order smoothness 
(possibly with stochastic access to Hessian, Hessian vector-products, or 
other higher-order derivatives) or finding second-order stationary points. 
While our techniques extend to higher order derivatives and smoothness, 
obtaining tight lower bounds requires a dedicated treatment and may pose 
new challenges.

\subsection*{Acknowledgements} 
Part of this work was completed while the authors were visiting the Simons
Institute for the Foundations of Deep Learning program. We thank Ayush
Sekhari, Ohad Shamir, Aaron Sidford and Karthik Sridharan for several
helpful discussions. YC was supported by the Stanford Graduate Fellowship.
JCD acknowledges support from NSF CAREER award 1553086, the Sloan
Foundation, and ONR-YIP N00014-19-1-2288. DF was supported by NSF TRIPODS
award \#1740751. BW was supported by the Google PhD Fellowship program.

 \newpage
 \ifdefined\usefunnybib
 \printbibliography
 \else
 \setlength{\bibsep}{6pt}

 \bibliographystyle{abbrvnat}
 \fi
 
  \newpage
 
\appendix

\part*{Appendix}

\section{Proofs from \pref{sec:zero_respecting}}
\label{app:basic}

\subsection{Basic technical results}
\label{app:basics}

Before proving the main results from \pref{sec:zero_respecting}, we first 
state two self-contained technical results that will be used in subsequent 
proofs. The first result bounds component functions $\Psi$ and $\Phi$ and 
gives the calculation for the parameter $\ell_1$ in 
\pref{lem:deterministic-construction}.\pref{item:lip}.
\begin{observation}%
	\label{obs:psi_phi_bounds} The functions $\Psi$ and $\Phi$ in 
	\pref{eq:psi_phi} and their derivatives satisfy
	\begin{equation}
	\label{eq:psi_phi_bounds}
	0\leq{}\Psi\leq{}e,~~ 0\leq{}\Psi'\leq\sqrt{54/e},~~ |\Psi''| \le 
	32.5,~~
	0\leq{}\Phi\leq{}\sqrt{2\pi{}e},~~
	0\leq{}\Phi'\leq{}\sqrt{e}~~\text{and}~~|\Phi''|\le 1.
	\end{equation}
\end{observation}

\begin{proof}[\pfref{lem:deterministic-construction}.\pref{item:lip}]
	We note that the Hessian of $\Funscaled$ is tridiagonal. Consequently, 
	for any $x\in\R^d$, 
	\begin{flalign*}
	\opnorm{\grad^2 \Funscaled(x)}
	&\le \max_{i\in[T]}|\grad_{i,i}^2 \Funscaled(x)|
	+ \max_{i\in[T]}|\grad_{i,i+1}^2 \Funscaled(x)| + 
	\max_{i\in[T]}|\grad_{i+1,i}^2 \Funscaled(x)|
	\\&
	\stackrel{(i)}{\le} \sup_{z\in\R} |\Phi''(z)|  \sup_{z\in\R} |\Psi(z)| + 
	\sup_{z\in\R} |\Phi(z)| \sup_{z\in\R} |\Psi''(z)| + 2 \sup_{z\in\R} 
	|\Phi'(z)| \sup_{z\in\R} |\Psi'(z)| 	\stackrel{(ii)}{\le}  152,
	\end{flalign*}
	where $(i)$ is a direct calculation using the 
	definition~\pref{eq:f_unscaled} of $\Funscaled$ and $(ii)$ follows 
	from~\pref{eq:psi_phi_bounds}.
\end{proof}

The second result is an $\Omega(\frac{\sigma^{2}}{\veps^{2}})$ lower 
bound on the sample complexity of finding 
stationary points whenever 
$\eps\leq{}O(\sqrt{\Delta{}L})$. This result handles an edge case in the 
proof of \pref{thm:mean_smooth_zero_respecting}. A similar lower bound 
appeared in \citet{foster2019complexity}, but the result we prove here is  
slightly stronger because it holds even for dimension $d=1$.

\begin{lemma}\label{lem:stat_lower_bound}
	There exists a number $c_0 > 0$ such that for any number of 
	simultaneous queries $K$, dimension $d$ and  
	$\eps\leq{}\sqrt{\frac{\LipGradBar\Delta}{8}}$, we 
	have
	\begin{equation}
	\label{eq:stat_lb}
	\minimaxZRMSS\geq{}\minimaxRandMSS
	\geq c_0 \cdot {\frac{\sigma^{2}}{\eps^{2}}}.
	\end{equation}
\end{lemma}

Our approach for proving \pref{lem:stat_lower_bound} is as follows. Given a dimension $d$, 
we construct a function $F:\bbR^{d}\to\bbR$, a family of distributions $P_z$, 
and a family of
functions $f(x,z)$ for which $F(x)=\En_{z}\brk*{f(x,z)}$, and for
which the initial suboptimality, variance, and mean-squared smoothness
are bounded by $\Delta, \sigma^2$ and $\LipGradBar$, respectively. We then 
prove a lower bound in the \emph{global stochastic model} in which at 
round $t$ the oracle returns the full function $f(\cdot, z\ind{t})$, rather 
than just its value and derivatives at the queried point.  The global
stochastic model is more powerful than the $K$-query
stochastic first-order model (with $g(x,z)=\grad_x f(x,z)$) for every
value of $K$, so this will imply the claimed result as a special
case. 
\begin{lemma}\label{lem:stat_lower_bound_global}
Whenever $\eps\leq{}\sqrt{\frac{\LipGradBar\Delta}{8}}$, the number of samples  
required to obtain an $\epsilon$-stationary point in the global stochastic 
model defined above is $\Omega(1) 
\cdot \frac{\sigma^{2}}{\eps^{2}}$.
\end{lemma}

\begin{proof}[\pfref{lem:stat_lower_bound_global}]
	The proof follows standard arguments used to derive information-theoretic 
	lower bounds for statistical estimation 
	\citep{lecam1973convergence,yu1997assouad}.

We consider a family of functions $f:\R^d\times \R\to\R$ given by
\begin{align}
f(x,z)=
\frac{\bar{L}}2 \left( \|x\|^2 - 2 z x_1  +r^2 \right),
\end{align}
where $r\in (0,\sqrt{2\Delta/\bar{L}})$ is a fixed parameter. We take $P_z$
to have the form $P_z^{s}\ldef\cN(rs,\frac{\sigma^2}{\bar{L}^2})$, where
$s\in\{-1,1\}$, and let $\theta_s \defeq (rs,0,\dots, 
0)\in\R^d$.
Then, when $P_z=P_z^{s}$, we have $F_{s}(x) \defeq \E_z\brk*{f(x,z)}
=\frac{\bar{L}}2 \|x-\theta_s\|^2$, and furthermore for any $x,y\in\bbR^{d}$ we have
\begin{align*}
&\E_z[ \| \nabla_xf(x,z) - \nabla F_{s}(x)  \|^2] 	
                 = \bar{L}\cdot\E_z[(z-rs)^2] = \sigma^2,
                 \intertext{and}
&\E_z[ \| \nabla_x f(x,z) -  \nabla_x f(y,z) \|^2] 
= \bar{L}^2\cdot\|x-y\|^2.
\end{align*}
Note that $F_s$ is indeed an $\bar{L}$-smooth, and has initial suboptimality at $x\ind{0}=0$ bounded as $F_s(0)-\inf_{x\in\R^d} F_s(x)=\bar{L}r^2/2\le 
\Delta$.

Now, we provide a distribution over the underlying
instance by drawing $S$ uniformly from $\pmo$, and consider any algorithm that 
takes as input samples $z_1,\ldots,z_T\sim P_z^S$, and returns iterate 
$\hat{x}$. To bound the 
expected norm of the gradient at $\hat{x}$ (over the randomness of the 
oracle, the randomness of the algorithm, and the choice of the
underlying instance $S$), we define $\hat{S} \defeq
\argmin_{s'\in\{1,-1\}} \| 
\nabla F_{s'}(\hat{x})\|$, with ties broken arbitrarily. Observe that we have
\begin{align}\label{eq:exp_basic_bound}
\E[\|\nabla F_{S}(\hat{x})\|] &\overset{(i)}{\ge} r\bar{L}\P\left(\|\nabla 
F_{S}(\hat{x})\| \ge r\bar{L}\right) 
\overset{(ii)}{\ge}  r\bar{L} \P(\hat{S}\neq S),
\end{align}
where (i) follows by Markov's inequality and (ii) follows because when 
$\hat{S}\neq{}S$, the definition of $\hat{S}$ implies %
\begin{align*}
2\cdot\nrm*{\grad{}F_{S}(\hat{x})}
&\ge \inf_{x\in\R^d} \{\|\nabla F_{-1}(x) 
\| + \|\nabla F_{1}(x) \|\}\\
&= \bar{L}\cdot\inf_{x\in\R^d} \{ \| x-\theta_1 \| + \|x-\theta_{-1} \| \}
 \ge \bar{L}\|\theta_{1} - \theta_{-1}\|= 2r\bar{L}.
\end{align*}
Next, for $s\in\pmo$ let
$\P_{s}=\cN^{\otimes T}(rs,\frac{\sigma^2}{\bar{L}^2})$ denote the law of 
$(z_1,\dots,z_T)$ conditioned on $S=s$. We have 
\begin{align*}
\P(\hat{S}\neq S) = 1- \P(\hat{S}= S) &\ge 1 - \frac12 \sup_{A \text{ is 
measurable}} \{\P_1(A) 
+ 
\P_{-1}(A^c)\} \\
&= \frac12 - \frac12 \sup_{A \text{ is measurable}} \{\P_1(A) -
\P_{-1}(A)\} \nonumber \\
&= \frac12 \left\{1-\|\P_{1}-\P_{-1}\|\right\}\\
&\ge \frac12 \left\{1-\sqrt{\frac12 D_{\mathrm{KL}}(\P_{1}||\P_{-1})}\right\}\\
&=
\frac12 \left(1- \frac{r\bar{L}\sqrt{T}}{\sigma}\right),
\end{align*} 
where the penultimate step follows by Pinsker's inequality and the last 
step uses that $\P_{s}=\cN^{\otimes T}(rs,\frac{\sigma^2}{\bar{L}^2})$.
Combining this lower bound with \pref{eq:exp_basic_bound} yields 
\begin{align*}
\E[\|F_S(\hat{x})\|] \ge \frac{r\bar{L} }{2}\left( 
1-\frac{r\bar{L}\sqrt{T}}{\sigma} 
\right).
\end{align*}
Finally, setting $r= 
\min\{\frac{\sigma}{2\bar{L}\sqrt{T}},\sqrt{\frac{2\Delta}{\bar{L}}}\}$, implies 
\begin{align*}
  \max\left\{\E[\|F_1(\hat{x})\|],\E[\|F_{-1}(\hat{x})\|]\right\}
  &\ge
\frac12\left(\E[\|F_1(\hat{x})\|]  + \E[\|F_{-1}(\hat{x})\|]\right)
  \\
  &= 
\E[\|F_S(\hat{x})\|] 
\ge \min\crl*{\frac{\sigma}{8\sqrt{T}},\sqrt{\frac{\bar{L}\Delta}{8}}}.
\end{align*}
Stated equivalently, whenever $\epsilon\le \sqrt{{\bar{L}\Delta}/{8}}$, there exists 
$s\in\{-1,1\}$ such that the number of 
oracle calls $T$ required to ensure $\E[\|\nabla F_{s}(\hat{x})\|]\le \epsilon$ 
satisfies
\begin{align*}
T\ge \frac{\sigma^2}{64 \epsilon^2},
\end{align*}
concluding the proof.
\end{proof}

\subsection{\pfref{lem:pzc-pair}}\label{app:pzc-pair}
	First, we note that $\E \brk*{\noisingfunc_i (x,z)}=1$ for all $x$ and $i$, and 
	therefore $\E \brk*{\gunscaled(x,z)} = \grad \Funscaled(x)$. Moreover, 
	by the same argument argument used in the proof of 
	\pref{lem:pzc-basic}, \pref{lem:deterministic-construction}.\ref{item:zero-chain} 
	implies that $\brk{\gunscaled(x,z)}_i = \grad_i \Funscaled(x) =0$ for all 
	$i>\prog(x)+1$, all $x\in\R^T$ and all $z\in\{0,1\}$. In addition, for 
	$i=\prog(x)+1$ we have $\threshfunc(|x_{\ge i}|) = 0$ and therefore 
	$\softindfunc_i(x)=\threshfunc(1)=1$ and 
	$\noisingfunc_i(x,z)=\frac{z}{p}$. Consequently, we have 
	$\P\prn{\brk{\gunscaled(x,z)}_{\prog(x)+1} \ne 0} \le p$, establishing 
	that the oracle is a probability-$p$ zero-chain. 
	
	To bound the variance of the gradient estimator we observe that for all 
	$i\le \prog[\half](x)$, $\nrm{\threshfunc(|x_{\ge i}|)}\ge 
	\threshfunc(1/2)=1$ and therefore $\softindfunc_{i}(x)=0$ and 
	$\noisingfunc_i(x,z)=1$, so that
	\begin{equation*}
	\brk{\gunscaled(x,z)}_i = \grad_i \Funscaled(x)~~\forall i \le 
	\prog[\half](x).
	\end{equation*}
	On the other hand, \pref{lem:deterministic-construction}.\ref{item:zero-chain}  
	gives 
	us that
	\begin{equation*}
	\brk{\gunscaled(x,z)}_i = \grad_i \Funscaled(x)=0~~\forall i > 
	1+\prog[\half](x).
	\end{equation*}
	We conclude that $\delta(x,z) = \gunscaled(x,z)  - \grad \Funscaled(x)$ 
	has at most a single nonzero entry in coordinate $i_x = 
	\prog[\half](x)+1$. Moreover, for every $i$
	\begin{equation*}
	\delta_i(x,z) = \grad_i \Funscaled(x) (\noisingfunc_i(x,z)-1) = 
	\grad_i \Funscaled(x) \softindfunc_i(x) \prn*{\frac{z}{p}-1}.
	\end{equation*}
	Therefore,
	\begin{equation*}
	\E \norm{\gunscaled(x,z) - \grad \Funscaled(x)}^2
	= \E \delta_i^2(x,z)
	= \abs*{\grad_{i_x} \Funscaled(x)}^2 \,\softindfunc_i^2(x)\,\frac{1-p}{p}
	\le 
	\frac{(1-p)23^2}{p},
	\end{equation*}
	where the final transition used 
	\pref{lem:deterministic-construction}.\ref{item:grad} and 
	$\softindfunc_i^2(x)\le1$ for all $x$ and $i$, establishing the variance 
	bound in~\pref{eq:pzc-var-mss} with $\varsigma = 23$.
	
	To bound $\E \norm{\gunscaled(x,z) - \gunscaled(y,z) }^2$, we use that $\E\brk*{\delta(\cdot ,z)}=0$ and that $\delta(\cdot,z)$ has at most 
	one nonzero coordinate to write
	\begin{flalign}
	\E \norm{\gunscaled(x,z) - \gunscaled(y,z) }^2  & = \E 
	\norm{\delta(x,z) - \delta(y,z) }^2 + 
	\norm{\grad\Funscaled(x)-\grad\Funscaled(y)}^2
	\nonumber \\ &
	\label{eq:sparse-mss-expression}
	= \sum_{i\in\{i_x, i_y\}} \E \prn*{\delta_{i}(x,z) - \delta_{i}(y,z)}^2 
	+ \norm{\grad\Funscaled(x)-\grad\Funscaled(y)}^2,
	\end{flalign}
	where $i_y=\prog[\half](y)+1$ is the nonzero index of $\delta(y,z)$. For 
	any $i\le T$, we have
	\begin{flalign*}
	\E \prn*{\delta_{i}(x,z) - \delta_{i}(y,z)}^2 
	&= \prn*{ \grad_i \Funscaled(x) \softindfunc_i(x) - \grad_i 
		\Funscaled(y) 
		\softindfunc_i(y)}^2 \, \E \prn*{\frac{z}{p}-1}^2\\
	&= \prn*{ \grad_i \Funscaled(x) \prn{ \softindfunc_i(x)  - 
			\softindfunc_i(y)} + \prn{\grad_i \Funscaled(x) - \grad_i \Funscaled(y) 
		}\softindfunc_i(y)}^2 \, \frac{1-p}{p}\\
	& \le \prn*{ 2 (\grad_i \Funscaled(x))^2 \prn{ \softindfunc_i(x)  - 
			\softindfunc_i(y)}^2  + 2\prn{\grad_i \Funscaled(x) - \grad_i 
			\Funscaled(y) 
		}^2\softindfunc_i^2(y)} \, \frac{1}{p} ~.
	\end{flalign*}
	By \pref{obs:thresh}.\ref{item:thresh-lip}, $\threshfunc_i$ is 
	6-Lipschitz. Since the Euclidean norm $\norm{\cdot}$ is 1-Lipschitz, we 
	have
	\begin{flalign*}
	\abs*{\softindfunc_i(x)  - 
		\softindfunc_i(y)}
	&\le 6 \, \abs[\big]{\,\norm*{\threshfunc(\abs{x_{\ge i}})} -
		\norm*{\threshfunc(\abs{y_{\ge i}})}\,} 
	\le 
	6 \, \norm[\big]{\,\threshfunc(\abs{x_{\ge i}})-
		\threshfunc(\abs{y_{\ge i}})\,}
	\\ &
	\le 6^2 \, \nrm[\big]{\,\abs{x_{\ge i}} - \abs{y_{\ge i}}\,} 
	\le 6^2 \, \nrm*{x-y}.
	\end{flalign*}
	That is, $\softindfunc_i$ is $6^2$-Lipschitz.  Since 
	$\softindfunc_i^2(y) \le 1$ and $(\grad_i \Funscaled(x))^2 \le 23^2$ by 
	\pref{lem:deterministic-construction}.\ref{item:grad}, we have
	\begin{equation*}
	\prn*{\delta_{i}(x,z) - \delta_{i}(y,z)}^2 
	\le 
	\frac{(23\cdot 6)^2 \norm{x-y}^2 + 2(\grad_i \Funscaled(x) - \grad_i 
		\Funscaled(y))^2}{p}
	\end{equation*}
	for all $i$. 
	Substituting back into~\pref{eq:sparse-mss-expression} we obtain
	\begin{flalign*}
	\E \norm{\gunscaled(x,z) - \gunscaled(y,z) }^2 &\le
	\frac{2\cdot (23\cdot 6)^2 \norm{x-y}^2 + 2\norm{\grad 
			\Funscaled(x) - \grad
			\Funscaled(y)}^2}{p} + \norm{\grad \Funscaled(x) - \grad
		\Funscaled(y)}^2.
	\end{flalign*}
	Recalling that $\norm{\grad \Funscaled(x) - \grad
		\Funscaled(y)} \le \lip{1} \norm{x-y}$ by 
	\pref{lem:deterministic-construction}.\ref{item:lip}, establishes 
	the mean-square smoothness bound in \pref{eq:pzc-var-mss} with 
	$\lipBar{1} = \sqrt{2\cdot (\varsigma\cdot 6)^2 + 3\lip{1}^2}$.

\subsection{\pfref{thm:mean_smooth_zero_respecting}}
\label{app:mean_smooth_zero_respecting}

Let $\Delta_0, \ell_1,\varsigma$ and $\lipBar{1}$ be the numerical 
constants in 
\pref{lem:deterministic-construction}.\ref{item:val}, 
\pref{lem:deterministic-construction}.\ref{item:lip} and 
\pref{lem:pzc-pair}, respectively. 
Let the accuracy parameter $\epsilon$, initial suboptimality $\Delta$,
mean-squared smoothness parameter $\LipGradBar$, and 
variance parameter $\sigma^2 $ be fixed, and let $\LipGrad \le 
\LipGradBar$ be specified later. We rescale $\Funscaled$ as in the proof of 
\pref{thm:zr_population},
\begin{align*}
\Fscaled(x)=\frac{\LipGrad\lambda^{2}}{\lip{1}}
\Funscaled\left(\frac{x}{\lambda}\right),\quad\text{ 
	where\quad
}\lambda=\frac{\lip{1} }{L}\cdot 2\epsilon,\quad\text{ and }\quad 
T=\floor*{
	\frac{\Delta}{\Delta_0 (\LipGrad \lambda^2 /\ell_1)}} =
\floor*{
	\frac{\LipGrad\Delta}{\Delta_0\lip{1} 
		(2\epsilon)^{2}}}.
\end{align*}
This guarantees that $\Fscaled\in\Fclass$ and that the corresponding 
scaled gradient estimator 
$\gscaled(x,z)=(L\lambda/\lip{1})\gunscaled(x/\lambda,z)$ is such that 
every zero respecting algorithm $\alg$ interacting with 
$\oracle_{\Fscaled}(x,z)=(\Fscaled(x),\gscaled(x,z))$ satisfies
\begin{equation*}
\E\norm[\big]{ \grad \Fscaled\prn[\big]{
		x\ind{t,k}_{\alg\brk{\oracleF}}
} }
>\eps,
\end{equation*}
for all $t\leq{} {(T-1)}/{2p}$ and $k\in[K]$. It remains to choose $p$ and 
$\LipGrad$ such that $\oracle_{\Fscaled}$ belongs to $\OclassMSS$. As in 
the proof of \pref{thm:zr_population},  setting $p=\min \left\{ 
{(2\vsigma\epsilon)^2}/{\sigma^2},1\right\}$ and using \pref{lem:pzc-pair}
guarantees a variance bound of 
$\sigma^{2}$. Moreover, by \pref{lem:pzc-pair} we have
\begin{align*}
\E \norm{\gscaled(x,z) - \gscaled(y,z) }^2 
&=	\left(\frac{L\lambda}{\lip{1}}\right)^2	\E 
\left\|{\gunscaled\left(\frac{x}{\lambda},z\right)  -
	\gunscaled\left(\frac{y}{\lambda},z\right)   }\right\|^2 \le 
\left(\frac{L\lambda}{\lip{1}}\right)^2\frac{\lipBar{1}^2}{p} 
\left\|\frac{x}{\lambda}-\frac{y}{\lambda}\right\|^2
\\&=
\prn*{\frac{\lipBar{1}L}{\lip{1}\sqrt{p}}}^2 \norm{x-y}^2.
\end{align*}
Therefore, taking
\begin{equation*}
L = \frac{\lip{1}}{\lipBar{1}} \LipGradBar \sqrt{p} = 
\frac{\lip{1}}{\lipBar{1}}\min\crl*{\frac{2\varsigma\veps}{\sigma},1}\LipGradBar
 \le \LipGradBar
\end{equation*}
guarantees membership in the oracle class
 and implies the lower bound
 \begin{equation*}
 \minimaxZRMSS
 >
 \frac{T-1}{2p}
 =\left(\left\lfloor 
 \frac{\LipGradBar\Delta\sqrt{p}}{4\lipBar{1}\Delta_0\epsilon^2}\right\rfloor
  -1\right)
 \frac{1}{2p}.
 \end{equation*}
 We consider  the cases 
 $\frac{\LipGradBar\Delta\sqrt{p}}{4\lipBar{1}\Delta_0\epsilon^2}\ge 3$ 
 and $\frac{\LipGradBar\Delta\sqrt{p}}{4\lipBar{1}\Delta_0\epsilon^2}<3$ 
 separately. 
 In the former case (which is the more interesting one), we use 
 $\floor{x}-1\ge x/2$ for $x \ge 3$ and the 
 setting of $p$ to write 
 \begin{equation}
  \label{eq:mss_lb}
  \minimaxZRMSS \ge 
  \frac{\LipGradBar\Delta}{16\lipBar{1}\Delta_0\epsilon^2\sqrt{p}}\ge 
 \frac{1}{32\lipBar{1}\Delta_0 \varsigma} \cdot \frac{\LipGradBar \Delta 
 \sigma}{\epsilon^3}.
 \end{equation}
Moreover, we choose $c'=12\lipBar{1}\Delta_0$ so that $\epsilon \le 
\sqrt{\frac{\LipGradBar 
 \Delta}{12\lipBar{1}\Delta_0}}\le \sqrt{\frac{\LipGradBar\Delta}{8}}$ holds. 
 By 
 \pref{lem:stat_lower_bound_global},
\begin{equation}
\label{eq:stat_lb2}
\minimaxZRMSS
> c_0\cdot\frac{\sigma^{2}}{\eps^{2}},
\end{equation}
where $c_0$ is a universal constant (this lower bound holds for any
value of $d$). Together, the bounds \pref{eq:mss_lb} and \pref{eq:stat_lb2} 
imply the desired result when 
$\frac{\LipGradBar\Delta\sqrt{p}}{4\lipBar{1}\Delta_0\epsilon^2}\ge 3$. 

Finally, we consider the edge case 
$\frac{\LipGradBar\Delta\sqrt{p}}{4\lipBar{1}\Delta_0\epsilon^2}<3$. We 
note that the assumption $\epsilon \le 
\sqrt{\frac{\LipGradBar\Delta}{12\lipBar{1}\Delta_0}}$ precludes the option 
that $p=1$ in this case. Therefore we must have 
$\frac{\LipGradBar\Delta\varsigma}{2\lipBar{1}\Delta_0\sigma\epsilon}<3$ 
or, 
equivalently, $\frac{\sigma^2}{\epsilon^2} > 
\frac{\varsigma}{6\lipBar{1}\Delta_0} \cdot 
\frac{\LipGradBar\Delta\sigma}{\epsilon^3}$. Thus, in this case the bound 
\pref{eq:stat_lb2}  implies \pref{eq:mss_lb} up to a constant, concluding 
the proof.

 \section{Proofs from \pref{sec:randomized}}
\label{app:randomized}
\newcommand{\evV}{\mathfrak{V}}
\newcommand{\evE}{\mathfrak{T}_j\ind{i-1}}
\newcommand{\evG}{\mathfrak{G}}

\subsection{Proof of \pref{lem:randomized_bounded}}
\label{app:proof-randomized-bounded}
The proof combines the techniques of the proofs
	of \pref{lem:prob-zero-chain} and Lemma 4 of 
	\cite{carmon2019lower_i}. 
  	Let us adopt the shorthand 
  	$x\ind{i}\ldef{}x\ind{i}_{\alg\brk{\oracle_{\Frandgeneric}}}$, 
  which we recall is defined via 
  \[x\ind{i}_{\alg\brk{\oracle_{\Frandgeneric}}}=\alg\ind{i}\prn*{r, 
    \oracle_{\Frandgeneric{}}\big(x\ind{1}_{\alg\brk{\oracle_{\Frandgeneric}}},z\ind{1}\big),
     \ldots,
  	\oracle_{\Frandgeneric{}}\big(x\ind{i-1}_{\alg\brk{\oracle_{\Frandgeneric}}},z\ind{i-1}\big)},
  \]
  	 where $r$ is the 
  algorithm's random seed. Further, recall that $x\ind{i}$ is a batch of $K$
  queries,
\[
x\ind{i}=(x\ind{i,1},\ldots,x\ind{i,K}).
\]
For each $i$ and each $k\in\brk*{K}$, define
\[
g\ind{i,k} = \grandgeneric{}(U^{\trn}x\ind{i,k},z\ind{i}),
\]
and let $g\ind{i}=(g\ind{i,1},\ldots,g\ind{i,K})$.  To keep notation compact for the $K$-query setup, we adopt the
following conventions throughout the proof: 
\begin{itemize}
\item $F\ind{i} \defeq \brk[\big]{\Frandgeneric(x\ind{i,1}), \ldots, 
\Frandgeneric(x\ind{i,K})} = \brk[\big]{F(U^\trn x\ind{i,1}), \ldots, 
F(U^\trn x\ind{i,K})}$,
\item $Ug\ind{i}\ldef{} \brk[\big]{Ug\ind{i,1},\ldots,Ug\ind{i,K}}$,
\item $U^{\trn}x\ind{i}\ldef{} 
\brk[\big]{U^{\trn}x\ind{i,1},\ldots,U^{\trn}x\ind{i,K}}$.
\end{itemize}
Note that with this notation we have $(F\ind{i}, U
g\ind{i})=\oracle_{\Frandgeneric{}}(x\ind{i},z\ind{i})$. 

Following the
strategy of \pref{lem:prob-zero-chain}, we define
  \begin{equation*}
  \pi\ind{t} = \max_{i\le t}\max_{k\in\brk{K}}\prog(U^{\trn} x\ind{i,k}) = \max
  \crl*{j\le T\mid 
  	\abs{\tri{u\ind{j}, x\ind{i,k}}} \ge \quarter \text{ for some }i\le
        t, k\in\brk*{K}}
  \end{equation*}
  and, similarly,
  \begin{equation*}
  \gamma\ind{t} = \max_{i\le t} \max_{k\in\brk{K}}\prog[0](g\ind{i,k}) = \max \crl*{j\le T\mid 
  	g_j\ind{i,k} \ne 0 \text{ for some }i\le t, k\in\brk*{K}}.
  \end{equation*}
The statement of the lemma is equivalent to
  \begin{equation*}
  \P \prn{ \pi\ind{t} \ge T } \le \delta ~~\mbox{for all}~~t\le 
  \frac{T-\log\frac{2}{\delta}}{2p}.
  \end{equation*}
  Define the event
  \begin{equation*}
  \evV\ind{t} \defeq \crl*{ \max_{k\in\brk{K}}\prog(U^{\trn} x\ind{t,k}) \le \gamma\ind{t-1} }.
  \end{equation*}
  Note that by definition $\cap_{i=1}^t \evV\ind{t}$ implies that $\pi\ind{t} 
  \le 
  \gamma\ind{t-1}\le \gamma\ind{t}$, and therefore
  \begin{equation}\label{eq:rand-bound-decomp}
  \P \prn{ \pi\ind{t} \ge T } \le 
  \P \prn*{ \gamma\ind{t} \ge T \;,\; \cap_{i=1}^t \evV\ind{i}} 
  + \P\prn*{ \brk*{\cap_{i=1}^t \evV\ind{i}}^c }.
  \end{equation}
  We bound each of the terms above in turn. With an argument similar to the 
  proof of \pref{lem:prob-zero-chain} we show that
  \begin{equation}\label{eq:rand-pzc-bound}
  \P \prn*{ \gamma\ind{t-1} \ge T \;,\;  \cap_{i=1}^t \evV\ind{i}}
  \le e^{2pt - T}.
  \end{equation}
  With an argument similar to the proof of Lemma 4 of 
  \cite{carmon2019lower_i}, we show that
  \begin{equation}\label{eq:rand-zc-bound}
  \P\prn*{ \brk*{\cap_{i=1}^t \evV\ind{i}}^c }
  \le 2 tKT e^{-\frac{d-tK-T}{32 R^2 tK}}.
  \end{equation}
  Taking $t\le \frac{T-\log\frac{2}{\delta}}{2p}$ and $d\ge  18 
  \frac{R^2 KT}{p}\log \frac{2KT^2}{p\delta}\ge tK + T +  32R^2 tK \log 
  \frac{2tKT}{\delta/2}$ and substituting~\pref{eq:rand-pzc-bound} 
  and~\pref{eq:rand-zc-bound} back into~\pref{eq:rand-bound-decomp} 
  gives $\P \prn{ \pi\ind{t} \ge T } \le \frac{\delta}{2} + 
  \frac{\delta}{2}\le\delta$, establishing the result. We now derive the 
  bounds~\pref{eq:rand-pzc-bound} and~\pref{eq:rand-zc-bound}.
  
  \subsubsection{Proof of the bound~\pref{eq:rand-pzc-bound}}
  Define the filtration
  \begin{equation*}
  \cG\ind{i} = \sigma(r,U^{\trn} x\ind{1}, g\ind{1},\ldots, U^{\trn} x\ind{i}, 
  g\ind{i}),
  \end{equation*}
  so that $\pi\ind{t},\gamma\ind{t}\in\cG\ind{t}$. 
The definition of the probability-$p$ zero chain property is that
  \begin{equation*}
  \P\prn*{ \max_{k\in\brk*{K}}\prog[0](g\ind{t,k}) \le 1 +\max_{k\in\brk*{K}}\prog(U^{\trn} x\ind{t,k})\mid\cG\ind{t-1} }= 
  1,
\end{equation*}
and
\begin{equation*}
  \P\prn*{ \max_{k\in\brk*{K}}\prog[0](g\ind{t,k}) = 1 + \max_{k\in\brk*{K}}\prog(U^{\trn} x\ind{t,k}) \mid \cG\ind{t-1}} 
  \le p.
  \end{equation*}
  Recalling the definitions $\evV\ind{t} = \crl*{ \max_{k\in\brk*{K}}\prog(U^{\trn} x\ind{t,k}) \le 
  \gamma\ind{t-1} }$ and 
	\[\gamma\ind{t} = \max\crl*{
  \gamma\ind{t-1}, \max_{k\in\brk*{K}}\prog[0](g\ind{t,k})},\] 
  we have (by the reasoning of Eq.~\pref{eq:iota}),
   \begin{equation*}
  \P\prn*{  \gamma\ind{t} -\gamma\ind{t-1} \notin 
  \{0,1\}\;,\;\evV\ind{t}\mid\cG^{t-1} 
  } 
  = 0
  ~~\mbox{and}~~
  \P\prn*{  \gamma\ind{t} -\gamma\ind{t-1} =1\;,\;\evV\ind{t} 
  \mid 
  \cG\ind{t-1}
  } 
  \le p.
  \end{equation*}
  
  Therefore, denoting $\iota\ind{t} \defeq \gamma\ind{t} - 
  \gamma\ind{t-1}$, we have via the Chernoff method
  \begin{equation*}
  \P \prn*{ \gamma\ind{t} \ge T \;,\;  \cap_{i=1}^t \evV\ind{i}} 
  = 
  \P\prn*{ e^{\sum_{i=1}^t \iota\ind{i}} \indicator{\cap_{i=1}^t 
  \evV\ind{i}}\ge 
  e^T}
  \le e^{-T}\E\brk*{ {e^{\sum_{i=1}^t \iota\ind{i}} }\indicator{\cap_{i=1}^t 
  \evV\ind{i}}}
  \end{equation*}
  Using $\iota\ind{t}, \evV\ind{t}\in\cG\ind{t}$ and 
  $\P(\iota\ind{t}=1~,~\evV\ind{t}\mid\cG\ind{t-1}) 
  =1-\P(\iota\ind{t}=0~,~\evV\ind{t}\mid\cG\ind{t-1})\le
  p$, we obtain
  \begin{equation*}
  \E \brk*{{e^{\sum_{i=1}^t \iota\ind{i}} } \indicator{\cap_{i=1}^t 
  	\evV\ind{i}}}
 = \E \prod_{i=1}^t e^{\iota\ind{i}}\indicator{\evV\ind{i}}
  = \E \prod_{i=1}^t \E\brk*{ e^{\iota\ind{i}}\indicator{\evV\ind{i}} 
  \;\Big\vert\; \cG\ind{i-1}}
  \le (1-p+ p\cdot e)^t %
  \le e^{2pt},
  \end{equation*}
  establishing the bound~\pref{eq:rand-pzc-bound}.

  \subsubsection{Proof of the bound~\pref{eq:rand-zc-bound}}
  Throughout, we fix a time horizon $t\in\N$. Let us adopt the
  convention that
  $\spn\crl*{x\ind{i}}=\spn\crl*{x\ind{i,1},\ldots,x\ind{i,K}}$. We start by defining, for every 
  $j\le T$ and $i\le t$,
  \begin{equation*}
  \projop_{j}\ind{i} \defeq \mbox{projection operator to orthogonal 
  complement of 
  }\mathrm{span}\crl{u\ind{1},\ldots,u\ind{j},x\ind{1},\ldots,x\ind{i}},
  \end{equation*}
  That is,  $\projop_{j}\ind{i} \in \R^{d\times d}$ is symmetric and satisfies 
  $(\projop_{j}\ind{i})^2 
  = \projop_{j}\ind{i}$ and $\projop_{j}\ind{i} w = 0$ for every $w$ in $\spn\crl*{u\ind{1},\ldots,u\ind{j},x\ind{1},\ldots,x\ind{i}}$. We
  also consider the operator $\projopP_{j}\ind{i} \ldef I -
  \projop_{j}\ind{i}$, which is simply the projection onto $\spn\crl*{u\ind{1},\ldots,u\ind{j},x\ind{1},\ldots,x\ind{i}}$.
We define the 
  events
  \begin{equation*}
  \evG_j\ind{i} \defeq \crl*{ \nrm*{\projopP_{j-1}\ind{i}\projop_{j-1}\ind{i-1} u\ind{j}}
  	< \frac{1}{4R\sqrt{t}}}.
  \end{equation*}
  The following is a linear-algebraic fact.
  \begin{lemma}\label{lem:g-to-v}
  For every $i\le t$ and $j\le T$, $\cap_{i'\le i} \evG_j\ind{i'}$ implies that 
\begin{equation*}
  \abs*{\tri{u\ind{j} , x\ind{i,k}}} < \frac{\norm{x\ind{i,k}}}{4R} 
  \sqrt{\frac{i}{t}} 
  \le  
  \quarter~~\mbox{for all }k.
\end{equation*}
  \end{lemma}
  \begin{proof}%
    Consider the operator $\projopP_{j-1}\ind{i} = I - \projop_{j-1}\ind{i}$ and observe that 
    $\projopP_{j-1}\ind{i}\projopP_{j-1}\ind{i-1} 
    =\projopP_{j-1}\ind{i-1}\projopP_{j-1}\ind{i}=\projopP_{j-1}\ind{i-1}$
    by the nesting of the subspaces. Therefore
    \begin{equation*}
    \projopP_{j-1}\ind{i'-1}+
    \projop_{j-1}\ind{i'-1}\projopP_{j-1}\ind{i'}\projop_{j-1}\ind{i'-1}
    =
    \projopP_{j-1}\ind{i'-1}+
    (I-\projopP_{j-1}\ind{i'-1})\projopP_{j-1}\ind{i'}(I-\projopP_{j-1}\ind{i'-1})
    =\projopP\ind{i'}_{j-1}.
    \end{equation*}
     Iterating this equality, we obtain
 	\begin{align}
 	\projopP_{j-1}\ind{i} &= \projopP\ind{0}_{j-1} + 
 	\sum_{i'=1}^{i}\projop_{j-1}\ind{i'-1}\projopP\ind{i'}_{j-1}\projop_{j-1}\ind{i'-1}
          = \sum_{j'=1}^{j-1} \brk{u\ind{j'}}\brk{u\ind{j'}}^{\trn}
 	+ \sum_{i'=1}^{i}\projop_{j-1}\ind{i'-1} 
 	\projopP\ind{i'}_{j-1}\projop_{j-1}\ind{i'-1}, \label{eq:proj-decomp}
 	\end{align}
        where the second equality uses that $\projopP\ind{0}_{j-1}=
        \sum_{j'=1}^{j-1} \brk{u\ind{j'}}\brk{u\ind{j'}}^{\trn}$ since $U$ is 
        orthogonal.

        Now, let $k\in\brk*{K}$ be fixed. Using the facts that $\projopP_{j-1}\ind{i}x\ind{i,k} = x\ind{i,k}$ and 
 	$\tri{u\ind{j'},u\ind{j}}=0$ for any $j'<j$, we may write
 	\begin{flalign}
 	\abs*{\tri{u\ind{j} , x\ind{i,k}}} 
 	 & =
 	\abs*{\tri{ u\ind{j} , \projopP_{j-1}\ind{i} x\ind{i,k}}}
        \stackrel{(i)}{=} \abs*{\sum_{i'=1}^{i}\tri*{\projop_{j-1}\ind{i'-1} 
        u\ind{j},
            \projopP\ind{i'}_{j-1}\projop_{j-1}\ind{i'-1}x\ind{i,k}}}\notag\\
        & \stackrel{(ii)}{=} 
        \abs*{\sum_{i'=1}^{i}\tri*{\projopP\ind{i'}_{j-1}\projop_{j-1}\ind{i'-1}
            u\ind{j},
            \projopP\ind{i'}_{j-1}\projop_{j-1}\ind{i'-1}x\ind{i,k}}}\notag\\
        &\stackrel{(iii)}{\le} 
         \sqrt{\sum_{i'=1}^{i}\nrm*{\projopP\ind{i'}_{j-1}\projop_{j-1}\ind{i'-1}
            u\ind{j}}^{2}}\cdot\sqrt{\sum_{i'=1}^{i}\nrm*{\projopP\ind{i'}_{j-1}\projop_{j-1}\ind{i'-1}x\ind{i,k}}^{2}}.
 				\label{eq:g-to-v-cs}
 	\end{flalign}
 where the transitions above follow from $(i)$ Eq.~\eqref{eq:proj-decomp}, 
 $(ii)$ the fact that $\brk*{\projopP\ind{i'}_{j-1}}^2 = \projopP\ind{i'}_{j-1}$ 
 and $(iii)$ Cauchy-Schwartz. Now, $\cap_{i'\le i} 
 	\evG_j\ind{i'}$ implies that 
 	\begin{equation}\label{eq:g-to-v-u}
\sum_{i'=1}^{i}\nrm*{\projopP\ind{i'}_{j-1}\projop_{j-1}\ind{i'-1}
            u\ind{j}}^{2}
 	\le 
 	\sum_{i'=1}^{i} \frac{1}{16 R^2 t} \le \frac{i}{16 R^2 t}.
 	\end{equation}
 	Moreover, the decomposition~\eqref{eq:proj-decomp} implies that 
 	$\sum_{i'=1}^{i}\projop_{j-1}\ind{i'-1} 
 	\projopP\ind{i'}_{j-1}\projop_{j-1}\ind{i'-1} \preceq
 	\projopP_{j-1}\ind{i} \preceq I$. Therefore,
 	\begin{equation}\label{eq:g-to-v-x}
\sum_{i'=1}^{i}\nrm*{\projopP\ind{i'}_{j-1}\projop_{j-1}\ind{i'-1}x\ind{i,k}}^{2}
=
\brk*{x\ind{i,k}}^\trn \prn*{ \sum_{i'=1}^i \projop_{j-1}\ind{i'-1} 
\projopP\ind{i'}_{j-1}\projop_{j-1}\ind{i'-1}} \brk*{x\ind{i,k}}
\le %
 	\norm{x\ind{i,k}}^2.
 	\end{equation}
 	Substituting~\pref{eq:g-to-v-u} and~\pref{eq:g-to-v-x} 
 	into~\pref{eq:g-to-v-cs} gives the lemma.
  \end{proof}

\pref{lem:g-to-v} has the following immediate consequence: for all 
$i\le t$,
\begin{equation*}
\cap_{i'\le i} \cap_{j > \gamma\ind{i-1}} \evG_{j}\ind{i'}
~~\mbox{implies}~~
\crl*{\max_{k\in\brk*{K}}\abs{\tri{u\ind{j}, x\ind{i,k}}} < \quarter \text{ for all }j> \gamma\ind{i-1}}
= \evV\ind{i}.
\end{equation*}
Furthermore, since 
$\gamma\ind{1}\le\gamma\ind{2}\le\cdots\gamma\ind{t}$,
\begin{equation*}
\cap_{i\le t} \cap_{j \le T} \prn*{ \evG_{j}\ind{i} ~\mbox{or}~ \crl{ j \le 
\gamma\ind{i-1}} }=
\cap_{i\le t} \cap_{j > \gamma\ind{i-1}} \evG_{j}\ind{i}
~~\mbox{implies}~~  \cap_{i\le t} \evV\ind{i}.
\end{equation*}
Therefore, we may bound the failure probability of $\cap_{i\le t} 
\evV\ind{i}$ as
\begin{equation}\label{eq:cv-to-cg}
\P\prn*{ \brk*{\cap_{i=1}^t \evV\ind{i}}^c }
\le 
\sum_{i=1}^t \sum_{j=1}^T \P \prn*{\brk*{\evG_j\ind{i}}^c , j > 
\gamma\ind{i-1}}.
\end{equation}
It remains to argue that $\P \prn*{\brk*{\evG_j\ind{i}}^c , j > 
	\gamma\ind{i-1}} \le 2Ke^{-\frac{d-tK-T}{32R^2 tK}}$, which we proceed to 
	do below; substituting this bound into~\pref{eq:cv-to-cg} gives the 
	desired result~\pref{eq:rand-zc-bound}.

\begin{lemma}\label{lem:rand-gc-bound}
	For all $i\le t$ and $j\le T$,
	\begin{equation*}
	\P \prn*{\brk*{\evG_j\ind{i}}^c , j > 
		\gamma\ind{i-1}} \le 2Ke^{-\frac{d-iK-j}{32R^2 tK}}
	\le 2K e^{-\frac{d-tK-T}{32R^2 tK}}.
	\end{equation*}
\end{lemma}
\begin{proof}%
	 For any $i,j$, define the $\sigma$-field
	\begin{equation*}
	\cU_{j}\ind{i} \defeq \sigma\prn{
	  r, u\ind{1}, \ldots, u\ind{j}, U^{\trn} x\ind{1}, z\ind{1}, 
	  \ldots, U^{\trn} 
	  x\ind{i}, z\ind{i}
	},
	\end{equation*}
	where $z\ind{i}$ is the randomness of the oracle at iteration $i$. Fixing 
	$i\le t$, $j\le T$, we also define
	\begin{equation*}
	\evE \defeq \crl{j > \gamma\ind{i-1}}
	~~\mbox{and}~~
	\hat{u}\ind{j} \defeq \frac{\projop_{j-1}\ind{i-1}u\ind{j}}
	{\norm{\projop_{j-1}\ind{i-1}u\ind{j}}}.
	\end{equation*}
	With this notation, we have
	\begin{flalign*}
	\P\prn*{\brk*{\evG_j\ind{i}}^c , \evE \mid \cU_{j-1}\ind{i-1}}
	 &\le \P\prn*{\brk*{\evG_j\ind{i}}^c  \mid \cU_{j-1}\ind{i-1}, \evE}
	\\& = 
	\P \prn*{
          \nrm*{
\projopP_{j-1}\ind{i}\projop_{j-1}\ind{i-1}u\ind{j}
}
		\ge \frac{1}{4R\sqrt{t}} \;\Big\vert\;  \cU_{j-1}\ind{i-1}, \evE
	}
	\\ &
	\le 
	\P \prn*{
          \nrm*{
            \projopP_{j-1}\ind{i}\hat{u}\ind{j}
}
		\ge \frac{1}{4R\sqrt{t}}
		 \;\Big\vert\;  \cU_{j-1}\ind{i-1}, \evE
	},
	\end{flalign*}
	where we use the convention $\P(\cdot \mid \cU_{j-1}\ind{i-1}, 
	\evE)=0$ when 
	$\P(\evE\mid \cU_{j-1}\ind{i-1})=0.$ 
	
	Observe that to compute the oracle response to query $x\ind{i}$ it 
	suffices to know $z\ind{i}, U^{\trn} x\ind{i}$ and $u\ind{1}, ... 
	u\ind{\gamma(i)}$. To see this, recall that 
	$\oracle_{\Frandgeneric}(x\ind{i,k},z\ind{i}) = (F(U^{\trn} x\ind{i,k}), U 
	g(U^{\trn} x\ind{i,k}, 
	z\ind{i}))$; knowing $z\ind{i}, U^{\trn} x\ind{i,k}$ is sufficient to compute 
	$F(U^{\trn} x\ind{i,k})$ and $g\ind{i,k} = g(U^{\trn} x\ind{i,k}, z\ind{i})$, 
	and hence 
	only the first $\max_{k\in\brk*{K}}\prog[0](g\ind{i})\le \gamma\ind{i}$ vectors in $U$ are 
	necessary to compute $U g\ind{i}$.
        (This also implies that $\evE \in 
	\cU_{j-1}\ind{i-1}$). Therefore, using 
	information in 
	$\cU_{\gamma\ind{i}}\ind{i}$ we can compute all oracle responses up to 
	iterate $i$, and 
	since the algorithm random seed $r\in \cU\ind{i}_j$ for all $j$, this 
	allows us to also compute the 
	next query. We 
	thus conclude that
	\begin{equation*}
	x\ind{i} \in \cU\ind{i-1}_{\gamma\ind{i-1}}.
	\end{equation*}
	In other words, $x\ind{1},x\ind{2}, \ldots, x\ind{i}$ are deterministic 
	conditional on $\cU_{j-1}\ind{i-1}$ and 
	$\evE$.
	
	The above discussion implies also that $\projopP_{j-1}\ind{i}$
        and $\projop_{j-1}\ind{i-1}$ are
	deterministic conditional on $\cU_{j-1}\ind{i-1}$ and 
	$\evE$. In contrast, we have the following characterization of 
	 $\hat{u}\ind{j}$.
	 \begin{lemma}\label{lem:rand-rot-invariance}
	 	Conditional on $\cU_{j-1}\ind{i-1}$ and 
	 	$\evE$, the unit vector $\hat{u}\ind{j}$ is uniformly distributed on the 
	 	unit sphere in the range of $\projop_{j-1}\ind{i-1}$, i.e.\ the linear 
	 	space 
	 	$\cS^\perp$ defined as the orthogonal complement of $\cS = 
	 	\mathrm{span}\crl{u\ind{1},\ldots,u\ind{j-1},x\ind{1},\ldots,x\ind{i-1}}$.
	 \end{lemma}
 
 	Before proving \pref{lem:rand-rot-invariance}, let us quickly show 
 	how it implies \pref{lem:rand-gc-bound}. Since $\hat{u}\ind{j}$
        is conditionally uniformly distributed on a sphere in
        $\cS^{\perp}$, and since the image $\projopP_{j-1}^{i}\cS^{\perp}$ has
        dimension at most $K$, we have
 	\[
        \P \prn*{
          \nrm*{
            \projopP_{j-1}\ind{i}\hat{u}\ind{j}
          }
 		\ge \frac{1}{4R\sqrt{t}}
 		\;\big\vert\;  \cU_{j-1}\ind{i-1}, \evE
 	} =
 \P\prn*{\sum_{i=1}^K v_i^2 \ge \frac{1}{16R^2 t}},
\]
 for $v$ uniform on the unit 
 	sphere in $\R^{d'}$, where $d'=\mathrm{dim}(\cS^\perp)\ge 
 	 d-iK-j$. Also, observe that we have
\[
\P\prn*{\sum_{i=1}^K v_i^2 \ge
  \frac{1}{16R ^2 t}}
\leq{} \P\prn*{\exists{}k\in\brk*{K}: v_k^2 \ge
  \frac{1}{16 R^2 t K}}
\leq{} K\cdot{}\P\prn*{v_1^2 \ge \frac{1}{16 R^2 t K}}.
\]
\pref{lem:rand-gc-bound} thus follows from the
concentration bound $\P(v_1^2 \ge \alpha) \le 2e^{-\half 
 	 \alpha
 	 d'}$~\cite[Lecture 8]{ball1997elementary}.
 	 \end{proof}
 
 	\begin{proof}[\pfref{lem:rand-rot-invariance}]
 	  Throughout, for any sequence of vectors $v\ind{1},\ldots,v\ind{N}$, we 
 	  adopt the notation 
 	  $v\ind{\ge 
 	  	n}$ (respectively $v\ind{< n}$) for a matrix with columns $v\ind{n}, 
 	  v\ind{n+1}, \ldots, v\ind{N}$ (respectively $v\ind{1}, \ldots, 
 	  v\ind{n-1})$. We define a number of densities as follows:
          \begin{itemize}
            \item $p_{\ge j}$ denotes the density of $u\ind{\ge j}$ conditional on 
	  $\cU_{j-1}\ind{i-1}$ and 
	  $\evE$.
          \item $\tilde{p}\ind{<i}$ denotes the density of $U^{\trn} 
            x\ind{<i}$ and $\evE$ conditional on $U,r,z\ind{<i}$.
          \item $\tilde{p}\ind{<i}_{<j}$ denotes the density of $U^{\trn} 
            x\ind{<i}$ and $\evE$ conditional on $u\ind{<j},r,z\ind{<i}$.
          \item $p_{\mathrm{rotation}}$ and $p_{\mathrm{rotation},<j}$
              denote the densities for $U$ and $u\ind{<j}$, respectively.
          \end{itemize}
          (Pedantically, densities are with respect to the product
          of Lebesgue and counting
          measure.)
          With these definitions, we have
	  \begin{equation}\label{eq:monster-chain-rule}
	  p_{\ge j}(u\ind{\ge j} \mid 
	  \cU_{j-1}\ind{i-1}, \evE) = 
	  \frac{
  			\tilde{p}\ind{< i}\prn*{U^{\trn} x\ind{<i},  \evE \mid U, r, 
  				z\ind{<i}} \,
  			p_{\mathrm{rotation}}(U)}
  		{\tilde{p}\ind{<i}_{<j}\prn*{U^{\trn} x\ind{<i},  \evE \mid u\ind{<j}, r, 
  				z\ind{<i}} \, p_{\mathrm{rotation},<j}\prn*{u\ind{<j}} \, 
  			},
	  \end{equation}
	  where we used the chain rule and that the randomness of the 
	  algorithm and the oracle is independent of $U$; 
	  the factor $p_{\text{seed}}(r,z\ind{<i})$ consequently cancels 
	  in the numerator and denominator. Note that 
	  $(U,r,z\ind{<i})$ is all the 
	  information necessary to compute $U^{\trn} x\ind{<i}$ and hence also 
	  $\evE$. Therefore, $\tilde{p}\ind{<i}$ is a Dirac delta 
	  constraining its 
	  argument to be consistent with the
          conditioning.
	  
	  Fix $r,z\ind{<i}$ and $U$ such that $\evE$ holds and let $W$ be any 
	  orthogonal transformation preserving $\cS$, i.e., a $d$ by $d$ matrix 
	  satisfying
	  \begin{equation}\label{eq:rot-w-reqs}
	  W^{\trn} W = I_d~~\mbox{and}~~W s = s = W^{\trn} s~\mbox{for all}~
	  s \in \cS = \mathrm{span}\crl{u\ind{1}, \ldots, u\ind{j-1}, x\ind{1}, 
	  \ldots, x\ind{i-1}}.
	  \end{equation}
	  Let ${x'}\ind{1}, \ldots, 
	  {x'}\ind{i-1}$ denote the iterates produced by the algorithm when we 
	  replace $U$ with  $U'=WU$ (with $z\ind{<i}$ 
	  and $r$ unchanged). We argue inductively that
	  \begin{equation}\label{eq:rot-invariance}
	  {x'}\ind{<i} = x\ind{<i}
	  ~~\mbox{and therefore, by definition of $W$,}~~
	  {U'}^{\trn} {x'}\ind{<i} = U^{\trn} W^{\trn} x\ind{<i} = U^{\trn} x\ind{<i}.
	  \end{equation}
	  To do so, for any $i'<i$ write the oracle response to ${x'}\ind{i'}$  as 
	  $\oracle_{\tilde{F}_{U'}}({x'}\ind{i'}, z\ind{i'}) = ({F'}\ind{i'}, U' 
	  {g'}\ind{i'})$, 
	  where ${F'}\ind{i'} = F({U'}^{\trn} {x'}\ind{i'})$ and ${g'}\ind{i'} = g\prn{ 
	  	{U'}^{\trn} {x'}\ind{i'}, z\ind{i'}}$. 	   
	  The basis of the induction is that 
	  ${x'}\ind{1}=x\ind{1}$ since they depend only on $r$. Assume that 
	  ${x'}\ind{<i'}=x\ind{<i'}$ for some $i'<i-1$; this also implies that 
	  \begin{equation*}
	  (U')^{\trn} {x'}\ind{<i'}=(U')^{\trn} x\ind{<i'}=U^{\trn} W^{\trn} x\ind{<i'} = U^{\trn}  
	  x\ind{<i'}
	  \end{equation*}
	  Therefore, ${F'}\ind{<i'}={F}\ind{<i}$ and 
	  ${g'}\ind{<i'}={g}\ind{<i}$. This means $U' {g'}\ind{<i'}=W 
	  U{g}\ind{<i'}= U {g}\ind{<i}$, where the final equality is due to the 
	  invariance~\pref{eq:rot-w-reqs} of $W$ and the fact that $U g\ind{i'} 
	  \in \mathrm{span}\crl{u\ind{1},\ldots, u\ind{\gamma\ind{i'}}}\subseteq 
	  \mathrm{span}\crl{u\ind{1},\ldots, u\ind{j-1}}$ by the 
	  assumption that $\evE$ holds. Therefore, all the oracle responses for 
	  the first $i'-1$ iterations are identical, and so we must have 
	  ${x'}\ind{i'}=x\ind{i'}$, completing the induction.
	  
	  The equality~\pref{eq:rot-invariance} means that the transformation 
	  $U \mapsto WU$ leaves $\cU_{j-1}\ind{i-1}$ 
	  unchanged 
	  and in particular that $\evE$ still holds. Thus, 
	  \begin{equation}\label{eq:invariance-1}
	  \tilde{p}\ind{< i}\prn*{U^\trn x\ind{<i},  \evE \mid WU, r, 
	  	z\ind{<i}}=\tilde{p}\ind{< i}\prn*{U^{\trn} x\ind{<i},  \evE \mid U, r, 
	  	z\ind{<i}} 
	  \end{equation}
  	(recall that $\tilde{p}\ind{< i}$ just checks consistency).
	   We also have 
	   \begin{equation}\label{eq:invariance-2}
	   p_{\mathrm{rotation}}(W U) = 
	   p_{\mathrm{rotation}}(U) 
	   \end{equation}
	   by the orthogonal invariance of the 
	   distribution of $U$. Substituting $W u\ind{\ge j}$ into 
	   equation~\eqref{eq:monster-chain-rule} for $p_{\ge j}(\cdot \mid 
	   \cU_{j-1}\ind{i-1}, \evE)$
	   gives
	   \begin{equation}\label{eq:monster-chain-rule-W}
	   p_{\ge j}(W u\ind{\ge j} \mid 
	   \cU_{j-1}\ind{i-1}, \evE) = 
	   \frac{
	   	\tilde{p}\ind{< i}\prn*{U^{\trn} x\ind{<i},  \evE \mid WU, r, 
	   		z\ind{<i}} \,
	   	p_{\mathrm{rotation}}(WU)}
	   {\tilde{p}\ind{<i}_{<j}\prn*{U^{\trn} x\ind{<i},  \evE \mid u\ind{<j}, r, 
	   		z\ind{<i}} \, p_{\mathrm{rotation},<j}\prn*{u\ind{<j}} \, 
	   },
	   \end{equation}
	   where we have used the facts that $\brk*{u\ind{<j}, W u\ind{\ge j}} = 
	   WU$ by 
	   definition of $W$, and  that the quantity $U^\trn x\ind{<i}$ 
	   appearing in~\eqref{eq:monster-chain-rule} is 
	   $\cU_{j-1}\ind{i-1}$-measurable and therefore independent of the 
	   argument to 
	   $p_{\ge j}(\cdot \mid 
	   \cU_{j-1}\ind{i-1}, \evE)$. Applying the 
	   equalities~\pref{eq:invariance-1} and~\pref{eq:invariance-2} to the 
	   numerator of~\pref{eq:monster-chain-rule-W} and comparing 
	   to~\pref{eq:monster-chain-rule}, we find that
	    \begin{equation*}
	    p_{\ge j}(W u\ind{\ge j} \mid 
	   \cU_{j-1}\ind{i-1}, \evE) =  p_{\ge j}(u\ind{\ge j} \mid 
	   \cU_{j-1}\ind{i-1}, \evE).
	    \end{equation*}
	    Marginalizing, we conclude that the 
	   distribution of $u\ind{j}$ conditional on $\cU_{j-1}\ind{i-1}$ is 
	   invariant to any linear transformation $W$ 
	   satisfying~\pref{eq:rot-w-reqs}. Since any rotation of $\cS^\perp$ 
	   can be extended to a rotation of $\R^d$ 
	   satisfying~\pref{eq:rot-w-reqs}, we have that the component of 
	   $u\ind{j}$ in $\cS^\perp$ is rotationally invariant, giving the lemma.
\end{proof}

\subsection{Proof of \pref{lem:randomized_unbounded}}
\label{app:proof-randomized-unbounded}

Before proving \pref{lem:randomized_unbounded} we first list the relevant 
continuity properties of the compression function
\begin{equation*}
\rho(x) = \frac{x}{\sqrt{1+\norm{x}^2/R^2}}.
\end{equation*}

\begin{lemma}
\label{lem:compression_smooth}
Let $J(x)=\frac{\partial\rho}{\partial{}x}(x) = 
\frac{I - \rho(x)\rho(x)^{\trn}/R^{2}}{\sqrt{1+\nrm*{x}^{2}/R^{2}}}$. For all 
$x,y\in\bbR^{d}$ we have
\begin{equation}
  \label{eq:compression_smooth}
  \nrm*{J(x)}_{\op}=\frac{1 }{\sqrt{1+\nrm*{x}^{2}/R^{2}}} \leq{}1, ~ 
  \nrm*{\rho(x)-\rho(y)}\leq{}\nrm*{x-y}, 
  ~ \text{and}~ 
  \nrm*{J(x)-J(y)}_{\op}\leq{}\frac{3}{R}\nrm*{x-y}.
\end{equation}
\end{lemma}
\begin{proof}[\pfref{lem:compression_smooth}]
Note that $\norm{\rho(x)}\le R$ and therefore $0 \preceq I - 
\rho(x)\rho(x)^{\trn}/R^{2} \preceq I$. Consequently, we have
$ %
\nrm*{J(x)}_{\op} = \prn{1+\nrm*{x}^{2}/R^{2}}^{-1/2} \le 1.
$ %
The guarantee $\nrm*{\rho(x)-\rho(y)}\leq{}\nrm*{x-y}$ follows
immediately by Taylor's theorem. For the last statement, define
$h(t)=\frac{1}{\sqrt{1+t^{2}}}$, and note that
$\abs*{h(t)},\abs*{h'(t)}\leq{}1$. By triangle inequality and the
aforementioned boundedness and Lipschitzness properties of $h$, we have
\begin{align*}
&\nrm*{J(x)-J(y)}_{\op}\\
&\leq{}
  h(\nrm*{y}/R)\cdot\nrm*{\rho(x)\rho(x)^{\trn}/R^{2}-\rho(y)\rho(y)^{\trn}/R^{2}}_{\op}
+ \nrm*{I-\rho(x)\rho(x)^{\trn}/R^2}_{\op}\cdot\abs*{h(\nrm*{x}/R)-h(\nrm*{y}/R)}
  \\
&\leq{}
  \nrm*{\rho(x)\rho(x)^{\trn}/R^{2}-\rho(y)\rho(y)^{\trn}/R^{2}}_{\op}
  + \nrm*{I-\rho(x)\rho(x)^{\trn}/R^2}_{\op}\cdot\abs*{\nrm*{x}/R-\nrm*{y}/R}.
\end{align*}
For the first term, observe that for any $x,y$ we have
$\nrm*{x},\nrm*{y}\leq{}1$,
we have $\nrm{xx^{\trn}-yy^{\trn}}_{\op}\leq{}2\nrm*{x-y}$; this
follows because for any $\nrm*{v}=1$, we have
$\nrm*{(xx^{\trn}-yy^{\trn})v}\leq{}\nrm*{x-y}\abs*{\tri*{v,x}}+\nrm*{y}\abs*{\tri*{v,x-y}}\leq{}2\nrm*{x-y}$. Since
$\nrm*{\rho(x)/R}\leq{}1$, it follows that
\begin{align*}
  \nrm*{\rho(x)\rho(x)^{\trn}/R^{2}-\rho(y)\rho(y)^{\trn}/R^{2}}_{\op}
  \leq{} \frac{2}{R}\nrm*{x-y}.
\end{align*}
For the second term, we again use that $\nrm*{\rho(x)}\leq{}R$ to write
\[
\nrm*{I-\rho(x)\rho(x)^{\trn}/R^2}_{\op}\cdot\abs*{\nrm*{x}/R-\nrm*{y}/R}
\leq{}\frac{1}{R}\abs*{\nrm*{x}-\nrm*{y}}\leq{}\frac{1}{R}\nrm*{x-y}.
\]
\end{proof}

\begin{proof}[\pfref{lem:randomized_unbounded}]
The argument here is essentially identical to~\cite[Lemma 
5]{carmon2019lower_i}. 
Define $y\ind{i}=(y\ind{i,1},\ldots,y\ind{i,K})$, where $y\ind{i,k}=\rho(x\ind{i,k})$. Observe that for each
$i$ and $k$, the oracle response
$(\Frandcomp(x\ind{i,k}),\grandcomp(x\ind{i,k},z\ind{i}))$ is a
measurable function of $x\ind{i,k}$ and
$(\Frand(y\ind{i,k}),\grand(y\ind{i,k},z\ind{i}))$. Consequently, we
can regard the sequence $y\ind{1},\ldots,y\ind{T}$ as realized by some
algorithm in $\AlgRand(K)$ applied to an oracle with
$\oracle_{\Frand}(y,z)=(\Frand(y),\grand(y,z))$. \pref{lem:randomized_bounded}
then implies that as long as $d\geq{}\ceil{18\cdot230^2
  \frac{KT^2}{p}\log \frac{2KT^2}{p\delta}}\geq{}\ceil{18 
  \frac{R^2 KT}{p}\log \frac{2KT^2}{p\delta}}$, we
have that with probability at least $1-\delta$,

\begin{equation}
\max_{k\in\brk*{K}}\prog(U^{\trn}\rho(x\ind{i,k})) =
\max_{k\in\brk*{K}}\prog(U^{\trn}y\ind{i,k}) < T,\label{eq:prog_compressed}
\end{equation}
as long as $i\leq{}(T-\log(2/\delta))/2p$.

We now show that the gradient
must be large for all of the iterates. Let $i$ and $k$
 be fixed. We first consider the case where
 $\nrm*{x\ind{i,k}}\leq{}R/2$. Observe that \pref{eq:prog_compressed}
 implies that $\prog[1](U^{\trn}y\ind{i,k})<T$ and so by
 \pref{lem:deterministic-construction}.\ref{item:large-grad}, if we
 set $j=\prog[1](U^{\trn}y\ind{i,k})+1$, we have
 \begin{equation}
\label{eq:good_condition}
\abs{\tri{u\ind{j},y\ind{i,k}}}<1\quad\text{and}\quad\abs*{\tri*{u\ind{j},\grad{}\Frand(y\ind{i,k})}}\geq{}1.
 \end{equation}
Now, observe that we have
\begin{align*}
  \tri*{u\ind{j},\grad{}\Frandcomp(x\ind{i,k})}
  &= \tri*{u\ind{j},J(x)^{\trn}\grad{}\Frand(y\ind{i,k})} +
  \eta\tri*{u\ind{j},x\ind{i,k}}.
\end{align*}
Using that $J(x) = 
\frac{I - \rho(x)\rho(x)^{\trn}/R^{2}}{\sqrt{1+\nrm*{x}^{2}/R^{2}}}$,
  this is equal to
  \begin{align*}
\frac{\tri*{u\ind{j},\grad{}\Frand(y\ind{i,k})}}{\sqrt{1+\nrm*{x\ind{i,k}}^{2}/R^{2}}}
- \frac{\tri*{u\ind{j},y\ind{i,k}}\tri*{y\ind{i,k},\Frand(y\ind{i,k})}/R^{2}}{\sqrt{1+\nrm*{x\ind{i,k}}^{2}/R^{2}}}
  + \eta\tri*{u\ind{j},y\ind{i,k}}\sqrt{1+\nrm*{x\ind{i,k}}^{2}/R^{2}}.
  \end{align*}
Since $\nrm*{y\ind{i,k}}\leq{}\nrm*{x\ind{i,k}}\leq{}R/2$, this
implies
\begin{align*}
\abs*{\tri*{u\ind{j},\grad{}\Frandcomp(x\ind{i,k})}}
\geq{}
  \frac{2}{\sqrt{5}}\abs*{\tri*{u\ind{j},\grad{}\Frand(y\ind{i,k})}}
- \abs*{\tri*{u\ind{j},y\ind{i,k}}}\prn*{
\frac{\nrm*{\Frand(y\ind{i,k})}}{2R}
+ \eta\frac{\sqrt{5}}{2}
}.
\end{align*}
By \pref{lem:deterministic-construction} we have
$\nrm*{\Frand(y\ind{i,k})}\leq{}23\sqrt{T}$. At this point, the choice
$\eta=1/5$, $R=230\sqrt{T}$, as well as \pref{eq:good_condition} imply
that $\abs*{\tri*{u\ind{j},\grad{}\Frandcomp(x\ind{i,k})}}
\geq{}
  \frac{2}{\sqrt{5}} - \prn*{\frac{1}{20} + \frac{1}{2\sqrt{5}}}\geq{}\frac{1}{2}$.

Next, we handle the case where $\nrm*{x\ind{i,k}}>R/2$. Here, we have
\[
\nrm*{\grad\Frandcomp(x\ind{i,k})} \geq{} \eta\nrm*{x\ind{i,k}}
- \nrm*{J(x\ind{i,k})}_{\op}\nrm*{\grad\Frand(y\ind{i,k})}
\geq{} \frac{R}{10} \geq{} \sqrt{T}.
\]
where the second inequality uses that $\nrm*{J(x\ind{i,k})}_{\op}\leq{} 
\frac{1}{\sqrt{1+\nrm*{x\ind{i,k}}^{2}/R^{2}}}\leq{}2/\sqrt{5}$ which follows 
from \pref{lem:compression_smooth} and $\nrm*{x\ind{i,k}}>R/2$.
\end{proof}

\subsection{Proof of \pref{lem:compression}}
\label{app:proof-compression}

To establish \pref{lem:compression} we first prove a generic result showing 
that composition with the compression function $\rho$ and an orthogonal 
transformation $U$ never significantly hurts the regularity requirements in 
our lower bounds. In the following, we use the notation $a\vee b \defeq 
\max\{a,b\}$.

\newcommand{\Ftemp}{\wh{F}_{U}}
\newcommand{\gtemp}{\wh{g}_{U}}
\begin{lemma}
\label{lem:compression_generic}
Let $F:\bbR^{T}\to\bbR$ be an arbitrary
twice-differentiable function with $\nrm*{\grad{}F(x)}\leq{}\ls_0$ and
$\nrm*{\grad{}F(x)-\grad{}F(y)}\leq{}\ls_1\cdot\nrm*{x-y}$, and let $g(x,z)$ 
and a random variable $z\sim P_z$ satisfy for all $x,y\in\bbR^{T}$,
\begin{equation}
\label{eq:g_properties}
\En\brk*{g(x,z)}=\grad{}F(x),\quad\En\nrm*{g(x,z)-F(x)}^{2}\leq{}\sigma^{2},\quad\text{and}\quad\En\nrm*{g(x,z)-g(y,z)}^{2}\leq{}\LipGradBar^2\nrm*{x-y}^{2}.
\end{equation}
Let $R\ge\lip{0} \vee 1$, $d\ge T$, and $U\in\Ortho(d,T)$. Then the functions
\begin{equation*}
\wh{F}_U(x) =F(U^{\trn}\rho(x))~~~~\mbox{and}~~~~
\wh{g}_U(x,z)=J(x)^{\trn}U{}g(U^{\trn}\rho(x),z)
\end{equation*}
satisfy the following properties.
\begin{enumerate}
\item $\Ftemp(0) - \inf_{x}\Ftemp(x) \leq F(0)-\inf_{x}F(x)$.
\item The first derivative of $\Ftemp$ is $(\lip{1}+3)$-Lipschitz
  continuous.%
\item $\En\nrm[\big]{\gtemp(x,z)-\grad{}\Ftemp(x)}^{2}\leq{} \sigma^2$ 
for all
  $x\in\bbR^{d}$.
\item $\En\nrm*{\gtemp(x,z)-\gtemp(y,z)}^{2}
\leq{}
(\bar{L}^{2} + 9\sigma^{2} + 9)\nrm*{x-y}^{2}$ for all
  $x,y\in\bbR^{d}$.
\end{enumerate}
\end{lemma}
\begin{proof}[\pfref{lem:compression_generic}]
Property 1 is immediate, since the range of $\rho$ is a subset of
$\bbR^{T}$. For property 2, we use the triangle inequality along with
\pref{lem:compression_smooth} and the assumed smoothness properties
of $F$ as follows:
\begin{align*}
&  \nrm*{\grad{}\Ftemp(x)-\grad{}\Ftemp(y)}\\
  &\leq{}
    \nrm*{J(x)^\trn U\grad{}F(U^{\trn}\rho(x))-J(x)^\trn 
    U\grad{}F(U^{\trn}\rho(y))}
    +
    \nrm*{J(x)U\grad{}F(U^{\trn}\rho(y))-J(y)U\grad{}F(U^{\trn}\rho(y))}\\
  &\leq{}
    \nrm[\big]{\grad{}F(U^{\trn}\rho(x))-\grad{}F(U^{\trn}\rho(y))}
    + \nrm[\big]{\grad{}F(U^{\trn}\rho(y))}\cdot\opnorm{J(x)-J(y)}\\
  &\leq{}
    \ls_1\cdot\nrm*{\rho(x)-\rho(y)}
+ \ls_0\cdot\nrm*{J(x)-J(y)} \\
  &\leq{} \prn*{\ls_1 + \frac{3\ls_0}{R}}\nrm*{x-y}.
\end{align*}

For the variance bound (property 3), observe that we have
\begin{align*}
\En\nrm*{\gtemp(x,z)-\grad{}\Ftemp(x)}^{2}
& =
  \En\nrm*{J(x)^{\trn}Ug(U^{\trn}\rho(x),z)-J(x)^{\trn}U\grad{}F(U^{\trn}\rho(x))}^{2}
  \\
&\leq{}
   \En\brk*{\nrm*{J(x)^{\trn}U}_{\op}^{2}\cdot\nrm*{g(U^{\trn}\rho(x),z)-\grad{}F(U^{\trn}\rho(x))}^{2}}
  \\
&\leq{}
  \En\nrm*{g(U^{\trn}\rho(x),z)-\grad{}F(U^{\trn}\rho(x))}^{2} \leq{} 
  \sigma^{2}.
\end{align*}
Here the second inequality follows from \pref{eq:compression_smooth}
and the fact that $U\in\Ortho(d,T)$,
and the third inequality follows because the variance bound in
\pref{eq:g_properties} holds uniformly for all points in the
domain $\bbR^{T}$ (in particular, those in the range of $x\mapsto{}U^{\trn}\rho(x)$).

Lastly, to prove property 4 we first invoke the triangle inequality and
 the elementary inequality $(a+b)^{2}\leq{}2a^{2}+2b^{2}$.
\begin{align*}
&\En\nrm*{\gtemp(x,z)-\gtemp(y,z)}^{2} \\
&=
  \En\nrm*{J(x)^{\trn}Ug(U^{\trn}\rho(x),z)-J(y)^{\trn}Ug(U^{\trn}\rho(y),z)
  }^{2} \\
&\leq{}
  2\En\nrm*{J(x)^{\trn}Ug(U^{\trn}\rho(x),z)-J(x)^{\trn}Ug(U^{\trn}\rho(y),z)
  }^{2}
+ 2\En\nrm*{\prn*{J(x)^{\trn}-J(y)^{\trn}}Ug(U^{\trn}\rho(x),z)}^{2}
\end{align*}
For the first term, we use the Jacobian operator norm bound from \pref{eq:compression_smooth} and the
assumed mean-squared smoothness of $g$:
\begin{align*}
  \En\nrm*{J(x)^{\trn}Ug(U^{\trn}\rho(x),z)-J(x)^{\trn}Ug(U^{\trn}\rho(y),z)}^{2}
  &\leq{} \En\nrm*{g(U^{\trn}\rho(x),z)-g(U^{\trn}\rho(y),z)}^{2} \\
  &\leq{} \LipGradBar^2\En\nrm*{\rho(x)-\rho(y)}^{2}\\
  &\leq{} \LipGradBar^2\En\nrm*{x-y}^{2}.
\end{align*}
For the second term, we use the Jacobian Lipschitzness from
\pref{eq:compression_smooth}:
\begin{align*}
  \En\nrm*{\prn*{J(x)^{\trn}-J(y)^{\trn}}Ug(U^{\trn}\rho(x),z)}^{2}&\leq{}  
  \frac{9}{R^{2}}\nrm*{x-y}^{2}\cdot\En\nrm[\big]{g(U^{\trn}\rho(x),z)}^{2}
\end{align*}
We now use the assumed Lipschitzness of
$F$ and variance bound for $g$:
\begin{align*}
\En\nrm*{g(U^{\trn}\rho(x),z)}^{2} 
  =\En\nrm*{g(U^{\trn}\rho(x),z)-\grad{}F(U^{\trn}\rho(x))}^{2} +
  \nrm*{\grad{}F(U^{\trn}\rho(x))}^{2}
\leq{} \sigma^{2}+\ls_0^2.
\end{align*}
Putting everything together, we have
\[
\En\nrm*{\gtemp(x,z)-\gtemp(y,z)}^{2}
\leq{}\prn*{\LipGradBar^2+9\sigma^{2}/R^{2}+9\ls_0^{2}/R^{2}}\cdot\nrm*{x-y}^{2}.
\]
\end{proof}

\begin{proof}[\pfref{lem:compression}] 
For property 1, observe that $\Frandcomp(0)=\Funscaled(0)$,
and \[\min_{x}\Frandcomp(x)\geq{}\min_{x}\Funscaled(U^{\trn}\rho(x))\geq\min_{x}\Funscaled(U^{\trn}x)\geq{}\min_{x}\Funscaled(x).\]

For properties 2, 3, and 4 we observe from
\pref{lem:compression_generic} that $\Frandcomp$ and $\grandcomp$, ignoring the quadratic regularization term, satisfy the same
smoothness, variance, and mean-squared smoothness bounds as in
\pref{lem:deterministic-construction}/\pref{lem:pzc-pair}/\pref{lem:pzc-saa}
up to constant factors. The additional regularization term in \pref{eq:compression}
leads to an additional $\eta=1/5$ factor in the smoothness and
mean-squared-smoothness.
\end{proof}

\subsection{\pfref{thm:main_randomized}}

We prove the lower bound for
  the \pop and \mss{} settings
  in turn. The proofs follow the same outline as the proofs of 
  \pref{thm:main_zero_respecting} and 
  \pref{thm:mean_smooth_zero_respecting}, relying on 
  \pref{lem:randomized_unbounded} and \pref{lem:compression} rather than 
  \pref{lem:prob-zero-chain} and \pref{lem:pzc-pair}, respectively. 
  Throughout, let $\Delta_0, \ell_1,\varsigma$ and $\lipBar{1}$ be the 
  numerical 
  constants in \pref{lem:compression}.
\paragraph{\Pop setting.}
  Given accuracy parameter  
$\epsilon$, initial suboptimality $\Delta$, smoothness parameter $\LipGrad$ and 
variance parameter $\sigma^2 $, we define for each $U\in\Ortho(d,T)$ a
scaled instance
\begin{align}\label{eq:final_scaling}
\Ffinal(x)=\frac{\LipGrad\lambda^{2}}{\ls_1} 
\Frandcomp\left(\frac{x}{\lambda}\right),~\text{where}~\lambda=\frac{\ls_1}{L}\cdot
 4\epsilon,~\text{and}~T=\floor*{
\frac{\Delta}{\Delta_0 (\LipGrad \lambda^2 /\ell_1)}} =
\floor*{
\frac{\LipGrad\Delta}{\lip{1} \Delta_0
	(4\epsilon)^{2}}}.
\end{align}
We assume  $T\ge 4$, or equivalently  
$\epsilon \le  \sqrt{\frac{L\Delta}{64\ls_1\Delta_0 
}}$. 
Let $\gscaled{}(x,z)$ denote the corresponding scaled version of the stochastic
gradient function $\grandcomp$.  Now, by \pref{lem:compression}, we 
have that
$\Ffinal\in\Fclass$ and moreover,
\begin{align*}
\E \norm{\gfinal(x,z) - \grad \Ffinal(x)}^2 &= 
\left(\frac{L\lambda}{\lip{1}}\right)^2
\E \left\|{\grandcomp\left(\frac{x}{\lambda},z\right) - \grad 
	\Frandcomp\left(\frac{x}{\lambda}\right)}\right\|^2 \le 
	\frac{(4\vsigma\epsilon)^2}{p}.
\end{align*}
Therefore, setting $p=\min \left\{ {(4\vsigma\epsilon)^2}/{\sigma^2},1\right\}$ 
guarantees a variance bound of 
$\sigma^{2}$.

Next, Let $\oracle$ be an oracle for which
$\oracle_{\Ffinal}(x,z)=(\Ffinal(x),\gfinal(x,z))$ for all $U\in\Ortho(d,T)$.  
Observe that for
any $\alg\in\AlgRand(K)$, we may regard the sequence
$\crl[\big]{x\ind{i,k}_{\alg[\oracle_{\Ffinal}]}/\lambda}$ as queries an 
algorithm 
$\alg'\in\AlgRand(K)$ interacting with the unscaled oracle 
$\oracle_{\Frandcomp}(x,z)=(\Frandcomp(x),\grandcomp(x,z))$. 
Instantiating \pref{lem:randomized_unbounded} for $\delta=\half$, we 
have that w.p.\ at least $\half$, $\min_{k\in\brk*{K}}\norm[\big]{ \grad 
\Frandcomp\prn[\big]{
		\frac{1}{\lambda} x\ind{t,k}_{\alg\brk{\oracle_{\Ffinal}}}
} }>\half
$ for all $t \leq{} 
\frac{T-2}{2p}$. Therefore, 
\begin{align}
\label{eq:final_grad_bound}
\E\min_{k\in\brk*{K}}\norm[\big]{ \grad \Ffinal\prn[\big]{
		x\ind{t,k}_{\alg\brk{\oracle_{\Ffinal}}}
} }
= \frac{\LipGrad\lambda}{\lip{1}}\cdot\E\min_{k\in\brk*{K}}\norm[\big]{ 
\grad \Frandcomp\prn[\big]{\tfrac{1}{\lambda}
			x\ind{t,k}_{\alg\brk{\oracle_{\Ffinal}}}
} }
\ge \frac{L\lambda}{4\ell_1}=\eps,
\end{align}
by which it follows that 
\begin{align*}
\minimaxRand
>
\frac{T-2}{2p}
=
\left(\left\lfloor \frac{L\Delta}{16\lip{1}\Delta_0 
	\epsilon^{2}}\right\rfloor -2\right)
\frac{1}{2p}
\ge 
\frac{1}{2^7 \lip{1}\Delta_0} \cdot \frac{L\Delta}{p \epsilon^2}
\ge 
\frac{1}{2^{11} \lip{1}\Delta_0\varsigma} \cdot 
\frac{L\Delta\sigma^2}{\epsilon^4},
\end{align*}
where the second inequality uses that $\floor{x}-2\geq{}x/4$ whenever 
$x\geq{}4$.

\paragraph{Mean-squared smooth setting.}
We use the scaling~\pref{eq:final_scaling}, choose $p=\min \left\{ 
{(4\vsigma\epsilon)^2}/{\sigma^2},1\right\}$  as above, and let
\begin{equation*}
L = \frac{\lip{1}}{\lipBar{1}} \LipGradBar \sqrt{p} = 
\frac{\lip{1}}{\lipBar{1}}\min\crl*{\frac{4\varsigma\veps}{\sigma},1}\LipGradBar
\le \LipGradBar.
\end{equation*}
Using \pref{lem:compression} and the calculation from the proof of 
\pref{thm:mean_smooth_zero_respecting}, this setting guarantees that 
$\oracle_{\Ffinal}(x,z)$ is in the class $\OclassMSS$. Consequently, the 
inequality~\pref{eq:final_grad_bound} implies the lower bound
 \begin{equation*}
\minimaxRandMSS
>
\frac{T-2}{2p}
=\left(\left\lfloor 
\frac{\LipGradBar\Delta\sqrt{p}}{16\lipBar{1}\Delta_0\epsilon^2}\right\rfloor
-1\right)
\frac{1}{2p}.
\end{equation*}
When  $\frac{\LipGradBar\Delta\sqrt{p}}{16\lipBar{1}\Delta_0\epsilon^2}\ge 
4$, we have $T\ge 4$ and \pref{eq:final_mss_lb} along with 
$\floor{x}-2\geq{}x/4$ for $x\ge4$ gives
 \begin{equation}\label{eq:final_mss_lb}
\minimaxRandMSS \ge 
\frac{\LipGradBar\Delta}{2^7\lipBar{1}\Delta_0\epsilon^2\sqrt{p}}\ge 
\frac{1}{2^9\lipBar{1}\Delta_0 \varsigma} \cdot \frac{\LipGradBar \Delta 
	\sigma}{\epsilon^3}.
\end{equation}
Moreover, we 
choose $c'$ so that $\epsilon \le \sqrt{\frac{\LipGradBar 
\Delta}{64\lipBar{1}\Delta_0}}\le \sqrt{\frac{\LipGradBar\Delta}{8}}$ holds. 
\pref{lem:stat_lower_bound_global} then gives the lower bound
\begin{equation}
\label{eq:fina_stat_lb}
\minimaxRandMSS
> c_0\cdot\frac{\sigma^{2}}{\eps^{2}},
\end{equation}
for a universal constant $c_0$. 
Together, the bounds \pref{eq:final_mss_lb} and \pref{eq:fina_stat_lb} 
imply the desired result when 
$\frac{\LipGradBar\Delta\sqrt{p}}{16\lipBar{1}\Delta_0\epsilon^2}\ge 4$. 
As we argue in the proof of \pref{thm:mean_smooth_zero_respecting}, in 
the complementary case 
$\frac{\LipGradBar\Delta\sqrt{p}}{16\lipBar{1}\Delta_0\epsilon^2}< 4$, the 
bound \pref{eq:fina_stat_lb} dominates \pref{eq:final_mss_lb}, and 
consequently the result holds there as well.

 \section{Proofs from \pref{sec:extensions}}
\label{app:extensions}
\subsection{Statistical learning oracles}

To prove the mean-squared smoothness properties of the 
construction~\eqref{eq:f_SAA} we must first argue about the continuity of 
$\grad \softindfunc_i$, where $\softindfunc_i:\R^T \to \R$ is the ``soft 
indicator'' 
function given by
\begin{equation*}
\softindfunc_{i}(x) \defeq 
\threshfunc\prn*{1-\prn*{\sum_{k=i}^T\threshfunc^2(|x_k|)}^{1/2}}
= \threshfunc\prn*{1-\norm*{\threshfunc\prn*{|x_{\ge 
				i}|}}}.
\end{equation*}

\begin{lemma}
\label{lem:theta_properties}
For all $i\geq{}j$, $\grad_i\Theta_j(x)$ is well-defined with
  \begin{equation}
\label{eq:grad_theta}
\grad{}_{i}\Theta_j(x) =
\begin{cases}
-\Gamma'(1-\nrm*{\Gamma(\abs*{x_{\geq{}j}})}) \cdot 
\frac{\Gamma(\abs*{x_i})}{
	\nrm*{\Gamma(\abs*{x_{\geq{}j}})}}\cdot{}\Gamma'(\abs*{x_i}) \cdot 
	\sgn(x_i),
& i \ge j ~\mbox{and}~\nrm{\Gamma(\abs{x_{\geq{}j}})}>0, \\
0, & \mbox{otherwise.}
\end{cases}
\end{equation}
Moreover, $\Theta_j$ satisfies the following properties:
\begin{enumerate}
\item
  $\nrm*{\grad{}\Theta_j(x)}\leq{}6^{2}$. \label{item:theta_lipschitz}
  \item 
  $\nrm*{\grad\Theta_j(x)-\grad\Theta_j(y)}\leq{}10^4\cdot\nrm*{x-y}$.\label{item:theta_smooth}
\end{enumerate}
\end{lemma}
\begin{proof}[\pfref{lem:theta_properties}]
First, we verify that the function
$x_i\mapsto{}\nrm*{\Gamma(\abs*{x_{\geq{}j}})}$ is differentiable
everywhere for each $i$. From here it follows from \pref{obs:thresh} that
$\Theta_j(x)$ is differentiable, and \pref{eq:grad_theta} follows from
the chain rule. Let $i\geq{}j$, and let
$a=\sqrt{\sum_{k\geq{}j,k\neq{}i}\Gamma^{2}(\abs*{x_k})}$. Then
$\nrm*{\Gamma(\abs*{x_{\geq{}j}})}=\sqrt{a^{2}+\Gamma^{2}\prn*{\abs*{x_i}}}$. This
function is clearly differentiable with respect to $x_i$ when $a>0$, and when $a=0$ it is
equal to $\Gamma(\abs*{x_i})$, which is also differentiable.

Property \pref{item:theta_lipschitz} follows because for all $j$,
  \begin{equation}
    \label{eq:grad_theta_norm}
    \nrm*{\grad{}\Theta_j(x)} \leq{}
    \frac{6}{\nrm*{\Gamma(\abs*{x_{\geq{}j}})}}\cdot\sqrt{\textstyle\sum_{i\geq{}j}\prn*{\Gamma(\abs*{x_i})\Gamma'(\abs*{x_i})}^{2}}\leq{}6^{2},
  \end{equation}
where we have used \pref{obs:thresh}.\pref{item:thresh-lip}.

To prove Property \pref{item:theta_smooth}, we restrict to the case $j=1$ so that
$x_{\geq{}j}=x$ and subsequently drop the `$\geq{}j$' subscript to
simplify notation; the case  $j>1$  follows as an immediate
consequence. Define $\mu(x)\in\bbR^{T}$ via $\mu_{i}(x) =
\Gamma(\abs*{x_i})\Gamma'(\abs*{x_i})\sgn(x_i)$. Assume without loss
of generality that
$0<\nrm*{\Gamma(\abs*{x})}\leq{}\nrm*{\Gamma(\abs*{y})}$. By triangle inequality, we have
\begin{align*}
\nrm*{\grad\Theta_1(x)-\grad\Theta_1(y)}&\leq{}
\abs*{\Gamma'(1-\nrm*{\Gamma(\abs*{x})})-\Gamma'(1-\nrm*{\Gamma(\abs*{y})})}\cdot\frac{\nrm*{\mu(x)}}{\nrm*{\Gamma(\abs*{x})}}\\
&~~~~+\Gamma'(1-\nrm*{\Gamma(\abs*{x})})\cdot\nrm*{\frac{\mu(x)}{\nrm*{\Gamma(\abs*{x})}}-\frac{\mu(y)}{\nrm*{\Gamma(\abs*{y})}}}.
\end{align*}
To proceed, we state some useful facts, all of which follow
from \pref{obs:thresh}.\ref{item:thresh-lip}:
\begin{enumerate}
\item $\Gamma$ is 6-Lipschitz.
\item $\Gamma'$ is 128-Lipschitz, and in particular
  $\Gamma'(1-\nrm*{\Gamma(\abs*{x})})\leq{}128\cdot
  \nrm*{\Gamma(\abs*{x})}$ (since $\Gamma'(1)=0$).
\item $\nrm*{\mu(x)}\leq{}6\cdot\nrm*{\Gamma(\abs*{x})}$ for all $x$.
\item $\nrm*{\mu(x)-\mu(y)}\leq{}(128\cdot 1 + 
6^2)\cdot\nrm*{x-y}=164\cdot\nrm*{x-y}$ for all $x,y$.
\end{enumerate}
Using the first, second, and third facts, we bound the first term as
\begin{align*}
  \frac{\nrm*{\mu(x)}}{\nrm*{\Gamma(\abs*{x})}}\cdot\abs*{
  	\Gamma'(1-\nrm*{\Gamma(\abs*{x})})-\Gamma'(1-\nrm*{\Gamma(\abs*{y})})}
&\leq{} 6\,
  \abs*{\Gamma'(1-\nrm*{\Gamma(\abs*{x})})-\Gamma'(1-\nrm*{\Gamma(\abs*{y})})}\\
&\leq{} 128\cdot 6 \,
  \abs*{\nrm*{\Gamma(\abs*{x})}-\nrm*{\Gamma(\abs*{y})}}\\
&\leq{} 128\cdot6^{2} \,\abs*{\nrm*{x}-\nrm*{y}}\\
&\leq{} 5000 \, \nrm*{x-y}.
\end{align*}
For the second term, we apply the second fact and the triangle inequality
to upper bound by 
\begin{align*}
&\Gamma'(1-\nrm*{\Gamma(\abs*{x})})\cdot\nrm*{\frac{\mu(x)}{\nrm*{\Gamma(\abs*{x})}}-\frac{\mu(y)}{\nrm*{\Gamma(\abs*{y})}}}\\
&\leq{}
  128\nrm*{\Gamma(\abs*{x})}\cdot\nrm*{\frac{\mu(x)}{\nrm*{\Gamma(\abs*{x})}}-\frac{\mu(y)}{\nrm*{\Gamma(\abs*{y})}}}\\
&\leq{}
  128\frac{\nrm*{\Gamma(\abs*{x})}}{\nrm*{\Gamma(\abs*{y})}}\cdot\nrm*{\mu(x)-\mu(y)}
+ 
128\nrm*{\Gamma(\abs*{x})}\nrm*{\mu(x)}\cdot\abs*{\frac{1}{\nrm*{\Gamma(\abs*{x})}}-\frac{1}{\nrm*{\Gamma(\abs*{y})}}}.
\end{align*}
Using the fourth fact and the assumption that
$\nrm*{\Gamma(\abs*{x})}\leq{}\nrm*{\Gamma(\abs*{y})}$, we have
\[
\frac{\nrm*{\Gamma(\abs*{x})}}{\nrm*{\Gamma(\abs*{y})}}
\cdot\nrm*{\mu(x)-\mu(y)}\leq{}164\nrm*{x-y}.
\]
Using the third fact and 
$\nrm*{\Gamma(\abs*{x})}\leq{}\nrm*{\Gamma(\abs*{y})}$, we have
\begin{flalign*}
\nrm*{\Gamma(\abs*{x})}\nrm*{\mu(x)}\cdot
\abs*{\frac{1}{\nrm*{\Gamma(\abs*{x})}}-\frac{1}{\nrm*{\Gamma(\abs*{y})}}}
& \leq{} 
6\nrm*{\Gamma(\abs*{x})}^{2}\cdot\abs*{\frac{1}{\nrm*{\Gamma(\abs*{x})}}-\frac{1}{\nrm*{\Gamma(\abs*{y})}}}
\\ &
=  6\frac{\nrm*{\Gamma(\abs*{x})}}{\nrm*{\Gamma(\abs*{y})}}\cdot
\abs*{{\nrm*{\Gamma(\abs*{x})}}{\nrm*{\Gamma(\abs*{y})}}}
\le 6^2 \norm{x-y}.
\end{flalign*}
Gathering all of the constants, this establishes that
\[
\nrm*{\grad\Theta_1(x)-\grad\Theta_1(y)}\leq{}10^4\cdot\nrm*{x-y}.
\]
  
\end{proof}
We are now ready to prove \pref{lem:pzc-saa}. For ease of reference, we 
restate the construction~\pref{eq:f_SAA}:
\begin{equation*}
\funscaled(x,z) =
-\Psi(1)\Phi(x_1)\noisingfunc_1(x,z) +
\sum_{i=2}^{T}\brk*{\Psi(-x_{i-1})\Phi(-x_i) -
	\Psi(x_{i-1})\Phi(x_i)}\,\noisingfunc_i(x,z),
\end{equation*}
where
\begin{equation*}
\noisingfunc_i(x,z) \defeq 1 + \softindfunc_{i}(x) \prn*{\frac{z}{p}-1}.
\end{equation*}

\newcommand{\progh}{\prog[\frac{1}{2}]}
\begin{proof}[\pfref{lem:pzc-saa}]
To begin, we introduce some shorthand. Define
\newcommand{\hone}{h_1}
\newcommand{\htwo}{h_2}
\begin{align*}
&H(s,t) = \Psi(-s)\Phi(-t)-\Psi(s)\Phi(t),\\
&\hone(s,t) = \Psi(-s)\Phi'(-t) +\Psi(s)\Phi'(t),\\
&\htwo(s,t) = \Psi'(-s)\Phi(-t) +\Psi'(s)\Phi(t).
\end{align*}
The gradient of the noiseless hard function $\Funscaled$ can then be written as
\begin{equation*}
\grad_i \Funscaled(x) = -\hone(x_{i-1},x_i) - \htwo(x_i, x_{i+1}).
\end{equation*}
Next, define 
\begin{align}
  g_i(x,z) = -\hone(x_{i-1},x_i)\cdot{}\nu_{i}(x,z)
  -\htwo(x_{i},x_{i+1})\cdot{}\nu_{i+1}(x,z).
\label{eq:saa_g}
\end{align}
With these definitions, we have the expression
\begin{align}
\grad{}_i\funscaled(x,z) = g_i(x,z)
  +\prn*{\frac{z}{p}-1}\sum_{j=1}^{i}H(x_{j-1},x_j)\cdot\grad{}_i\Theta_j(x).
\label{eq:saa_grad}
\end{align}
We first prove that $\grad\funscaled$ is a probability-$p$ zero
chain. Since $\En\brk*{\nu_i(x,z)}=1$ for all $i$ and
$\En\prn{\tfrac{z}{p}-1}=1$, it follows immediately from
\pref{eq:saa_grad} that
$\En\brk*{\grad\funscaled(x,z)}=\grad{}F(x)$. Now, let $x$ be fixed
and let $i>\prog(x)+1$. We claim that
$\brk*{\grad{}\funscaled(x,z)}_i=0$ with probability $1$. Since
$\abs*{x_{i-1}},\abs*{x_i}<1/4$, it follows from
\pref{eq:saa_g} that $g_i(x,z)=0$ and from \pref{eq:grad_theta} that
$\grad{}_i\Theta_j(x)=0$ for all $j$. This establishes that
$\brk*{\grad{}\funscaled(x,z)}_i=\grad{}_i\Funscaled(x)=0$ for all
$z\in\crl*{0,1}$. Now, consider the case $i=\prog(x)+1$. Here (since 
$\abs*{x_i}<1/4$) we still
have $\grad{}_i\Theta_j(x)=0$ for all $j$, so
$\grad_i\funscaled(x,z)=g_i(x,z)$. Since
$\Gamma(\abs*{x_{\geq{}i}})=\Gamma(\abs*{x_{\geq{}i+1}})=0$, we have
$\nu_{i}(x,z)=\nu_{i+1}(x,z)=\tfrac{z}{p}$, so
$g_i(x,z)=\grad_i\Funscaled(x)\cdot\tfrac{z}{p}$. It follows immediately
that $\P\prn{\exists{}x:\brk*{\grad\funscaled(x,z)}_{\prog(x)+1}\neq{}0}\leq{}p$.

To bound the variance and mean-squared smoothness of
$\grad\funscaled$, we begin by analyzing the sparsity pattern of the error
vector
\begin{equation*}
\delta(x,z)\ldef{}\grad\funscaled(x,z)-\grad\Funscaled(x,z).
\end{equation*}
Let
$i_x=\progh(x)+1$. Observe that if $j<i_x$, we have
$\nrm*{\Gamma(\abs*{x_{\geq{}j}})}\geq{}\Gamma(\abs*{x_{i_{x}-1}})\geq{}\Gamma(1/2)=1$,
and so $\Gamma'(1-\nrm*{\Gamma(\abs*{x_{\geq{}j}})})=0$ and consequently
$\grad_{i}\Theta_j(x)=0$ for all $i$. Note also that if $j>i_x$, we
have $H(x_{j-1},x_j)=0$. We conclude
that \pref{eq:saa_grad} simplifies to
\begin{equation}
\grad{}_i\funscaled(x,z) = g_i(x,z)
+\prn*{\frac{z}{p}-1}\cdot{}H(x_{i_x-1},x_{i_x})\cdot\grad{}_i\Theta_{i_x}(x).
\label{eq:saa_grad_simple}
\end{equation}
As in \pref{lem:pzc-pair},  we have $\nu_{i}(x,z)=1$ for all $i<i_x$
and $g_i(x,z)=\grad{}_i\Funscaled(x)=0$ for all $i>i_x$. Thus, using
the expression \pref{eq:saa_g} along with \pref{eq:saa_grad_simple},
we have
\begin{align}
  \label{eq:saa_delta}
  \delta_{i}(x,z) =\prn*{\tfrac{z}{p}-1}H(x_{i_x-1},x_{i_x})\cdot\grad{}_i\Theta_{i_x}(x)-\prn*{\tfrac{z}{p}-1}\left\{
  \begin{array}{ll}
    \htwo(x_{i_x-1},x_{i_x})\cdot{}\Theta_{i_x}(x),&\quad{}i=i_x-1,\\
     \hone(x_{i_x-1},x_{i_x})\cdot{}\Theta_{i_x}(x),&\quad{}i=i_x,\\
    0,&\quad{\text{otherwise.}}
  \end{array}
\right.
\end{align}
It follows immediately that the variance can be bounded as
\begin{align*}
\En\nrm*{\grad\funscaled(x,z)-\grad\Funscaled(z)}^{2}
&\leq{}
\frac{2}{p}H(x_{i_x-1},x_{i_x})^2\nrm*{\grad\Theta_{i_x}(x)}^{2}\\
&~~+ \frac{2}{p}h_1^2(x_{i_x-1},x_{i_x})\cdot{}\Theta_{i_x}(x)^{2}
+\frac{2}{p}h_2^2(x_{i_x-1},x_{i_x})\cdot{}\Theta_{i_x}(x)^{2}.
\end{align*}
From \pref{eq:grad_theta_norm} we have
$\nrm*{\grad{}\Theta_{i_x}(x)}\leq{}6^2$, and from
\pref{eq:psi_phi_bounds} we have $|H(x,y)|\leq{}12$, so the first
term contributes at most $\tfrac{2\cdot{}144\cdot{}6^{4}}{p}$. Since 
$|\Theta_i(x)|\leq{}1$, \pref{lem:deterministic-construction}
implies that the second and third term together contribute at most
$\tfrac{4\cdot{}23^{2}}{p}$. To conclude, we may take
\[
\En\nrm*{\grad\funscaled(x,z)-\grad\Funscaled(z)}^{2} \leq{} \frac{\vsigma^{2}}{p},
\]
where $\vsigma\leq{}10^{3}$. 

To bound the mean-squared smoothness $\E \norm{\grad\funscaled(x,z) - \grad\funscaled(y,z) }^2$,
we first use that $\E\brk*{\delta(x ,z)}=0$, which implies
\[
\E \norm{\grad\funscaled(x,z) -  \grad\funscaled (y,z) }^2 =\E 
\norm{\delta(x,z) - \delta(y,z) }^2 + 
    \norm{\Funscaled(x)-\Funscaled(y)}^2.
\]
We have $\norm{\grad \Funscaled(x) - \grad
    \Funscaled(y)} \le \lip{1} \norm{x-y}$ by \pref{lem:deterministic-construction}.\ref{item:lip}. For the
  other term, we use the sparsity pattern of
  $\delta(x,z)$ established in \pref{eq:saa_delta} along with the fact
  that $\En\prn*{\tfrac{z}{p}-1}^{2}\leq\tfrac{1}{p}$ to show
  \begin{align*}
\En\nrm*{\delta(x,z)-\delta(y,z)}^{2}    &\leq{} 
\frac{3}{p}\underbrace{\sum_{i\in\{i_x, i_y\}} 
\prn*{h_1(x_{i-1},x_{i})\cdot{}\Theta_{i}(x) 
-h_1(y_{i-1},y_{i})\cdot{}\Theta_{i}(y)}^{2}}_{\rdef{}\cE_1}\\
 &~~+  \frac{3}{p}\underbrace{\sum_{i\in\{i_x, i_y\}} 
 \prn*{h_2(x_{i-1},x_{i})\cdot{}\Theta_{i}(x) 
 -h_2(y_{i-1},y_{i})\cdot{}\Theta_{i}(y)}^{2}}_{\rdef{}\cE_2}\\
 &~~+ \frac{3}{p}\underbrace{\sum_{i=1}^{T} \prn*{H(x_{i_x-1},x_{i_x})\cdot
 		\grad{}_i\Theta_{i_x}(x)-H(y_{i_y-1},y_{i_y})
 		\cdot\grad{}_i\Theta_{i_y}(y)}^2}_{\rdef{}\cE_3},
  \end{align*}
  where $i_y=\prog[\half](y)+1$. %

We bound $\cE_1$ and $\cE_2$ using similar arguments to
\pref{lem:pzc-pair}. Focusing on $\cE_1$, and letting
$i\in\crl*{i_x,i_y}$ be fixed, we have
\begin{align*}
&\prn*{h_1(x_{i-1},x_{i})\cdot{}\Theta_{i}(x)
  -h_1(y_{i-1},y_{i})\cdot{}\Theta_{i}(y)}^{2}\\
&\leq
2\prn*{h_1(x_{i-1},x_{i}) -h_1(y_{i-1},y_{i})}^{2}\Theta_i(x)^{2}
+ 2\prn*{\Theta_{i}(x)  -\Theta_{i}(y)}^{2}h_1(y_{i-1},y_{i})^2.
\end{align*}
Note that by~\pref{lem:theta_properties}, (i) $\Theta_i$ is $6^2$ Lipschitz 
and $\Theta_i\leq{}1$ and
(ii) $h_1$ is $23$-Lipschitz and $\abs*{h_1}\leq{}5$ (from
\pref{obs:psi_phi_bounds} and 
\pref{lem:deterministic-construction}). Consequently,
\[
\cE_1\leq{}2\cdot{}10^{5}\cdot\nrm*{x-y}^{2}.
\]
Since $h_{2}$ is $23$-Lipschitz and has $\abs*{h_2}\leq{}20$, an
identical argument also yields that
\[
\cE_2\leq{}5\cdot{}10^{6}\cdot\nrm*{x-y}^{2}.
\]

To bound $\cE_3$, we use the earlier observation that for all $i$ and $j\ne 
i_x$ we have
$H(x_{j-1},x_j)\grad_i\Theta_j(x)=0$, and likewise that
$H(y_{j-1},y_j)\grad_i\Theta_j(y)=0$ for all $j\neq{}i_y$. This allows
us to write
\begin{align*}
\cE_3&=\sum_{i=1}^{T}
       \prn*{\sum_{j\in\crl*{i_x,i_y}}H(x_{j-1},x_{j})\cdot\grad{}_i\Theta_{j}(x)-H(y_{j-1},y_{j})\cdot\grad{}_i\Theta_{j}(y)}^2\\
&\leq{}2\sum_{j\in\crl*{i_x,i_y}}\sum_{i=1}^{T} \prn*{H(x_{j-1},x_{j})\cdot\grad{}_i\Theta_{j}(x)-H(y_{j-1},y_{j})\cdot\grad{}_i\Theta_{j}(y)}^2.
\end{align*}
Letting $j\in\crl*{i_x,i_y}$ be fixed, we upper bound the inner
summation as
\begin{align*}
  &\sum_{i=1}^{T}\prn*{H(x_{j-1},x_{j})\cdot\grad{}_i\Theta_{j}(x)-H(y_{j-1},y_{j})\cdot\grad{}_i\Theta_{j}(y)}^2\\
  &\leq
  2\sum_{i=1}^{T}\prn*{H(x_{j-1},x_{j})\cdot(\grad{}_i\Theta_{j}(x)-\grad{}_i\Theta_{j}(y))}^2+
    \prn*{(H(x_{j-1},x_{j})-H(y_{j-1},y_{j}))\cdot\grad{}_i\Theta_{j}(y)}^2\\
&= 2H(x_{j-1},x_j)^{2}\nrm*{\grad{}\Theta_j(x)-\grad\Theta_j(y)}^{2} + 2(H(x_{j-1},x_j)-H(y_{j-1},y_{j}))^{2}\nrm*{\grad\Theta_j(y)}^{2}.
\end{align*}
We may now upper bound this quantity by applying the following basic
results: 
\begin{enumerate}
\item $H(x_{j-1},x_j)\leq{}12$ by \pref{eq:psi_phi_bounds}.
\item $\abs*{H(x_{j-1},x_j)-H(y_{j-1},y_{j})}\leq{}20\nrm*{x-y}$, by \pref{eq:psi_phi_bounds}.
\item $\nrm*{\grad\Theta_j(y)}\leq{}6^{2}$ by
  \pref{lem:theta_properties}.\pref{item:theta_lipschitz}.
\item
  $\nrm*{\grad\Theta_j(x)-\grad\Theta_j(y)}\leq{}10^4\cdot\nrm*{x-y}$,
  by \pref{lem:theta_properties}.\pref{item:theta_smooth}.
\end{enumerate}
It follows that
$\cE_3\leq{}3\cdot{}10^{10}\cdot{}\nrm*{x-y}^{2}$. Collecting the
bounds on $\cE_1$, $\cE_2$, and $\cE_3$, 
this establishes that
\[
\En\norm{\grad\funscaled(x,z) - \grad\funscaled(y,z) }^2\leq{} \frac{\lipBar{1}^{2}}{p}\cdot\nrm*{x-y}^{2}.
\]
with $\lipBar{1} \leq{} \sqrt{10^{11} + \lip{1}^2}$.
\end{proof}

\subsection{Active oracles}\label{app:active}
\begin{proof}[\pfref{lem:prob-zero-chain-active}]
	Adopting the notation of  the proof of \pref{lem:prob-zero-chain} 
	(with $K=1$), we see that the equality $\P(\gamma\ind{t} - 
	\gamma\ind{t-1} \notin 
	\{0,1\}\mid \cG\ind{t-1})=0$ holds for our setting as well. Moreover, we 
	claim that
	\begin{equation}\label{eq:active-claim}
	\P(\gamma\ind{t} - \gamma\ind{t-1} =1\mid \cG\ind{t-1})
	\le 
	2p. %
	\end{equation}
	Given the bound~\pref{eq:active-claim}, the remainder of the proof is 
	identical to that of \pref{lem:prob-zero-chain}, with $2p$ replacing $p$.
	To see why~\pref{eq:active-claim} holds, let 
	$(x\ind{1},i\ind{1}),\ldots,(x\ind{t},i\ind{t})\in\cG\ind{t-1}$ 
	denote the sequence of queries made by the algorithm. 
	We first observe that, by the construction of $g_\pi$, we have 
	$\gamma\ind{t} = 1+ 
	\gamma\ind{t-1}$ only if 
	$\zeta_{1+\gamma\ind{t-1}}(\pi(i\ind{t}))=1$. Therefore,
	\begin{equation}\label{eq:active-first-obs}
	\P(\gamma\ind{t} - \gamma\ind{t-1} =1\mid \cG\ind{t-1}) 
	\le 
	\P\prn[\big]{\zeta_{1+\gamma\ind{t-1}}(\pi(i\ind{t}))=1\mid 
		\cG\ind{t-1}}.
	\end{equation}
	
	Next, let $b\in\{0,1\}^{N^T}$ denote a (random) vector whose $i$th entry 
	is $b_i \defeq \zeta_{1+\gamma\ind{t-1}}(\pi(i))$. The vector $b$ has 
	$N^{T-1}$ elements equal to 1 and its distribution is permutation 
	invariant. Note that, by construction, the vector $b$ is independent of 
	$\{\zeta_j(\pi(i))\}_{j\ne 1+\gamma\ind{t-1},i\in N^T}$. Consequently, 
	the gradient estimates $g\ind{1},\ldots,g\ind{t-1}$ depend on $b$ only 
	through their $(1+\gamma\ind{t-1})$th coordinate, which for iterate 
	$t'\le t-1$ is
	\begin{equation*}
	g_{1+\gamma\ind{t-1}}\ind{t'} = 
	\brk*{\grad_{1+\gamma\ind{t-1}}\Funscaled(x\ind{t'})} b_{i\ind{t'}}.
	\end{equation*}
	From this expression we see that $g\ind{t'}$ depends on $b$ only for 
	index queries in the set
	\begin{equation*}
	S\ind{t-1}\defeq \{i\ind{t'}\mid t'<t 
	~~\mbox{and}~~\grad_{1+\gamma\ind{t-1}} 
	\Funscaled(x\ind{t'}) \ne 0\}\in\cG\ind{t-1}.
	\end{equation*}
	Moreover, for every $i\in S\ind{t-1}$ we have that $b_i=0$, 
	because otherwise there exists $t'<t$ such that 
	$g_{1+\gamma\ind{t-1}}\ind{t'} \ne 0$ which gives 
	the contradiction $\gamma\ind{t-1}\ge \gamma\ind{t'}\ge 
	\prog[0](g\ind{t'})\ge 1+\gamma\ind{t-1}>\gamma\ind{t-1}$. In 
	conclusion, we have for every $i\in N^T$
	\begin{equation}\label{eq:active-second-obs}
	\P\prn[\big]{\zeta_{1+\gamma\ind{t-1}}(\pi(i))=1\mid 
		\cG\ind{t-1}}
	=
	\P\prn[\big]{b_i = 1 \mid b_{j}=0~\forall j\in S\ind{t-1}}
	= 
	\begin{cases}
	\frac{N^{T-1}}{N^T - |S\ind{t-1}|} & i \notin S\ind{t-1} \\
	0 & \text{otherwise,} \\
	\end{cases}
	\end{equation}
	where the last equality follows from the permutation invariance of $b$. 
	
	Combining the observations above with the fact that $|S\ind{t-1}| \le t-1 
	\le \frac{T}{4p} \le \frac{1}{4}NT \le \half N^T$ gives the desired 
	result~\pref{eq:active-claim}, since
	\begin{equation*}
	\P(\gamma\ind{t} - \gamma\ind{t-1} =1\mid \cG\ind{t-1}) 
	\stackrel{\text{\pref{eq:active-first-obs}}}{\le} 
	\P\prn[\big]{\zeta_{1+\gamma\ind{t-1}}(\pi(i\ind{t}))=1\mid 
		\cG\ind{t-1}}
	\stackrel{\pref{eq:active-second-obs}}{\le}  
	\frac{N^{T-1}}{N^{T}-t} \le \frac{2}{N} = 
	2p.
	\end{equation*} 

We remark that the argument above depends crucially on using a different bit for every 
coordinate. Indeed, had we instead used the original construction 
$\gunscaledBasic$ in Eq.~\pref{eq:basic-construction} and  
set $g_\pi(x;i)=\gunscaledBasic(\zeta_1(\pi(i)))$, an algorithm that queried 
roughly $N$ random indices would find an index $i^\star$ such that 
$\zeta_1(\pi(i^\star))=1$ and could then continue to query it exclusively, 
achieving a unit of progress at every query. This would decrease the lower 
bound from  $\Omega(T/p)=\Omega(NT)$ to $\Omega(N+T)$.
\end{proof}

\end{document}